\documentclass[12pt,a4paper]{amsart}

\usepackage[margin=2cm]{geometry}

\usepackage{amsmath,amssymb,amsthm,enumitem}

\usepackage[normalem]{ulem}

\usepackage{xcolor}
\usepackage{tikz}

\newtheorem{Auxiliary}{Auxiliary}[section]
\newtheorem{Thm}[Auxiliary]{Theorem}
\newtheorem{Cor}[Auxiliary]{Corollary}

\newtheorem{Lem}[Auxiliary]{Lemma}
\theoremstyle{definition}
\newtheorem{Rem}[Auxiliary]{Remark}
\newtheorem{Def}[Auxiliary]{Definition}
\newtheorem{example}[Auxiliary]{Example}
\newtheorem{Not}[Auxiliary]{Notation}

\newcommand{\Z}{\mathbb{Z}}
\newcommand{\Q}{\mathbb{Q}}
\newcommand{\R}{\mathbb{R}}
\DeclareMathOperator{\Supp}{Supp}

\newcommand{\graphscale}{1cm}

\newcommand\restr[2]{{
		\left.\kern-\nulldelimiterspace 
		#1 
		\right|_{#2} 
}}
\newcommand\restrd[2]{{
		\left.\kern-\nulldelimiterspace 
		#1 
		\vphantom{\big|} 
		\right|_{#2} 
}}

\usepackage{tikz}
\usepackage{pgfplots}
\pgfplotsset{compat=1.15}
\usepackage{mathrsfs}
\usetikzlibrary{arrows}
\usetikzlibrary{positioning}

\title[Splicing and Poincar\'e series]
{Splice formulae for Poincar\'e series of integral homology spheres}

\author[T. L\'aszl\'o]{Tam\'as L\'aszl\'o}
\address{Babe\c{s}-Bolyai University, Faculty of Mathematics and Computer Science,\newline \hspace*{6mm} 
Str. Mihail Kog\u{a}lniceanu nr. 1, 400084 Cluj-Napoca, Romania}
\email{tamas.laszlo@ubbcluj.ro}

\author[A. N\'emethi]{Andr\'as N\'emethi}
\address{Alfr\'ed R\'enyi Inst. of Math.,  
Re\'altanoda utca 13-15, H-1053, Budapest, Hungary \newline \hspace*{3mm}
Babe\c{s}-Bolyai Univ., 
str. M. Kog\u{a}lniceanu 1, 400084 Cluj-Napoca, Romania \newline \hspace*{3mm}
BCAM, 
Mazarredo, 14 E48009 Bilbao, Basque Country, Spain}
\email{nemethi.andras@renyi.hu}

\author[Gy. T\H ot\H os]{Gy\"orgy T\H ot\H os}
\address{Babe\c{s}-Bolyai University, Faculty of Mathematics and Computer Science,\newline \hspace*{6mm} 
Str. Mihail Kog\u{a}lniceanu nr. 1, 400084 Cluj-Napoca, Romania}
\email{gyorgy.totos@ubbcluj.ro}

\thanks{All the authors are partially supported by NKFIH Grant `\'Elvonal (Frontier)' KKP 144148.  T.L. is supported by the `J\'anos Bolyai Research Scholarship' of the Hungarian Academy of Sciences. Gy.T. and T.L. acknowledge  the support of the project `Singularities and Applications' - CF 132/31.07.2023 funded by the European Union - NextGenerationEU - through Romania's National Recovery and Resilience Plan.}

\begin{document}

\begin{abstract}
    Let $M$ be a plumbed integral homology sphere 3-manifold associated with a connected negative definite 
    plumbing graph $\Gamma$. One of its most important invariants is its multivariable Poincar\'e series
    (or zeta function) 
    $f_\Gamma(\mathbf {t})$. Several numerical invariants can be read from $f_\Gamma(\mathbf{t})$, or even from its 
    `polynomial part' $ \operatorname{Pol}_{\Gamma}(\mathbf{t})$. 
    For example, the `normalized' Casson's invariant equals $ \operatorname{Pol}_{\Gamma}(1)$. 

    In this note we provide several splice (surgery) formulae for $f_\Gamma$ and $\operatorname{Pol}_\Gamma$
     (reduced  to the node variables, or to the node variables of connected sub-graphs). In this way, these global invariants can be recovered 
      from a collection of certain smaller graphs. 

      Recall that some (integral homology sphere) 3-manifold  invariants are additive with 
      respect to the splice decomposition (like the Casson's invariant). However, some invariants 
      need some `splice correction terms'. In our  formulae the correction terms are easily computable one variable Alexander polynomials. 
\end{abstract}

\maketitle

\section{Introduction}

Let $M$ be an integral homology sphere 3-manifold associated with a connected negative definite plumbing graph $\Gamma$. The manifold $M$ might appear naturally also as an object of the low dimensional topology, but also as the link of a 
complex analytic normal surface singularity. In this second
case any dual resolution graph can serve as the plumbing graph of $M$. 

Several invariants of $M$ can be read from its lattice $L$. It is freely generated by the set of vertices $\mathcal{V}$ of $\Gamma$, that is, $L=\Z\langle E_v\rangle _{v\in \mathcal{V}}$.  It carries a negative definite intersection form $(\cdot\,,\cdot)$ read from $\Gamma$. 
In the case of a singularity, the lattice is given by the second integral homology of the total space of a resolution with its natural intersection form. This article focuses more on the topological (and combinatorial) properties of $M$ and $\Gamma$,  hence we will say less about singularity theoretical analytic properties of singular germs.

Since $M$ is an integral homology sphere, it is a plumbing of $S^1$-bundles over 2--spheres, 
the graph $\Gamma$ is a tree, and the intersection form is unimodular. In this case, $M$ can be coded by another graph as well, by the `splice diagram' $\mathfrak{S}$ \cite{eisenbud2016three}.
The nodes of $\Gamma$ and $\mathfrak{S}$ can be identified, while the strings (bamboos) in $\Gamma$ are 
replaced by edges in $\mathfrak{S}$. 
In fact, by the very conceptual and concise form of
the decorations of $\mathfrak{S}$, this diagram is  more convenient
in the discussion (and proofs) of splice decompositions and surgery formulae 
associated with integral homology sphere graph 3-manifolds. 
In our discussions and statements we will use  this language as well.

From the lattice $L$ one can read several key invariants. One of them, is the multivariable 
(topological) Poincar\'e series $f(\mathbf {t})$, or called also the `zeta function' associated with $L$.
Its definition will be given next. First, the fixed basis $\{E_v\}_v$ determines a dual basis 
$\{E^*_v\}_v$ in $L$,  defined as $( E_v^*, E_u)=-\delta_{uv}$ for every $v$ and $u$ (and $\delta_{uv} $ is the Kronecker delta). The negative sign in front of $\delta_{uv} $
is motivated by the following  fact -- as a consequence of the negative definiteness of 
the form --,  if we define $E_v^*$ via this sign convention then  all its 
$E_v$--coefficients  are strictly positive.

The lattice $L$ carries two partial orderings. If $l=\sum_{v}l_vE_v$ and $l'=\sum_{v}l'_vE_v$ 
then we say that $l\geq l'$ if and only if $l_v\geq l'_v$ for every $v$. Moreover, 
 we say that  $l\succ l'$ if 
    $l_v>l_v'$ for every $v$. 

Then, by definition, the Poincar\'e series $Z(\mathbf {t})=Z(t_1,\ldots, t_{|\mathcal{V}|})\in 
\Z[[t_1, \ldots, t_{|\mathcal{V}|}]]$ is  the
Taylor expansion at $\mathbf{t}=0$  of the zeta function 
 \begin{equation}\label{eq:zeta:defi}
            f(\mathbf{t}):= \prod_{v\in\mathcal{V}}(1-\mathbf{t}^{E_v^*})^{\delta_v-2},
        \end{equation}
        where for any  $l = \sum_v l_vE_v$ 
        we write $\mathbf{t}^{l} =\prod_v t_v^{l_v}$, and $\delta_v$ is the valency of the vertex $v$ in $\Gamma$.   
We usually identify $f$ with its expansion $Z$.  

The series contains very deep information about $M$, for several properties and applications see e.g.  \cite{CDGZ04,CDGZ08universal,laszlo2022canonical,LNN19,LNN20,LN14Erhart,LSz17,N04invariants, N07poincare,N08seiberg,NO09}. For its connection with the sign refined Turaev torsion of $M$ 
(and, hence with the Seiberg-Witten invariant of $M$)
in the rational homology sphere case see \cite{BN10,N02seiberg,N04seiberg,N06seiberg,N08seiberg,NO09}. 
For different surgery formulae which help its computations and the extraction of different information from it, see \cite{BN10,laszlo2022canonical,LNN19,LNN20}. For its connection with
Ehrhart theory see e.g. \cite{LN14Erhart,LSz17}.
For its asymptotic behaviour (e.g. certain summation of its coefficients), or its motivic generalizations  see different parts in \cite{nemethi2022normal}.

An important feature of  $f(\mathbf {t})$
 is that it has a canonical decomposition into its 
`polynomial part' and `negative degree part', and several invariants read from $f(\mathbf{t})$
are already determined from the polynomial part (though its asymptotic behaviour is coded by its 
negative degree part). 

Indeed, $f(\mathbf{t})$ has a decomposition  of type 
$\operatorname{Pol}(\mathbf{t})+f^-(\mathbf{t})$ with the following properties
(see \cite{LSz18,LNN19}):

\vspace{1mm}

(i) $\operatorname{Pol}(\mathbf{t})$  is a finite sum (`polynomial')  of type 
$\sum_{j}n_j \mathbf {t}^{c_j}$ with $c_j\not\prec 0$ for all $j$;

(ii) $f^-(\mathbf{t})$ is a rational function  with negative degree  in all variables $t_v$.

\noindent Furthermore, such a decomposition is unique.

\vspace{1mm}

The phenomenon can be exemplified  by the following elementary decomposition
$$\frac{t_1^3t_2}{1-t_1^2t_2^3}=-t_1t_2^{-2}+\frac{t_1t_2^{-2}}{1-t_1^2t_2^3}.$$
For a (different) algorithm for the decomposition of $f(\mathbf {t}) $  see section 4. 
For another concrete geometrical example of such decomposition see e.g. the last paragraph of \ref{sss:6.1.1}. 

As an application
for ${\rm Pol}(\mathbf{t})$, in section 6 we will use the identity
(cf.  \cite{A14casson,Ls96global})
        \begin{equation}\label{eq:intr1}
            \operatorname{Pol}(1)=  -\lambda(M)-\frac{Z_K^2+|\mathcal{V}|}{8},
        \end{equation}
        where $\lambda(M)$ is the Casson's invariant of $M$ and $Z_K\in L$ is the 
        anticanonical cycle computed by the adjunction relations in the lattice $L$ (see   \ref{ss:2.1}). 
        In this way, ${\rm Pol}(\mathbf{t})$ appears as a multivariable `polynomial' generalization 
        of the numerical invariant $\lambda+(Z_K^2+|\mathcal{V}|)/8$. 

In fact, almost all the information coded in $f(\mathbf {t})$ and its polynomial part ${\rm Pol}(\mathbf {t})$ are already coded in their restriction to the variables of the nodes $\mathcal{N}=\{v\in\mathcal{V}\,:\, \delta_v\geq 3\}$, that is, in the series 
$f(\mathbf{t}_{\mathcal{N}}):= f(\mathbf{t})|_{t_v=1\ \tiny{\mbox{for any $v\not\in \mathcal{N}$}}}$,
and ${\rm Pol}(\mathbf{t}_{\mathcal{N}})
:= {\rm Pol}(\mathbf{t})|_{t_v=1\ \tiny{\mbox{for any $v\not\in \mathcal{N}$}}}$. 

For example, in (\ref{eq:intr1}) the substitution $\mathbf {t}_{\mathcal{N}}=1$
in  ${\rm Pol}(\mathbf{t}_{\mathcal{N}})$ works equally well. 

\vspace{2mm}

In this article we present several surgery/splice  formulae for 
$f(\mathbf {t}_{\mathcal{N}})$ and its polynomial part 
${\rm Pol}(\mathbf {t}_{\mathcal{N}})$. 
(When we replace  $\Gamma$ by $\mathfrak{S}$ we keep all its nodes, hence $\mathcal{N}$ is essentially
the set of non-ends of $\mathfrak{S}$.)
The goal is to represent 
these functions 
associated with $\Gamma$ (or $\mathfrak{S}$)  
in terms of the corresponding functions associated with smaller graphs 
(with less nodes), after convenient reduction of variables and substitutions. 
In these formulae we do not have an exact  additivity of the contributing terms 
(as e.g. in the splice formula of the Casson's invariant \cite{BN,FM}),
however the correction terms are easily computable one-variable Alexander type polynomials. In this way the identities are high  generalizations (to the 
multivariable series level) e.g. of the splice formula of the numerical
invariant $Z_K^2+|\mathcal{V}|
$  from \cite{NW}. 

 Note that  smaller graphs 
 produce series in smaller number of variables (and the information they code is more local than global),  hence 
 usually $f(\mathbf {t}_{\mathcal{N}})$ in $|\mathcal{N}|$ variables, cannot be recovered  from a/any 
 family of sub-graphs if they are too small. 
 Hence, one has to consider smaller (but not very small)  sub-graphs which still have $|\mathcal{N}|$ `basic', well-chosen vertices, which produce series in  $|\mathcal{N}|$ variables, and which (after certain change of variables) reconstruct $f(\mathbf{t}_{\mathcal{N}})$.

 Such possible sub-diagrams are the {\it peelings} of $\mathfrak{S}$, cf. section 3. In their case the subgraphs are obtained by deleting certain `end-nodes' of $\mathfrak{S}$, and the variables of the deleted nodes 
 are replaced by `gluing vertices'. Even for such peeled sub-diagrams, the fact that from their
`semi-local' information one can recover the global $f(\mathbf{t}_{\mathcal{N}})$ is rather surprising. 

The present surgery formulae are formally and conceptually different than those proved in \cite{BN10,LNN20,O08}
(or formulae based on several  exact triangles of low dimensional topology), where one compares invariants 
of the graphs $\Gamma$ and $\Gamma\setminus \{v\}$, where $v$ is a vertex of $\Gamma$
(or modification of the Euler numbers of $\Gamma$).  

\vspace {2mm}

The structure of the paper is the following. 

\vspace {2mm}

Section  1 contains some preliminaries about plumbing graphs, splice diagrams, the Poincar\'e series, 
and its connection with Alexander polynomials of links.
Section  3 introduces the notion of peeling of a splice diagram  and presents a surgery formula for $f(\mathbf {t}_{\mathcal{N}})$ in terms of series of all the 
peelings of $\mathfrak{S}$. Section  4 reduces this surgery formula 
to the polynomial part. In section  5 we present an even more general 
surgery formula when we recover by surgery $f(\mathbf{t})$  and ${\rm Pol}(\mathbf {t})$
reduced to the 
variables of the nodes of an arbitrary connected sub-diagram of $\mathfrak{S}$.
In section  6 we present some applications regarding the substitution 
${\rm Pol}(\mathbf{t}=1)$. 
(This numerical invariant ${\rm Pol}(1)$
is called in the literature the `periodic constant' of $f(\mathbf{t})$.)
This substitution will provide new surgery formulae for
the numerical invariant $Z_K^2+|\mathcal{V}|$ (using the previous surgery formulae for ${\rm Pol}$, the identity 
(\ref{eq:intr1}), and the splice formula for the Casson's invariant \cite{BN,FM}).

    \section{Preliminaries}
    \subsection{Resolution/plumbing graphs}   \label{ss:2.1}
        Let $(X,0)$ be a normal surface singularity with a {\it  rational homology sphere link} $M$ and $\pi\colon \widetilde{X}\to X$ a good resolution, that is, the exceptional divisor $E = \pi^{-1}(0)$ is a normal crossing divisor on the smooth surface $\widetilde{X}$. Note that
        $\partial \widetilde{X}\simeq M$. Assume that 
        the irreducible components of $E$ are $\{E_v\}_{v\in\mathcal{V}}$, and let $\Gamma$ be the dual resolution graph of $\pi$. 
        
        Note that $\Gamma$ is a connected plumbing graph with a negative definite intersection form $\mathcal{I} = ((E_u,E_v))_{u,v}$. We denote the self-intersection numbers by $-b_v:=(E_v,E_v)$ for all $v\in\mathcal{V}$ and we use the notation $\det\Gamma :=\det(-\mathcal{I})$ for the determinant of the graph $\Gamma$. In the sequel, when we refer to the vertices of $\Gamma$,  we might use both the indices from $\mathcal{V}$ and the corresponding irreducible components $E_v$. For the pair of vertices $u,v\in\mathcal{V}$ we will write $u-v$ if they are connected by an edge. Let us denote by $\delta_v$ the valency of the vertex $v$ in $\Gamma$. Denote by $\mathcal{N}$ the set of nodes (vertices with $\delta_v\geq 3$) and by $\mathcal{E}$ the set of ends (vertices with $\delta_v=1$). Furthermore, vertices with $\delta_v=2$ will be called bamboo vertices.
        \vspace{0.2cm}
        
        The assumption that the link $M$ is a rational homology sphere means that $\Gamma$ is a tree and 
        the genus $g(E_v)$ of any $E_v$ is zero.
         Moreover, $M$ is an \emph{integral homology sphere} if additionally $\det\Gamma=1$. 
        \vspace{0.2cm}
        
        We associate with $\Gamma$ a  lattice and its dual as follows. Define the lattice  $L(\Gamma) = H_2(\widetilde{X},\Z)\cong \Z\langle E_v\rangle_{v\in\mathcal{V}}$. It is  endowed with the negative definite intersection form $(\cdot\,,\cdot)$. Its elements will be called \emph{(integral) cycles}. The dual lattice $L'(\Gamma) := \operatorname{Hom}(L,\Z)\simeq H_2(\widetilde{X},M,\mathbb{Z})\simeq H^2(\widetilde{X},\mathbb{Z})$ embeds into $L(\Gamma)\otimes \Q$ and can be identified with the rational cycles $l'\in L(\Gamma)\otimes \Q$ for which $(l',l)_{\mathbb{Q}}\in \Z$ for any $l\in L(\Gamma)$. Here  $(\cdot,\cdot)_{\mathbb{Q}}$ is the extension of $(\cdot,\cdot)$ to $L(\Gamma)\otimes \Q$,
        in the sequel still denoted by $(\cdot,\cdot)$. Let $\{E_v^*\}_{v\in\mathcal{V}}$ be the (anti)dual basis which generates $L'(\Gamma)$, chosen in such a way that $(E_v^*,E_u) = -\delta_{u,v}$ (Kronecker delta) for every $u,v\in \mathcal{V}$. In fact, expressed in the $\{E_v\}_v$ basis,  the cycles $\{E_v^*\}_v$ are the columns of $(-\mathcal{I})^{-1}$.  For any cycle $l'\in L'(\Gamma)$ and $v\in\mathcal{V}$ denote by $m_v(l') = (-E_v^*,l')$ and $a_v(l') = (-E_v,l')$ the corresponding coefficients. That is
        \begin{equation}\label{eq:coords}
            l' = \sum_vm_v(l')E_v = \sum_v a_v(l')E_v^*.
        \end{equation}
        The quotient group $L'(\Gamma)/L(\Gamma)$ is the finite group $H_1(M,\Z)$ of order $\det\Gamma$. 
        \vspace{0.2cm}
        
         We define the anti-canonical cycle $Z_K\in L'(\Gamma)$  as the unique cycle of $L'(\Gamma)$ such that $-Z_K$ is numerically equivalent with the canonical divisor of the smooth surface $\widetilde{X}$. Thus, by the adjunction relations, it is determined topologically by the linear system of equations 
         $\{(Z_K,E_v)=-b_v+2\}_v$. This system reads also as  $Z_K = E+\sum_v(\delta_v-2)E_v^*$.
        \vspace{0.2cm}

             If we consider several plumbing graphs, in our notations we might indicate which graph the corresponding objects are associated with. Eg., if needed, we will denote by $E_{v,\Gamma}$ the irreducible exceptional divisors corresponding to $\Gamma$.

        More details regarding this part can be found in most of our references, see eg.
        \cite{N04invariants, nemethi2022normal}.
        
    \subsection{Splice diagrams} Assume that $M$ is an integral homology sphere, that is  $\det \Gamma = 1$, or, equivalently  $L'(\Gamma) = L(\Gamma)$. Then 
    with the  resolution (or negative definite plumbing) tree one can associate another weighted tree, called the `splice diagram'. The splice diagrams were introduced by Siebenmann, and extensively studied by Eisenbud and Neumann \cite{eisenbud2016three}. If the singularity link is an integral homology sphere, the  resolution graph can be codified completely in a more compact form by this diagram. 
    Since the technical details of this article are largely based on the arithmetic of the splice diagram, we will set some notations and recall in more detail the transition between the resolution graphs and the splice diagrams.
        \subsubsection{\bf From plumbing tree to splice diagram \cite{eisenbud2016three} } \label{ss:splice}
        Let $\Gamma$ be a connected negative definite plumbing tree with $\det \Gamma=1$. We say that two vertices $u,v\in \mathcal{V}$ are connected by a bamboo, if the shortest path from $u$ to $v$ contains only bamboo vertices, or $u-v$ (ie. they are connected by an edge). In the sequel, this property  will be denoted by $u-_bv$. Furthermore, for any two vertices $u,v\in\mathcal{V}$, $u\not=v$,  we will write $d_{v,u} := \det\Gamma_{v,u}$ for the determinant of the subgraph $\Gamma_{v,u}$, which is the connected component of $\Gamma\setminus\{v\}$ containing $u$.

        Roughly speaking, a splice diagram is a `finite tree' shaped diagram with vertices only of valency $\geq 3$ (call them `nodes')  and valency one (call them `ends'), and also 
        with a collection of integer weights at each node, associated with the edges departing from that node. 
        
        Associated with a plumbing/resolution graph $\Gamma$,  the splice diagram $\mathfrak{S}$ can be constructed in the following way. The nodes of the splice diagram are exactly the nodes of $\Gamma$. Two nodes $u,v$ are connected by an edge in the splice diagram if $u-_b v$ in $\Gamma$, and for simplicity, we will refer to this edge in $\mathfrak{S}$ by $u\sim v$. Additionally, for every end-vertex $e\in\mathcal{E}$ of $\Gamma$ we take an end-vertex for $\mathfrak{S}$: if 
       $n\in \mathcal{N}$ is the  unique node in $\Gamma$ such that $n-_b e$, we associate an edge in the splice diagram which connects the corresponding end-vertex and node in $\mathfrak{S}$.
       Similarly as above, we will denote this fact  by $n\sim e$;  these edges are also called the \emph{legs} of the splice diagram. An edge of type $u\sim v$ ($u,v\in \mathcal{N}$) is decorated at $v$ with $d_{v,u}$, while an edge $n\sim e$ is decorated only at $n$ with the corresponding determinant $d_{n,e}$.

        By \cite{eisenbud2016three}, the integral homology spheres that are singularity links (that is, the intersection matrix of their plumbing tree  is negative definite and unimodular) 
        are in one-to-one correspondence with splice diagrams satisfying the following properties:
        \begin{itemize}
        \item[(a)] the weights around a node are positive and pairwise relatively prime; 
        \item[(b)] the weight on an edge $n\sim e$ is $>1$;
        \item[(c)] for any edge $u\sim v$ ($u,v\in \mathcal{N}$) the \emph{edge determinant}, which is defined by the product of the two weights on the edge minus the product of the weights adjacent to this edge, is positive. That is, 
        \[
             d_{u,v}d_{v,u} - \prod_{w\sim u}^{w\neq v}d_{u,w}\prod_{w \sim v}^{w\neq u}d_{v,w}>0.
        \]
        \end{itemize}
        Such splice diagrams are called `minimal'. A non-minimal splice diagram satisfies (a) and (c), but not necessarily (b). If a diagram has legs decorated by 1 then it can be transformed into a minimal diagram by removing these  legs (and the diagrams will represent the same 3-manifold).
        
        Now, for any two vertices $u,v\in\mathcal{N}\cup\mathcal{E}$ in $\mathfrak{S}$ we define $\mathfrak{S}[u,v]\subseteq\mathcal{N}\cup\mathcal{E}$ to be the  set of vertices in $\mathfrak{S}$
        on the shortest  path from $u$ to $v$ in $\mathfrak{S}$, including $u$ and $v$ as well. We define $w(\mathfrak{S}[u,v])$ as the product of the weights of the splice diagram {\it 
        along but not on }the path from $u$ to $v$. We also define $w(\mathfrak{S}[u,v))$ to be the same product of weights, with the exception of those surrounding $v$. In other words:
        \begin{align*}
            w(\mathfrak{S}[u,v]) &:= \prod\Big\{d_{\nu,\nu'}\big| \nu\sim\nu',\nu\in\mathfrak{S}[u,v],\nu'\notin\mathfrak{S}[u,v]\Big\},\\
            w(\mathfrak{S}[u,v)) &:= \prod\Big\{d_{\nu,\nu'}\big| \nu\sim\nu',\nu\in\mathfrak{S}[u,v]\setminus\{v\},\nu'\notin\mathfrak{S}[u,v]\Big\}.
        \end{align*}
        Similarly, we define $w(\mathfrak{S}(u,v])  := w(\mathfrak{S}[v,u))$.
        
        This setup has the following advantage regarding the determination of the multiplicities of the cycles. The coefficient $m_{u}(E_v^*)$ for any $u,v\in\mathcal{N}\cup\mathcal{E}$ ---
        and for $v\neq u$ if $u,v\in\mathcal{E}$ --- equals
        to $w(\mathfrak{S}[u,v])$ (see Lemma 20.2 in \cite{eisenbud2016three}), which can be read off easily from the splice diagram.
        
        In parallel,  for any $u,v\in\mathcal{V}$ in $\Gamma$, we define the sub-graph $\Gamma[u,v]$ of $\Gamma$ to be the shortest  path (full subgraph) from $u$ to $v$. Then for any $u,v\in\mathcal{V}$ we have $m_u(E_v^*) = \det(\Gamma\setminus\Gamma[u,v])$ (see \cite[section 2.2.5]{nemethi2022normal}). We also define $\Gamma[u,v) := \Gamma[u,v]\setminus\{v\}$, and $\Gamma(u,v) := \Gamma[u,v]\setminus\{v,u\}$.

        \subsubsection{\bf From splice diagram to plumbing tree}\label{ss:splicetoplumb}
        For the construction of the plumbing tree from a splice diagram we refer to \cite[sect. 22]{eisenbud2016three}), or see eg. \cite[sect. 5.2]{nemethi2022normal}). 
        
        For each node $v$ of the (not necessarily minimal) splice diagram one considers  the  star-shaped sub-diagram (with unique node $v$ and all its adjacent edges and their decorations). It  represents the integral homology sphere $\Sigma_v:=\Sigma(\left\{d_{v,u}\mid u\sim v\right\})$.
        In the next paragraph we associate a star-shaped plumbing graph to it. 
        
        Choose one  $u\sim v$ and take the negative (or Hirzebruch) continued fraction of
        \[
            \frac{1}{d_{v,u}}w(\mathfrak{S}[v,u))=\frac{1}{d_{v,u}}\cdot \prod_{w\sim v}^{w\neq u} d_{v,w} = b_0-\cfrac{1}{b_1 - \cfrac{1}{\ddots-\cfrac{1}{b_s}}},
        \]
        where $b_0\ge1$ and $b_1,\ldots,b_s\ge2$. Now, we attach a bamboo with decorations $-b_s,\ldots,-b_1$ to the vertex $E_v$ (see below); and, if $u$ is a node in the splice diagram, an additional star $\star$ vertex with decoration $-b_0$ is also attached to the end. 
        \begin{center}
            \begin{tikzpicture}[roundnode/.style={circle, draw=black, fill=black, very thick, inner sep = 1pt},line width=1]
                \node[roundnode] (v) {};                        \path (v)++(90:0.5) node {$E_v$};
                \path (v) ++ (1,0) node[roundnode] (s) {};      \path (s)++(90:0.5) node {$-b_s$};
                \path (v) ++ (2,0) node (left) {};
                \path (v) ++ (3,0) node (right) {};
                \path (v) ++ (4,0) node[roundnode] (1) {};      \path (1)++(90:0.5) node {$-b_1$};
                \path (v) ++ (5,0) node (0) {};                 \path (0) node {$\star$}; \path (0)++(90:0.5) node {$-b_0$};
                \draw (v)--(left) (right)--(0);
                \draw[dashed] (left)--(right);
                
            \end{tikzpicture}
        \end{center}
        For $d_{v,u}=1$,  we attach a vertex with decoration $-1$ to $E_v$ if $v\sim u$ is a leg; otherwise a star vertex  with decoration $\big(-w(\mathfrak{S}[v,u))\big)$.
        Repeat this for all $u\sim v$ and decorate $E_v$ with the unique $-b_v$ in such a way that the corresponding determinant of the constructed star-shaped graph (without the star vertices) equals  $1$.
(This decorated graph, without the star vertices, is the plumbing graph of $\Sigma_v$. The star vertices serve as `gluings' for the plumbing graph of $\Gamma$.)

        Next, the plumbing graph $\Gamma$, which corresponds to the whole splice diagram, is obtained from the plumbing graphs (including the star vertices)
        associated with the pieces $(\Sigma_v)_{v\in\mathcal{N}}$ as follows. For every $u\sim v$, $u,v\in \mathcal{N}$  we connect the corresponding two $\star$ vertices 
        (one of them is an end-vertex in the plumbing graph associated with $\Sigma_v$, the other one in 
        $\Sigma_u$) with an edge. The plumbing graph obtained in this way is usually not negative definite, and the bamboos are not minimal either. After a number of blow downs of $(-1)$ vertices
        and a single $0$-absorption on each bamboo between nodes, the final plumbing tree is obtained. 
        (For the plumbing calculus see \cite{NeuCalc}.)
        Note that the nodes of the splice diagram survive after the plumbing calculus as nodes in the plumbing graph.
        
        Since it will be used later, for any two nodes $u,v\in \mathcal{N}$ in $\Gamma$
        with $u-_b v$ we denote by $\theta_{u,v}$ the vertex of the plumbing graph obtained from the $0$-absorption, as a vertex of $\Gamma$. Note that, in general $\theta_{u,v}$ could be any of the vertices of the bamboo from $u$ to $v$, including the nodes.
        \begin{Rem}
            When referring to the plumbing graph associated to some splice diagram $\mathfrak{S}$, we mean the negative definite tree obtained from the algorithm described above, and we denote it by $\Gamma(\mathfrak{S})$. To be more precise, other than the necessary blow-downs and the $0$-absorptions we do not perform any more blow-downs.
        \end{Rem}
        \begin{example}
            Let us consider the splice diagram on left hand side of Figure \ref{fig:splice_0_absorption_on_node}. In the same figure,  in the middle, there 
            are the plumbing graphs of the integral homology spheres $\Sigma(2,3,11)$ and $\Sigma(4,5,11)$ respectively (undecorated vertices have weight $-2$), with the $\star$ vertices connected by an edge. After 7 blow-downs of bamboo vertices non-adjacent to the nodes, the vertex with weight 0 is neighbouring the left node. As such, it will be absorbed into it, and the final self-intersection number of this node becomes $-3$.  In the final plumbing tree there are no vertices between the two nodes, they are connected by an edge, cf. the right hand  
            side of figure \ref{fig:splice_0_absorption_on_node}.
            \begin{figure}[!h]
                \centering
                \begin{tikzpicture}[roundnode/.style={circle, draw=black, fill=black, very thick, inner sep = 1},line width=1,x=\graphscale,y=\graphscale]
                	\node[roundnode] (1) {};
                	\path (1)++(2,0) node[roundnode] (3) {};
                	\path (1)++(120:1.5) node[roundnode] (10) {};
                	\path (1)++(-120:1.5) node[roundnode] (11) {};
                	\path (3)++(60:1.5) node[roundnode] (30) {};
                	\path (3)++(-60:1.5) node[roundnode] (31) {};
                	
                	\draw (10)--(1)--(11) (1)--(3) (30)--(3)--(31);
                	\begin{small}
                		\renewcommand{\r}{0.5}	
                		\path (1)++(100:\r) node {$3$};	
                		\path (1)++(-100:\r) node {$2$};
                        \path (1)++(20:\r) node {$11$};
                		\path (3)++(80:\r) node {$4$};
                		\path (3)++(-80:\r) node {$5$};
                		\path (3)++(160:\r) node {$11$};	
                	\end{small}
                \end{tikzpicture}
                \hspace{0.25cm}
                \begin{tikzpicture}[roundnode/.style={circle, draw=black, fill=black, very thick, inner sep = 1},line width=1,x=\graphscale,y=\graphscale]
                	\node[roundnode] (0) {};
                    \path (0) ++ (120:0.5) node[roundnode] {};
                    \path (0) ++ (120:1) node[roundnode] {};
                    \path (0) ++ (-120:0.5) node[roundnode] {};
                    \foreach \i in {1,2,3,4} {
                        \path (0)++({\i*0.5},0) node[roundnode] {};
                    }
                    \path (0)++(2.5,0) node[roundnode] (01) {};
                    \draw (0)-++(120:1) (0)-++(-120:0.5) (0)-++(3,0);

                    \path (0)++(3.25,0) node (1) {$\star$};
                    \path (0)++(5,0) node (2) {$\star$};

                    \path (2)++(0.75,0) node[roundnode] (3){};
                    \path (3)++(0.5,0) node[roundnode] {};
                    \path (3)++(1,0) node[roundnode] (4) {};
                    \path (4)++(60:1) node[roundnode] (5) {};
                    \path (4)++(-60:1) node[roundnode] (6) {};

                    \draw (1) --(2);

                    \draw (2)--(4) (5)--(4)--(6);
                    
                	\begin{small}
                		\renewcommand{\r}{0.5}		
                		\path (01)++(90:\r) node {$-3$};
                        \path (0)++(3.25,\r) node {$-1$};
                        \path (2)++(0,\r) node {$-2$};
                        \path (3)++(0,\r) node {$-6$};
                        \path (4) ++(0:\r) node {$-1$};
                        \path (5)++(0:\r) node {$-4$};
                        \path (6)++(0:\r) node {$-5$};
                	\end{small}  
                \end{tikzpicture}
                \hspace{0.25cm}
                \begin{tikzpicture}[roundnode/.style={circle, draw=black, fill=black, very thick, inner sep = 1},line width=1,x=\graphscale,y=\graphscale]
                	\node[roundnode] (0) {};
                    \path (0) ++ (120:0.5) node[roundnode] {};
                    \path (0) ++ (120:1) node[roundnode] {};
                    \path (0) ++ (-120:0.5) node[roundnode] {};
                    
                    \draw (0)-++(120:1) (0)-++(-120:0.5);

                    \path (0)++(1,0) node[roundnode] (4){};;
                    \path (4)++(60:1) node[roundnode] (5) {};
                    \path (4)++(-60:1) node[roundnode] (6) {};

                    \draw (0)--(4) (5)--(4)--(6);
                    
                	\begin{small}
                		\renewcommand{\r}{0.5}		
                        \path (0)++(180:\r) node {$-3$};
                        \path (4) ++(0:\r) node {$-1$};
                        \path (5)++(0:\r) node {$-4$};
                        \path (6)++(0:\r) node {$-5$};
                	    \path (4)++(-115:\r) node {$E_2$};
                        \path (0)++(-65:\r) node {$E_1$};
                    \end{small}

                \end{tikzpicture}
                \caption{}
                \label{fig:splice_0_absorption_on_node}
            \end{figure}
        \end{example}

    \subsection{The topological Poincar\'e series \cite{CDGZ04,CDGZ08universal,N07poincare}}\label{ss:2.3}
        Let $\Gamma$ be a connected negative definite plumbing graph as above with $\det\Gamma=1$.
        The \emph{(topological) Poincar\'e  series} $Z(\mathbf{t}):=\sum_{l\in L(\Gamma)}z(l)\mathbf{t}^{l}$ associated with $\Gamma$ is the Taylor expansion at $\mathbf{t}=0$ of the  multi-variable rational function
        \begin{equation}\label{eq:zeta:def}
            f(\mathbf{t}):= \prod_{v\in\mathcal{V}}(1-\mathbf{t}^{E_v^*})^{\delta_v-2},
        \end{equation}
        where for any  $l = \sum_v l_vE_v$ 
        we write $\mathbf{t}^{l} =\prod_v t_v^{l_v}$.  
        The non-zero coefficients correspond to exponents of the form
        $
        l = \sum_{e\in\mathcal{E}}a_eE_e^*+ \sum_{n\in\mathcal{N}}a_nE_n^*,
        $
        with integers $a_e\ge0$ and $0\le a_n\le \delta_n-2$. Moreover, for any such $l$, 
        $
        z(l) = \prod_{n\in\mathcal{N}}(-1)^{a_n}\binom{\delta_n-2}{a_n}
        $.    
        Note that the series is supported on the so-called \emph{Lipman cone} $\mathcal{S}_\Gamma:=\Z_{\ge0}\langle E_v^* \rangle_{v\in \mathcal{V}}$. 
    \subsubsection{\bf Reduced series}\label{sss:redseries}
        Let $I\subseteq \mathcal{V}$ be a non-empty subset and consider the natural projection map $\Z\langle E_v\mid v\in\mathcal{V}\rangle\to\Z\langle E_v\mid v\in I\rangle$, $l=\sum_vm_vE_v\mapsto \restr{l}{I}=\sum_{v\in I}m_vE_v$. The fibers of this map intersected with the Lipman cone have finitely many elements, thus the substitution $t_v\mapsto1$ for $v\in\mathcal{V}\setminus I$ is well defined.
        The output $Z(\mathbf{t})_{t_v\mapsto1,v\notin I}$
        is called the reduced series, and it is denoted by $Z(\mathbf{t}_I)$. Note that we can  reduce the  function $f$ as well by taking $f(\mathbf{t}_I)= f(\mathbf{t})_{t_v\mapsto1,v\notin I}$, and the Taylor expansion of $f(\mathbf{t}_I)$ will be $Z(\mathbf{t}_I)$. For more see  eg. 
        \cite{N07poincare, LNN19,LNN20}.

  \subsection{\bf The multivariable Alexander polynomial associated with graph links}
  The topological Poincar\'e series (and several other series and polynomials considered in this note) can be naturally linked with the multivariable Alexander polynomials of graph links. 
        For the definition of the Alexander polynomial see eg. 
        \cite[sect. 5]{eisenbud2016three}.

        Let $\Gamma$ be a negative definite plumbing tree as before whose associated plumbed 3--manifold $M$ is an integral homology sphere. Consider the splice diagram associated with $\Gamma$ and we put some arrows on the legs, indexed by $I\subset\mathcal{E}$. 
        This represents a graph link $\mathcal{L}_I \subset M$. 
        An arrowhead supported on $v\in I$ represents a generic $S^1$-fiber above $E_v$ in the plumbing construction.
        Then, the Alexander polynomial $\Delta_{\mathcal{L}_I}(\mathbf{t}_I)\in\Z[t_i\mid i\in I]$ of this graph link $\mathcal{L}_I$ is given by (see \cite[Theorem 12.1]{eisenbud2016three}) 
        \begin{align}\label{eq:alexander}
        \Delta_{\mathcal{L}_I}(\mathbf{t}_I) = \begin{cases}
            \prod\limits_{v\in K_I}(1-\mathbf{t}_I^{E_v^*})^{\delta_v-2}& \mbox{if} \ |I|>1,\\
            (1-\mathbf{t}_I)\prod\limits_{v\in K_I}(1-\mathbf{t}_I^{E_v^*})^{\delta_v-2} & \mbox{if} \ |I|=1,
        \end{cases}            
        \end{align}
        where $K_I:= (\mathcal{N}\cup\mathcal{E})\setminus I$ and $\mathbf{t}_I^l = \prod_{i\in I}t_i^{m_i(l)}$ for any $l\in L(\Gamma)$. 

        Note that in the above definition we mean the `normalized Alexander polynomial', that is, it satisfies $\Delta_{\mathcal{L}_I}(0)=1$. (In \cite{eisenbud2016three} the signs in the definition are reversed, that is, instead of the factors $1-\mathbf{t}_I^l$ the terms $\mathbf{t}_I^l-1$ are considered.)     
    
\subsection{Further notations}\label{ss:NOT}

     Fix a splice diagram $\mathfrak{S}$ as in  \ref{ss:splice}.

  For any node $n\in\mathcal{N}$ we denote by $\nu_n$ the node valency, that is, the number of nodes connected to $n$ by an edge in $\mathfrak{S}$. Denote by
            $\mathcal{N}_{\mathrm{e}}$ the set of nodes with $\nu_n = 1$, and set $\mathcal{N}_{\mathrm{n}}:=\mathcal{N}\setminus \mathcal{N}_{\mathrm{e}}$. The vertices from  $\mathcal{N}_{\mathrm{e}}$ are called \emph{end-nodes}. For an end-node $n$, denote by $v_n$ the node which is connected  to it by an edge of $\mathfrak{S}$.
            
 For every node $n$ we set $\omega_n = \prod_{e} d_{n,e}$, where $e$ runs over all end-vertices which are connected to $n$ by an edge. Moreover, for every $e\in \mathcal{E}$ we introduce $\omega_{e} :=\omega_n/d_{n,e}$, where $n$ is the unique node connected to $e$ by an edge.
       
         \subsection{\bf The polynomial  functions $\Delta_n(t)$ and $P_n(t)$ }  \label{ss:normalizedAlexander}   
         Fix $\mathfrak{S}$ as above. 
        For every node $n\in\mathcal{N}$ we consider the star-shaped sub splice diagram 
        with central node $n$ and with all edges (connecting nodes) and legs adjacent to it.  On all new legs that were edges in $\mathfrak{S}$ (connecting to other nodes of $\mathfrak{S}$)  we place an arrow, and we denote their set by $V_n$ (see Figure \ref{fig:splice1}). The graph link 
        (sitting in $\Sigma_n$) 
        corresponding to this splice diagram will be denoted by  $\mathcal{L}_n$. Then, by (\ref{eq:alexander}), its Alexander polynomial  $\Delta_{\mathcal{L}_n}(\mathbf{t}_{V_n}) $
        can be expressed as $\Delta_n(T_n)$, where $\Delta_n$ is a polynomial in one-variable. Indeed, if we set  
        \begin{equation}\label{eq:AlexanderSingleNode}
            \Delta_{n}(t) := \frac{(1-t^{\omega_n})^{\delta_n-2}}{\prod_{n\sim e}^{e\in\mathcal{E}}(1-t^{\omega_{e}})}\cdot\begin{cases}
                1& \mbox{if} \ \ n\in\mathcal{N}_{\mathrm{n}},\\
                (1-t) &\mbox{if} \ \ n\in\mathcal{N}_{\mathrm{e}},
            \end{cases}
        \end{equation}
        and $T_n := \prod_{v\in V_n}t_v^{\alpha_{nv}}$ where  $\alpha_{nv} = \prod\{d_{n,u}\mid u\in V_n,u\neq v\}$, 
        then 
         $\Delta_{\mathcal{L}_n}(\mathbf{t}_{V_n}) =\Delta_n(T_n)$.

\vspace{1mm}

       Next, it is convenient to modify slightly the polynomials $\Delta_n(t)$
       as follows.

        For the nodes $n\in\mathcal{N}_{\mathrm{n}}$, we set 
        \begin{equation}\label{eq:Pn_NodeNode}
         P_n(t):=   \Delta_n(t) =\frac{(1-t^{\omega_n})^{\delta_n-2}}{\prod_{n\sim e}^{e\in\mathcal{E}}(1-t^{\omega_{e}})}\in\Z[t],
        \end{equation} 
        while for the nodes $n\in\mathcal{N}_{\mathrm{e}}$ one defines 
        \begin{equation}\label{eq:Pn_EndNode}
            P_n(t) = \frac{1-\Delta_n(t)}{1-t}=\frac{1}{1-t}-\frac{(1-t^{\omega_n})^{\delta_n-2}}{\prod_{n\sim e}^{e\in\mathcal{E}}(1-t^{\omega_{e}})}. 
        \end{equation}
        Note that if $n\in \mathcal{N}_{\mathrm{n}}$ then  $P_n(t)=   \Delta_n(t)$ is automatically a polynomial. 
        If $n\in \mathcal{N}_{\mathrm{e}}$ then $\mathcal{L}_n$ has a single component, hence  $\Delta_n(1)=1$.  In particular, $P_n(t)$ is  a polynomial
        in this case as well.  In such cases, when there is a single component, the rational function
        $\Delta_n(t)/(1-t)$ has a pole of order 1 at $t=1$, and its decomposition into
         $1/(1-t)-P_n(t)$ splits (naturally and in a unique form) into a rational function with negative degree $1/(1-t)$ and a polynomial $P_n(t)$.
         For the general picture see \cite{LNN19,LNN20,LN14Erhart,LSz17}. Compare also with \ref{ss:6.1}.

        Another motivation  for considering the polynomial  $P_n$ is the fact that it appears as a correction term  in  Theorem 
        \ref{Thm:zetaSurgery:ZHS} (the right hand side), 
        and (since this correction term is a polynomial)
        it survives in the `polynomial part' identity of Theorem 
        \ref{Thm:PolSurgery} as well. 

       One can think about the polynomials $P_n$ as elementary `atoms'
        from which several  multivariable  functions can be shuffled. 
       
        \begin{example}\label{rem:2.3}
           Note that the reduced function $f_\Gamma(\mathbf{t}_{\mathcal{N}})$ can be expressed in the form 
            \[
                f_{\Gamma}(\mathbf{t}_{\mathcal{N}}) = \prod_{n\in\mathcal{N}_{\mathrm{n}}}P_n(x_n)\prod_{n\in\mathcal{N}_{\mathrm{e}}}\left(\frac{1}{1-x_n}-P_n(x_n)\right) = \prod_{n\in\mathcal{N}_{\mathrm{n}}}\Delta_n(x_n)\prod_{n\in\mathcal{N}_{\mathrm{e}}}\frac{\Delta_n(x_n)}{1-x_n}, 
            \]
            where $x_n = \mathbf{t}_{\mathcal{N}}^{E_{n,\Gamma}^*/\omega_n}$ for every node $n$ (compare with   (\ref{eq:reduced_duals})).
        \end{example}
        
        \begin{figure}
            \centering
            \begin{tikzpicture}[roundnode/.style={circle, draw=black, fill=black, very thick, inner sep = 1},line width=1,x=\graphscale,y=\graphscale]
        	\node[roundnode] (1) {};
        	\path (1)++(1.5,0) node[roundnode] (2) {};
        	\path (1)++(3,0) node[roundnode] (3) {};
        	\path (1)++(120:1.5) node[roundnode] (10) {};
        	\path (1)++(-120:1.5) node[roundnode] (11) {};
        	\path (2)++(90:1.5) node[roundnode] (20) {};
        	\path (3)++(60:1.5) node[roundnode] (30) {};
        	\path (3)++(-60:1.5) node[roundnode] (31) {};
        	
        	\draw (10)--(1)--(11) (1)--(3) (2)--(20) (30)--(3)--(31);
        	\begin{small}
        		\renewcommand{\r}{0.5}
        		\path (1)++(20:\r) node {$19$};	
        		\path (1)++(100:\r) node {$3$};	
        		\path (1)++(-100:\r) node {$4$};
        		\path (2)++(20:\r) node {$5$};
        		\path (2)++(160:\r) node {$7$};
        		\path (2)++(70:\r) node {$2$};	
        		\path (3)++(80:\r) node {$2$};
        		\path (3)++(-80:\r) node {$3$};
        		\path (3)++(160:\r) node {$19$};	
        	\end{small}
        \end{tikzpicture}
        \hspace{2cm}
        \begin{tikzpicture}[roundnode/.style={circle, draw=black, fill=black, very thick, inner sep = 1},line width=1,x=\graphscale,y=\graphscale]
        	\node[roundnode] (1) {};
        	\path (1)++(3,0) node[roundnode] (2) {};
        	\path (2)++(3,0) node[roundnode] (3) {};
            \path (1)++(1.25,0) node (1arrow) {};
            \path (2)++(-1.25,0) node (2Larrow) {};
            \path (2)++(1.25,0) node (2Rarrow) {};
        	\path (1)++(120:1.5) node[roundnode] (10) {};
        	\path (1)++(-120:1.5) node[roundnode] (11) {};
        	\path (2)++(90:1.5) node[roundnode] (20) {};
        	\path (3)++(60:1.5) node[roundnode] (30) {};
        	\path (3)++(-60:1.5) node[roundnode] (31) {};
            \path (3)++ (-1.25,0) node (3arrow) {};
        	
        	\draw (10)--(1)--(11) (2)--(20)  (30)--(3)--(31);
            \draw[->] (1)--(1arrow);
            \draw[->] (2)--(2Larrow);
            \draw[->] (2)--(2Rarrow);
            \draw[->] (3)--(3arrow);
        	\begin{small}
        		\renewcommand{\r}{0.5}
        		\path (1)++(20:\r) node {$19$};	
        		\path (1)++(100:\r) node {$3$};	
        		\path (1)++(-100:\r) node {$4$};
        		\path (2)++(20:\r) node {$5$};
        		\path (2)++(160:\r) node {$7$};
        		\path (2)++(70:\r) node {$2$};	
        		\path (3)++(80:\r) node {$2$};
        		\path (3)++(-80:\r) node {$3$};
        		\path (3)++(160:\r) node {$19$};	
        	\end{small}
        \end{tikzpicture}
        \caption{}
        \label{fig:splice1}
        \end{figure}
        \begin{example}\label{example:part:0}
        The splice diagram from the left hand side of Figure \ref{fig:splice1}  decomposes as shown by the right hand side of the same picture. The corresponding polynomials are
        \begin{gather*}
            P_1(t) = t+t^2+t^5 = \frac{1}{1-t} - \frac{1-t^{12}}{(1-t^3)(1-t^4)},\quad P_2(t) = 1+t=\frac{1-t^2}{1-t},\  \\P_3(t)=t = \frac{1}{1-t} -\frac{1-t^6}{(1-t^2)(1-t^3)}, \ \ \mbox{and}
             \end{gather*}
            \[ f_\Gamma(\mathbf{t}_{\mathcal{N}}) =
            \frac{1-x_1^{12}}{(1-x_1^3)(1-x_1^4)}\cdot\frac{1-x_2^2}{1-x_2}\cdot \frac{1-x_3^6}{(1-x_3^2)(1-x_3^3)}.
        \]
        \end{example}
    
    \section{Surgery formula for the topological Poincar\'e series}\label{s:surgtopseries}
        \subsection{The `peeled'   graphs}\label{ss:inducedgraphs} For notations $\mathcal{N}_{\mathrm{n}}$
        and  $\mathcal{N}_{\mathrm{e}}$ see \ref{ss:NOT}. Clearly, they can be defined for the graph $\Gamma$ as well. 
          
Let $\mathcal{P}(\mathcal{N}_{\mathrm{e}})$
        be the power  set of $\mathcal{N}_{\mathrm{e}}$. Then we define    $\mathcal{P}^*(\mathcal{N}_{\mathrm{e}})$ as 
         $\mathcal{P}(\mathcal{N}_{\mathrm{e}})$ if $|\mathcal{N}|\not=2$ (or, equivalently,
        if $\mathcal{N}_{\mathrm{n}}\not=\emptyset$)
         and  $\mathcal{P}(\mathcal{N}_{\mathrm{e}})\setminus 
         \{\emptyset\}$ otherwise.

            Associated with any subset  $I\in \mathcal{P}^*(\mathcal{N}_{\mathrm{e}})$  we construct the plumbing graph $\Gamma_I$ as follows. Consider the splice diagram associated with  $\Gamma$ and  delete all the 
            end-nodes  belonging to $\mathcal{N}_{\mathrm{e}}\setminus I $ and all the legs connected to them (together with their end-vertices). By assumption regarding $I$, the resulted splice diagram is not empty, and it might be non-minimal as well. 
            Then, from this new splice diagram, one constructs the plumbing graph $\Gamma_I$ by the algorithm from \ref{ss:splicetoplumb}.
           
         Clearly, corresponding to $I=\mathcal{N}_{\mathrm{e}}$ we have $\Gamma_I=\Gamma$. 

Our goal is to relate the invariants associated with $\Gamma$ and $\Gamma_I$, eg. the functions $f_\Gamma$ and $f_{\Gamma_I}$. In fact, in the final surgery formula (Theorem \ref{Thm:zetaSurgery:ZHS}) all the functions 
$\{f_{\Gamma_I}\}_{I\in \mathcal{P}^*(\mathcal{N}_{\mathrm{e}})}$ will appear.

\vspace{2mm}

For the next definitions we fix some $I\in  \mathcal{P}^*(\mathcal{N}_{\mathrm{e}})$. 

\vspace{2mm}

        For each end-node $n\in\mathcal{N}_{\mathrm{e}}\setminus I$ let $e_n$ be the end-vertex of the leg in $\Gamma_I$ attached to $v_n$ (where this  leg corresponds in the splice diagram $\mathfrak{S}$ 
        to the edge connecting $n$ and $v_n$). Recall that $\Gamma$ can be obtained 
        from $\Gamma_I$ if we proceed by the algorithm from \ref{ss:splicetoplumb}
        and we glue back the plumbing graphs of the star-shaped diagrams associated with the end-nodes $\mathcal{N}_{\mathrm{e}}\setminus I$.
        Now, we choose a vertex $\theta_n$ in $\Gamma_I$  as follows.
        If  $\theta_{v_n,n}$ (well-defined by this gluing, see
        \ref{ss:splicetoplumb})  is a bamboo vertex in $\Gamma$ 
        then $\theta_n:=\theta_{v_n,n}$ identified with that  vertex of $\Gamma_I$.
        Otherwise set $\theta_n$ to be the neighbour of $v_n$ in the path from $v_n$ to $e_n$ (in $\Gamma_I$). In this case,  in the graph  $\Gamma$ too we mark by 
        $\theta_n$ the neighbour of $v_n$ in the path from $v_n$ to $n$.

        Associated with $I$ we define the set $\Theta_I := \{\theta_n\mid n\in\mathcal{N}_{\mathrm{e}}\setminus I\}$ and
        \begin{equation}\label{eq:JI}
        	J_I := \mathcal{N}(\Gamma_I)\cup\Theta_I=\mathcal{N}_{\mathrm{n}}\cup I\cup\Theta_I \subset \mathcal{V}(\Gamma_I) \ \ \mbox{of cardinality $|J_I|=|\mathcal{N}|$}.
        \end{equation}
         These vertices will play a crucial role in the reduction of the corresponding Poincar\'e series.

        Finally, for any $n\in\mathcal{N}$ we define the cycles
        $X_n\in L(\Gamma)$ and $Y_n\in L(\Gamma_I)$
        	\begin{align}\label{eq:reduced_duals}
        	    X_n = \frac{1}{\omega_n}\restr{E_{n,\Gamma}^*}{\mathcal{N}}\quad\mbox{and}\quad Y_n = \begin{cases}
        			\frac{1}{\omega_n}\restr{E_{n,\Gamma_I}^*}{J_I}& \mbox{if} \ \  n\in \mathcal{N}_{\mathrm{n}}\cup I,\\
        			\restrd{E_{e_n,\Gamma_I}^*}{J_I}& \mbox{if} \ \  n\in\mathcal{N}_{\mathrm{e}}\setminus I,
        		\end{cases}
        	\end{align}
        	where  $E_{v,\Gamma}^*\in L(\Gamma)$, $E_{v,\Gamma_I}^*\in L(\Gamma_I)$. These cycles will be used to write the reduced functions $f_{\Gamma}(\mathbf{t}_{\mathcal{N}})$, $f_{\Gamma_I}(\mathbf{t}_{J_I})$ in a more compact form. In fact, they provide a formal change of variables 
            from  $\mathbf{t}_{\mathcal{N}}$ to $\mathbf{t}_{J_I}$. This will be discussed in section \ref{ss:varchange}.

            Let $\mathcal{B} = (\mathcal{B}_{u,v})_{u,v\in\mathcal{N}}$ be the orbifold intersection matrix of $\Gamma$ (cf. \cite[\S 4.1.3]{BN07}).
            It can be defined via its inverse, which 
            equals  $\mathcal{B}^{-1} = ((E_{u,\Gamma}^*,E_{v,\Gamma}^*))_{u,v\in\mathcal{N}}$. We have that $\mathcal{B}_{u,v}=0$ if $u$ and $v$ are \emph{not} connected by a  bamboo. This means that for any $l = \sum_{n}a_n\restr{E_{n,\Gamma}^*}{\mathcal{N}} = \sum_nm_nE_{n,\Gamma}$, we have
        \begin{equation}\label{eq:Orbifold}
            a_{v} = \mathcal{B}_{v,v} m_v+ \sum_{u\sim v} \mathcal{B}_{v,u}m_u.
        \end{equation}
        In particular, if $n\in\mathcal{N}_{\mathrm{e}}$, then $a_n = \mathcal{B}_{n,n}m_n + \mathcal{B}_{v_n,n}m_{v_n}$. For the concrete values of $\mathcal{B}_{u,v}$ see also \cite[\S 2.1.4]{LSz17}.
        \begin{Lem}\label{lem:XnYnBasis}
            $\{X_n\mid n\in\mathcal{N}\}$ is a basis of\, $\Q\langle E_{n,\Gamma}\mid n\in\mathcal{N}\rangle\subset L(\Gamma)\otimes\Q$ and $\{Y_n\mid n\in \mathcal{N}\}$ is a basis of $\Q\langle E_{v,\Gamma_I}\mid v\in J_I\rangle\subset L(\Gamma_I)\otimes\Q$.
        \end{Lem}
        \begin{proof}
            First, note that $\{\restr{E^*_{n,\Gamma}} {\mathcal{N}}\}_{n\in\mathcal{N}}$  is a basis for $\Q\langle E_{n,\Gamma}\mid n\in\mathcal{N}\rangle$. 
            Indeed, the transition matrix between 
            $\{\restr{E^*_{n,\Gamma}}{\mathcal{N}}\}_{n\in\mathcal{N}}$ and $\{E_{n,\Gamma}\}_{n\in\mathcal{N}}$ is given by $(m_{u}(E_{v,\Gamma}^*))_{u,v\in\mathcal{N}} = ((-E_{u,\Gamma}^*,E_{v,\Gamma}^*))_{u,v\in\mathcal{N}} = (-\mathcal{B})^{-1}$.
            Then, since  $\restr{E_{n,\Gamma}^*}{\mathcal{N}} = \omega_n X_n$, 
            $\{X_n\}_{n\in\mathcal{N}}$ is a basis as well. 

            The second part follows in a similar way.  
            On one hand,  $\omega_n Y_n = \restr{E_{n,\Gamma_I}^*}{J_I}$ for $n\in\mathcal{N}_{\mathrm{n}}\cup I$. On the other hand, $\det\big((\Gamma_I)_{\theta_n,e_n}\big) \cdot Y_n=\det\big((\Gamma_I)_{\theta_n,e_n}\big) \cdot \restr{E_{e_n,\Gamma_I}^*}{J_I} = \restr{E_{\theta_n,\Gamma_I}^*}{J_I}$ for $n\in\mathcal{N}_{\mathrm{e}}\setminus I$ (recall from \ref{ss:splice} the definition of $(\Gamma_I)_{\theta_n,e_n}$).  
            Thus, it suffices to show that
            $\{\restr{E_{v,\Gamma_I}^*}{J_I}\}_{v\in J_I}$ is a basis of $\Q\langle E_{v,\Gamma_I}\mid v\in J_I\rangle$. For this consider the tree $\Gamma_I'$ which is obtained by blowing up each vertex  $\theta_n$ in $\Gamma_I$, such that each $\theta_n$ becomes a node. In this case, we have  $m_u(E_{v,\Gamma_I'}^*) = m_u(E_{v,\Gamma_I}^*)$ for all $u,v\in J_I$. Thus, the transition matrix from $\{\restr{E_{v,\Gamma_I}^*}{J_I}\}_{v\in J_I}$ to $\{E_{v,\Gamma_I}\}_{v\in J_I}$ is the negative inverse of the orbifold intersection matrix of $\Gamma_I'$.
        \end{proof}
        
            With all these notations we can express the function $f_{\Gamma_I}$ associated with the graph $\Gamma_I$ and reduced to $J_I$ in the following form. 
        \begin{Lem}\label{lem:induced_expr}
        	Set $y_n := \mathbf{t}_{J_I}^{Y_n}$ for every $n\in\mathcal{N}$. Then we have
        	\[
        		f_{\Gamma_I}(\mathbf{t}_{J_I}) = \prod_{n\in\mathcal{N}_{\mathrm{n}}\cup I}\frac{(1-y_n^{\omega_n})^{\delta_n-2}}{\prod\limits_{\substack{e\in\mathcal{E}\\n-_be}}(1-y_n^{\omega_{e}})}\prod_{n\in \mathcal{N}_{\mathrm{e}}\setminus I}\frac{1}{1-y_n}.
        	\]
        \end{Lem}
        \begin{proof}
        	By definition (\ref{eq:zeta:def}), the function $f_{\Gamma_I}$ reduced to $J_I$ can be written as 
        	\[
        		f_{\Gamma_I}(\mathbf{t}_{J_I}) = \prod_{n\in\mathcal{N}_{\mathrm{n}}\cup I}\frac{(1-\mathbf{t}_{J_I}^{E_{n,\Gamma_I}^*})^{\delta_n-2}}{\prod\limits_{\substack{e\in\mathcal{E}\\n-_be}}(1-\mathbf{t}_{J_I}^{E_{e,\Gamma_I}^*})}\prod_{n\in \mathcal{N}_{\mathrm{e}}\setminus I}\frac{1}{1-\mathbf{t}_{J_I}^{E_{e_n,\Gamma_I}^*}}.
        	\]
        	Since for $e\in\mathcal{E}$, $n-_be$, the cycle $-E_{n,\Gamma_I}^*+d_{n,e}E_{e,\Gamma_I}^*$ is supported on $(\Gamma_I)_{n,e} $
            (cf. \cite[Lemma 2.2.1]{LSz17}), we have $d_{n,e}\restr{E_{e,\Gamma_I}^*}{J_I} = \restr{E_{n,\Gamma_I}^*}{J_I} = \omega_n Y_n$ and the result follows.
        \end{proof}

    \subsection{Change of variables}\label{ss:varchange}

        In the following, we show that there is a formal  change of variables from $\mathbf{t}_{J_I}$ to $\mathbf{t}_{\mathcal{N}}$ that transforms $y_n:=\mathbf{t}_{J_I}^{Y_n}$ into $x_n:=\mathbf{t}_{\mathcal{N}}^{X_n}$.
       
        By Lemma \ref{lem:XnYnBasis} we can consider the $\Q$--linear isomorphism $$\varphi_I\colon \Q\langle E_{v,\Gamma_I}\mid v\in J_I\rangle\to \Q\langle E_{n,\Gamma}\mid n\in\mathcal{N}\rangle$$ 
        given by $\varphi_I(Y_n) = X_n$ for $n\in\mathcal{N}$.
        Then  $\varphi_I$ induces  a change of variable map from $\mathbf{t}_{J_I}$ to $\mathbf{t}_{\mathcal{N}}$. More precisely, if $\varphi_I$ is given by
        \[
        \varphi_I(E_{v,\Gamma_I}) = \sum_{n\in\mathcal{N}}\alpha_{v,n} E_{n,\Gamma},\quad v\in J_I ;
        \]
        then, for any $l\in  \Q\langle E_{v,\Gamma_I}\mid v\in J_I\rangle$, 
        \[
        	\mathbf{t}_{J_I}^l \mapsto \mathbf{t}_{\mathcal{N}}^{\varphi_I(l)}\colon t_v\mapsto \prod\limits_{n\in\mathcal{N}}t_n^{\alpha_{v,n}}.
        \]
        In the following, the map $\varphi_I$ will be also called the  change of variable map from $\Gamma_I$ to $\Gamma$.
        \begin{Def}\label{def:operator}
        For any series $S(\mathbf{t}_{J_I})$ in variables $\mathbf{t}_{J_I}$ we will denote by  $(\varphi_I^*S)(\mathbf{t}_\mathcal{N})$ the series in variables $\mathbf{t}_\mathcal{N}$, which is obtained from $S(\mathbf{t}_{J_I})$ by applying the  change of variables given by $\varphi_I$. 
        \end{Def}
        \begin{Not}
            If $S(\mathbf{t}_{V})$ is a series in variables $\mathbf{t}_V$, with $J_I\subseteq V$, then $(\varphi_I^*S)(\mathbf{t}_\mathcal{N})$ denotes the series in $\mathbf{t}_{\mathcal{N}}$ obtained by applying $\varphi_I^*$ to the reduced series $S(\mathbf{t}_{J_I}):=S(\mathbf{t}_{V})_{t_v\mapsto1,v\in V\setminus J_I}$. 
        \end{Not}

        Now, applying the operator $\varphi^*_I$ from Definition \ref{def:operator} to the function $f_{\Gamma_I}(\mathbf{t}_{J_I})$,  we get the following expression in variables $\mathbf{t}_{\mathcal{N}}$
        \begin{equation}\label{eq:reducedzeta_xn}
         (\varphi_I^*f_{\Gamma_I})(\mathbf{t}_{\mathcal{N}}) = \prod_{n\in\mathcal{N}_{\mathrm{n}}\cup I}\frac{(1-x_n^{\omega_n})^{\delta_n-2}}{\prod_{n-_be}(1-x_n^{\omega_{e}})}\prod_{n\in \mathcal{N}_{\mathrm{e}}\setminus I}\frac{1}{1-x_n}, 
         \ \mbox{where} \ \ 
         x_n:= \mathbf{t}_{\mathcal{N}}^{X_n}. 
        \end{equation}
        We emphasize again that this is a series in variables $\mathbf{t}_{\mathcal{N}}$, 
        uniformly for any $I$. 
        
        \begin{Rem} {\bf Why the peeling?} \ The zeta function $Z(\mathbf {t})$, considered in all its variables 
        $\{t_v\}_{v\in\mathcal{V}}$ usually is very hard to compute and handle. On the other hand,
        (almost) all the important information coded in $Z(\mathbf {t})$ is already coded in its reduced version
        $Z(\mathbf{t}_{\mathcal{N}})$, reduced to the variables of the nodes. (In some sense, this fact is already suggested by the existence of the splice diagram as well.)
        But, if we reduce more, then usually we lose certain information.
        
         One way to understand and compute $Z(\mathbf {t}_{\mathcal{N}})$ is by surgeries, or cutting the original 
         graph (or diagram) in smaller pieces. However, if the pieces produce series in smaller number of variables, 
         then usually $Z(\mathbf {t}_{\mathcal{N}})$ in $|\mathcal{N}|$ variables, cannot be recovered. Eg. this happens if they are too small. 
         Hence, one has to try to consider smaller (but not very small)  pieces which still have $|\mathcal{N}|$ `basic', well-chosen vertices, which produce series in  $|\mathcal{N}|$ variables and carry essential global information, and which (after certain change of variables) reconstruct $Z(\mathbf{t}_{\mathcal{N}})$.
        
         Such possible sub-diagrams are the peelings of $\mathfrak{S}$. In their case the deleted nodes $\mathcal{N}_{{\rm e}}\setminus I$
         are replaced by the `gluing vertices' $\Theta_I$  of the same cardinality. 
        \end{Rem}
        
    \subsubsection{\bf Properties of the operator  $\varphi_I$}
        Though the Definition \ref{def:operator} gives a well defined operator  $\varphi_I$,
        from that definition we do not immediately see its structure. In fact, this operator 
        preserves all the coefficients of the $\mathcal{N}_n\cup I$ coordinates 
        (recall that these are exactly the common coordinates of the index sets 
        $J_I$ and $\mathcal{N}$,  $J_I\cap\mathcal{N}=\mathcal{N}_n\cup I $)
        and the modification of the other coefficients can also be determined naturally from the decorations of the splice diagram.

         \begin{Lem}\label{lem:varRestr}
            $\restr{\varphi_I(l)}{\mathcal{N}_{\mathrm{n}}\cup I} = \restr{l}{\mathcal{N}_{\mathrm{n}}\cup I}$ for all $l\in \Q\langle E_{v,\Gamma_I}\mid v\in J_I\rangle$.
             In particular, if $I=\mathcal{N}_{\mathrm{e}}$ then $\varphi_{\mathcal{N}_{\mathrm{e}}}^*$ is the identity operator.
        \end{Lem}
        \begin{proof}
            By linearity, it suffices to show that for any $v\in\mathcal{N}_{\mathrm{n}}\cup I$ we have that $m_v(E_{n,\Gamma}^*)= m_v(E_{n,\Gamma_I}^*)$ for $n\in \mathcal{N}_{\mathrm{n}}\cup I$, and $\omega_n m_v(E_{e_n,\Gamma_I}^*) = m_v(E_{n,\Gamma}^*)$ for $n\in\mathcal{N}_{\mathrm{e}}\setminus I$. 
            Consider the splice diagrams $\mathfrak{S}$, $\mathfrak{S}_I$ of $\Gamma$ and $\Gamma_I$ respectively. For the first part observe that the path from $v$ to $n$ is the same, together with the surrounding weights. In other words $\mathfrak{S}_I[v,n] = \mathfrak{S}[v,n]$ and $w(\mathfrak{S}_I[v,n]) = w(\mathfrak{S}[v,n])$, hence $m_v(E_{n,\Gamma}^*)= m_v(E_{n,\Gamma_I}^*)$.
            For the second part observe that $\mathfrak{S}_I[e_n,v]\setminus\{e_n\} = \mathfrak{S}[n,v]\setminus\{n\}$, thus $ w(\mathfrak{S}_I[e_n,v])=w(\mathfrak{S}(n,v])$. Therefore $\omega_n\cdot m_v(E_{e_n,\Gamma_I}^*) = w(\mathfrak{S}[n,v])= m_v(E_{n,\Gamma}^*)$.
        \end{proof} 

        For the remaining coefficients, we have to compare the 
        $\mathcal{N}_{\mathrm{e}}\setminus I$ coefficients of the cycles $X_n$ with the 
        $\Theta_I=\{\theta_n\,:\, n\in \mathcal{N}_{\mathrm{e}}\setminus I\}$ coefficients of the cycles $Y_n$. We observe that in the definition of the vertex $\theta_n$ we had some kind of freedom: 
        in fact, one could choose any other vertex from the leg $(\Gamma_I)_{v_n,e_n}$ to be $\theta_n$ and serve as a base element for the splice operation. 
    
        With all such choices  the defining property  $\varphi_I(Y_n)=X_n$, and Lemma \ref{lem:varRestr} still hold. The present choice of $\theta_n$ 
        from \ref{ss:inducedgraphs} gives a simpler (and natural) expression of $\varphi_I(l)$ outside $\mathcal{N}_{\mathrm{n}}\cup I$ as we will determine  next. 

        To do this we need a way to determine the $E_{w}$ coordinate ($\delta_{w}=2$) of a cycle $l$ which is given by a linear combination of the cycles $E_n^*$.  Fix two vertices $u,v$ connected by a bamboo, and assume that $l\in L(\Gamma)\otimes\Q$ satisfies $(l,E_w)_{\Gamma}=0$ for $w\in\Gamma(u,v)$ (that is, $a_w(l)=0$, cf. (\ref{eq:coords})). Let us assume that $\Gamma(u,v)$ is non-empty and denote $u',v'$ the vertices connecting to $u,v$ respectively in $\Gamma[u,v]$. Consider the reduced cycle $\restr{l}{\Gamma(u,v)}\in L(\Gamma(u,v))\otimes \Q$. We have that $(\restr{l}{\Gamma(u,v)},E_w)_{\Gamma(u,v)} =0$ for $w\in\Gamma(u',v')$, $(\restr{l}{\Gamma(u,v)},E_{u'})_{\Gamma(u,v)}  = -m_u(l)$, and $(\restr{l}{\Gamma(u,v)},E_{v'})_{\Gamma(u,v)}  = -m_v(l)$. In other words
        $\restr{l}{\Gamma(u,v)} = m_u(l)E_{u',\Gamma(u,v)}^* + m_v(l)E_{u',\Gamma(u,v)}^*$.
        Thus for any $\theta\in\Gamma(u,v)$ we have
        \begin{equation}\label{eq:midBamboo}
            m_{\theta}(l)\cdot\det\Gamma(u,v)  = m_{u}(l)\cdot \det\Gamma(\theta,v) +m_{v}(l)\cdot \det\Gamma(u,\theta).
        \end{equation}
        To extend the domain of this formula we define $\det\Gamma(w,w):=0$. With this, the above formula holds when $u$ is connected to $v$, or $\theta\in\{u,v\}$ as well.
    
        \begin{Rem}\label{Rem:varLocality}
            It is important to mention that all the values $\mathcal{B}_{n,n}$, $\mathcal{B}_{v_n,n}$, $\det\Gamma(v_n,n)$, $\det\Gamma(v_n,\theta)$, $\det\Gamma(\theta,n)$ are computed from the subtree $\Gamma_{v_n,n}$, therefore they depend only on the weights in $\mathfrak{S}$ that are surrounding the nodes $v_n$ and  $n$. This ensures that the coefficients $m_{\theta_n}(\varphi_I^{-1}(l)) = \alpha_n\cdot m_{n}(l) + \beta_n\cdot m_{v_n}(l)$ are local with respect to the nodes $\{v_n, n\}$. 
        \end{Rem}

        \begin{Lem}\label{Lem:legRestrs}
            For each $n\in\mathcal{N}_{\mathrm{e}}$ there exist constants $\alpha_n,\beta_n\in\Q$ such that if $n\in\mathcal{N}_{\mathrm{e}}\setminus I$ then $m_{\theta_n}(\varphi_I^{-1}(l)) = \alpha_n m_n(l) + \beta_n m_{v_n}(l)$. These constants depend only on the weights in $\mathfrak{S}$ that surround the nodes $v_n$, $n$.
            In particular, if $\theta_{v_n,n} = \theta_n$, then
            \[
            \alpha_n = \frac{\det\Gamma(v_n,\theta_n)}{\det\Gamma(v_n,n)},\quad \beta_n = \frac{\det\Gamma(\theta_n,n)}{\det\Gamma(v_n,n)}.
            \]
        \end{Lem}
     
        \begin{proof}
            The proof consists of three main steps. First we assume that $n\in\mathcal{N}_{\mathrm{e}}\setminus I$ and introduce constants $c_n,c_n'$ such that $\omega_u\cdot m_{\theta_n}(Y_u) = c_n\cdot m_{\theta_n}(E_{u,\Gamma}^*)$ for $u\in\mathcal{N}\setminus\{n\}$, and $\omega_n\cdot m_{\theta_n}(Y_n) = c_n'\cdot m_{\theta_n}(E_{n,\Gamma}^*)$.
            In the second step we compute $m_{\theta_n}(\varphi_I^{-1}(l))$ in terms of $m_{n}(l)$ and $m_{v_n}(l)$, that will use $c_n$, $c_n'$, (\ref{eq:Orbifold}), (\ref{eq:midBamboo}) and some additional sub-tree determinants.
            Finally in step three we will show that $\alpha_n,\beta_n$ depend only on the weights in $\mathfrak{S}$ that surround the nodes $n$ and $v_n$.

            {\bf Step 1.} 
            To construct $\alpha_n,\beta_n$ we assume that $n\in\mathcal{N}_{\mathrm{e}}\setminus I$.         
            Denote the weight in $\mathfrak{S}_I$ corresponding to $u\sim v$ with $d_{u,v}^I$. Note that $d_{v_n,e_n}^I =d_{v_n,n}$.

            Consider the determinants $d_{\theta_n,e_n}^I$, $d_{\theta_n,v_n}^I$ corresponding to $\Gamma_I$ as defined in section \ref{ss:splice}. Similarly $d_{\theta_n,n}$, $d_{\theta_n,v_n}$ for $\Gamma$ with the caveat that in the case $\theta_n = n$ we define $d_{\theta_{n},n}$ to be $\omega_n$. Thus the constants $c_n :=d_{\theta_n,e_n}^I/d_{\theta_n,n}$ and $c_n' :=d_{\theta_n,v_n}^I/d_{\theta_n,v_n}$
            allow us to write that
            \[
                m_{\theta_n}(E_{u,\Gamma_I}^*) =w(\mathfrak{S}_I[u,v_n]) \frac{d_{\theta_n,e_n}^I}{d_{v_n,e_n}^I}= c_nw(\mathfrak{S}[u,v_n])\frac{d_{\theta_n,n}}{d_{v_n,n}} =   c_n m_{\theta_n}(E_{u,\Gamma}^*)
            \]
            for all $u\in \mathcal{N}_{\mathrm{n}}\cup I$, 
            \[
                \omega_u\cdot m_{\theta_n}(E_{e_u,\Gamma_I}^*) = \omega_u w(\mathfrak{S}_I[e_u,v_n]) \frac{d_{\theta_n,e_n}^I}{d_{v_n,e_n}^I}= c_n w(\mathfrak{S}[u,v_n])\frac{d_{\theta_n,n}}{d_{v_n,n}}= c_nm_{\theta_n}(E_{u,\Gamma}^*),
            \]for $u\in\mathcal{N}_{\mathrm{e}}\setminus I$ with $u\neq n$, and
            \[
                \omega_n\cdot m_{\theta_n}(E_{e_n,\Gamma_I}^*) =\omega_n d_{\theta_n,v_n}^I=c_n' \omega_nd_{\theta_n,v_n}= c_n'm_{\theta_n}(E_{n,\Gamma}^*).
            \]

            {\bf Step 2.} Suppose $l\in\Q\langle E_u\mid u\in\mathcal{N} \rangle$ is given by $l = \sum_u a_u(l)\restr{E_{u,\Gamma}^*}{\mathcal{N}}$ and consider the (unrestricted) cycle $l_1 = \sum_u a_u(l)E_{u,\Gamma}^*$. Note that $\restr{l_1}{\mathcal{N}} = l$. We have that
            \begin{align*}
                m_{\theta_n}(\varphi_I^{-1}(l)) &= \sum_{u\in\mathcal{N}_{\mathrm{n}}\cup I}a_u(l)m_{\theta_n}(E_{u,\Gamma_I}^*) + \sum_{\substack{u\in\mathcal{N}_{\mathrm{e}}\setminus I\\u\neq n}}a_u(l)\omega_um_{\theta_n}(E_{e_u,\Gamma_I}^*) + a_n(l)\omega_nm_{\theta_n}(E_{e_n,\Gamma_I}^*)\\
                &=c_n m_{\theta_n}(l_1) + (c_n'-c_n)a_n(l)m_{\theta_n}(E_{n,\Gamma}^*).
            \end{align*}
            Since   $m_{\theta_{n}}(E_{n,\Gamma}^*) = d_{\theta_n,v_n}\omega_n$,  $a_n(l) = \mathcal{B}_{n,n}m_n(l) + \mathcal{B}_{v_n,n}m_{v_n}(l)$ (see (\ref{eq:Orbifold})), and 
            \[
                m_{\theta_n}(l_1) = \frac{\det\Gamma(v_n,\theta_n)}{\det{\Gamma}(v_n,n)}m_n(l_1) +\frac{\det\Gamma(\theta_n,n)}{\det{\Gamma}(v_n,n)}m_{v_n}(l_1),
            \]
            (see (\ref{eq:midBamboo}))
             we have $m_{\theta_n}(\varphi_I^{-1}(l)) = \alpha_nm_{n}(l) + \beta_n m_{v_n}(l)$, where
            \begin{align*}
                \alpha_n &:= c_n \frac{\det\Gamma(v_n,\theta_n)}{\det{\Gamma}(v_n,n)} + (c_n'-c_n)d_{\theta_n,v_n}\omega_n\mathcal{B}_{n,n},\\
                \beta_n &:= c_n\frac{\det\Gamma(\theta_n,n)}{\det{\Gamma}(v_n,n)} + (c_n'-c_n)d_{\theta_n,v_n}\omega_n\mathcal{B}_{v_n,n}.
            \end{align*}

            {\bf Step 3.} Now we prove that $\alpha_n$ and $\beta_n$ depend only on the weights in $\mathfrak{S}$ surrounding $v_n$, $n$. It suffices to show this same dependence for $c_n$, $c_n'$, and $d_{\theta_n,v_n}$, as everything else appearing in the definition of $\alpha_n$, $\beta_n$ are computed from $\Gamma_{v_n,n}$ (see Remark \ref{Rem:varLocality}). Using the graph determinant formula (cf. \cite[Lemma 4.0.1]{BN10}) we can write that
            \begin{align*}
                \det\Gamma(v_n,\theta_n) &= d_{\theta_n,v_n}d_{v_n,n} - w(\mathfrak{S}[v_n,n))d_{\theta_n,n},\\
                \det\Gamma_I(v_n,\theta_n) &= d_{\theta_n,v_n}^Id_{v_n,e_n}^I - w(\mathfrak{S}_I[v_n,n))d_{\theta_n,e_n}^I.
            \end{align*}
            By construction the graphs $\Gamma_I(v_n,\theta_n)$ and $\Gamma(v_n,\theta_n)$ are identical, thus
            \begin{align*}
                d_{\theta_n,v_n} &= \frac{1}{d_{v_n,n}}(\det\Gamma(v_n,\theta_n) + w(\mathfrak{S}[v_n,n))d_{\theta_n,n}),\\
                d_{\theta_n,v_n}^I &= \frac{1}{d_{v_n,n}}(\det\Gamma(v_n,\theta_n) + w(\mathfrak{S}[v_n,n))d_{\theta_n,e_n}^I).
            \end{align*}            
            Therefore $c_n'$, $c_n$, $d_{\theta_n,v_n}$ can be computed from $\Gamma_{v_n,n}$ and $(\Gamma_{I})_{v_n,e_n}$, which are constructed using the weights in $\mathfrak{S}$ that surround the nodes $v_n$, $n$.

            Finally let us consider the case when $\theta_{v_n,n} = \theta_n$. For this part of the proof, as shown above, we may assume that $I =\mathcal{N}_{\mathrm{e}}\setminus\{n\}$. Under these two assumptions the graphs $\Gamma_{\theta_n,v_n}$ and  $(\Gamma_I)_{\theta_n,v_n}$ are the same.  Hence $d_{\theta_n,v_n} = d_{\theta_n,v_n}^I$, meaning that $c_n' = c_n = 1$. 
        \end{proof}
        \begin{Rem}
            In fact, $c_n=c_n'=1$ if the decoration of the node in $\Sigma_{v_n}(d_{v_n,u}\mid u\sim v_n)$ remains the same after the 0-absorption in the gluing with $\Sigma_n(d_{n,u}\mid u\sim n)$ along the edge $v_n\sim n$.
        \end{Rem}
        The expression of $\varphi_I$ is given by
        \begin{equation}\label{eq:changeofvar}
        	\varphi_I(l) = \restr{l}{\mathcal{N}_{\mathrm{n}}\cup I} + \sum_{n\in\mathcal{N}_{\mathrm{e}}\setminus I}\frac{m_{\theta_n}(l)-\beta_n m_{v_n}(l)}{\alpha_n}E_{n,\Gamma}.
        \end{equation}
        Equivalently
        \begin{equation}\label{eq:varChange_basis}
            \varphi_I(E_{\theta_n,\Gamma_I})= \frac{1}{\alpha_n}E_{n,\Gamma},\quad\varphi(E_{u,\Gamma_I}) = E_{u,\Gamma}-\sum_{\substack{u-_b n\\ n\in\mathcal{N}_{\mathrm{e}}\setminus I}}\frac{\beta_n}{\alpha_n}E_{n,\Gamma},\,\forall u\in\mathcal{N}_{\mathrm{n}}\cup I.
        \end{equation} 
        By (\ref{eq:changeofvar}), since $\alpha_n,\beta_n$ depend only on the weights surrounding the nodes $n,v_n$, as long as $n\notin I$, the linear combination determining the $E_n$ coefficient (exactly) from the coefficients $E_{v_n},E_{\theta_n}$ is the same. This will be used later in the proof of Lemma \ref{Lem:PolIndependence}.
       
        \begin{example}\label{example:part:1}
            Let us consider again the splice diagram from the left hand side of Figure \ref{fig:splice1}. The corresponding plumbing graph $\Gamma$ is shown by Figure \ref{fig:plumbing1}. 
            \begin{figure}[h!]
                \centering
                \begin{tikzpicture}[roundnode/.style={circle, draw=black, fill=black, very thick, inner sep = 1},line width=1,x=\graphscale,y=\graphscale]
                	\node[roundnode] (1) {};
                	\path (1)++(3,0) node[roundnode] (2) {};
                	\foreach \i/\j in {1/4,2/5,3/2,4/6,5/3}{
                		\path (1)++(\i,0) node[roundnode] (\j) {};
                	}
                	\path (1)++(120:1) node[roundnode] (8) {};
                	\path (8)++(180:1) node[roundnode] (7) {};
                	
                	\path (1)++(-120:1) node[roundnode] (11) {};
                	\foreach \i/\j in {1/10,2/9}{
                		\path (11)++(180:\i) node[roundnode] (\j) {};
                	}
                	\path (2)++(0,1) node[roundnode] (12) {};
                	\path (3)++(60:1) node[roundnode] (13) {};
                	\path (3)++(-60:1) node[roundnode] (14) {};
                	
                	\draw (7)--(8)--(1)--(11)--(9) (1)--(3) (2)--(12) (13)--(3)--(14);
                	\begin{small}
                		\renewcommand{\r}{0.35}
                		
                		\foreach \w/\i in {2/1}{
                			\path (\i)++(180:\r) node {$-\w$};
                		}
                		\foreach \w/\i in {2/13,3/14,1/3}{
                			\path (\i)++(0:\r) node {$-\w$};
                		}
                		\foreach \w/\i in {2/4,7/5,1/2,11/6,2/11,2/10,2/9}{
                			\path (\i)++(-90:\r) node {$-\w$};
                		}
                		\foreach \w/\i in {2/7,2/8,2/12}{
                			\path (\i)++(90:\r) node {$-\w$};
                		}
                		\path (1)++({\r/tan(60)},\r) node {$E_1$};
                		\path (2)++({\r/tan(60)},\r) node {$E_2$};
                		\path (3)++({\r/tan(120)},\r) node {$E_3$};
                		\path (5)++(0,\r) node {$E_4$};
                		\path (6)++(0,\r) node {$E_5$};
                	\end{small}
                \end{tikzpicture}
                \caption{}
                \label{fig:plumbing1}
            \end{figure}
        The node vertices are $E_1,E_2,E_3$, and in this case $\mathcal{N}_{\mathrm{e}} = \{E_1,E_3\}$ and $\mathcal{N}_{\mathrm{n}} = \{E_2\}$. The vertices obtained from zero absorption are $E_4,E_5$; both are bamboo vertices.

        For constructing $\Gamma_{I}$ with $I = \{E_1\}$, we first remove from the splice diagram of $\Gamma$ the node $\mathcal{N}_{\mathrm{e}}\setminus I = \{E_3\}$
        and all the legs connected to it.
        Then we 
        construct the associated plumbing graph $\Gamma_I$, see  Figure \ref{fig:plumbing2}. Note that the vertex $E_5$ is still present in this plumbing graph, though with a different weight. 
        \begin{figure}[h!]
            \centering
            \begin{tikzpicture}[roundnode/.style={circle, draw=black, fill=black, very thick, inner sep = 1},line width=1,x=\graphscale,y=\graphscale]
        	\node[roundnode] (1) {};
        	\path (1)++(1.5,0) node[roundnode] (2) {};
        	\path (1)++(3,0) node[roundnode] (3) {};
        	\path (1)++(120:1.5) node[roundnode] (10) {};
        	\path (1)++(-120:1.5) node[roundnode] (11) {};
        	\path (2)++(90:1.5) node[roundnode] (20) {};
        	
        	\draw (10)--(1)--(11) (1)--(3) (2)--(20);
        	\begin{small}
        		\renewcommand{\r}{0.5}
        		\path (1)++(20:\r) node {$19$};	
        		\path (1)++(100:\r) node {$3$};	
        		\path (1)++(-100:\r) node {$4$};
        		\path (2)++(20:\r) node {$5$};
        		\path (2)++(160:\r) node {$7$};
        		\path (2)++(70:\r) node {$2$};		
        	\end{small}
        \end{tikzpicture}
        \begin{tikzpicture}[roundnode/.style={circle, draw=black, fill=black, very thick, inner sep = 1},line width=1,x=\graphscale,y=\graphscale]
        	\node[roundnode] (1) {};
        	\path (1)++(3,0) node[roundnode] (2) {};
        	\foreach \i/\j in {1/4,2/5,3/2,4/6}{
        		\path (1)++(\i,0) node[roundnode] (\j) {};
        	}
        	\path (1)++(120:1) node[roundnode] (8) {};
        	\path (8)++(180:1) node[roundnode] (7) {};
        	
        	\path (1)++(-120:1) node[roundnode] (11) {};
        	\foreach \i/\j in {1/10,2/9}{
        		\path (11)++(180:\i) node[roundnode] (\j) {};
        	}
        	\path (2)++(0,1) node[roundnode] (12) {};
        	
        	\draw (7)--(8)--(1)--(11)--(9) (1)--(6) (2)--(12);
        	\begin{small}
        		\renewcommand{\r}{0.35}
        		
        		\foreach \w/\i in {2/1}{
        			\path (\i)++(180:\r) node {$-\w$};
        		}
        		\foreach \w/\i in {2/4,7/5,1/2,5/6,2/11,2/10,2/9}{
        			\path (\i)++(-90:\r) node {$-\w$};
        		}
        		\foreach \w/\i in {2/7,2/8,2/12}{
        			\path (\i)++(90:\r) node {$-\w$};
        		}
        		\path (1)++({\r/tan(60)},\r) node {$E_1$};
        		\path (2)++({\r/tan(60)},\r) node {$E_2$};
        		\path (5)++(0,\r) node {$E_4$};
        		\path (6)++(0,\r) node {$E_5$};
        	\end{small}
        \end{tikzpicture}
            \caption{}
            \label{fig:plumbing2}
        \end{figure}
        
        The variables of the corresponding reduced function $f_{\Gamma_I}$ are $J=\{E_1,E_2,E_5\}$. Moreover, it can be expressed as
        \[
         f_{\Gamma_{I}}(t_1,t_2,t_5) = \frac{1-y_1^{12}}{(1-y_1^3)(1-y_1^4)}\cdot\frac{1-y_2^2}{1-y_2}\cdot\frac{1}{1-y_3},
        \]
        where $y_1 = t_1^{19}t_2^{10}t_5^{2}$, $y_2 = t_1^{60}t_2^{35}t_5^{7}$, and $y_3 = t_1^{24}t_2^{14}t_5^{3}$. To obtain the change of variable map, observe that $11m_5(l) = m_2(l)+m_{3}(l)$ for all $l$ with $(l,E_5)_{\Gamma}=0$. Thus $\alpha_3 =\beta_3 = 1/11$. 
        By (\ref{eq:varChange_basis}) the change of variable map $\varphi_I\colon\Q\langle E_1,E_2,E_5\rangle\to\Q\langle E_1,E_2,E_3\rangle$ is given by
        \[
        \varphi_I(E_1) = E_1,\quad \varphi_I(E_2) = E_2-\frac{\beta_3}{\alpha_3}E_3,\quad \varphi_I(E_5) = \frac{1}{\alpha_3}E_3.
        \]
        The corresponding formal variable change is: $t_1\mapsto t_1$, $t_2\mapsto t_2t_3^{-1}$, $t_5\mapsto t_3^{11}$. 
        Note that
        \begin{align*}
            \varphi_I^*(y_1) = t_1^{19}t_2^{10}t_3^{2\cdot 11 - 10} = t_1^{19}t_2^{10}t_3^{12} &=: x_1,\quad
            \varphi_I^*(y_2) = t_1^{60}t_2^{35}t_3^{7\cdot 11 - 35} = t_1^{60}t_2^{35}t_3^{42} =: x_2,\\
            \varphi_I^*(y_3)&= t_1^{24}t_2^{14}t_3^{3\cdot 11 - 14} = t_1^{24}t_2^{14}t_3^{19} =: x_3.
        \end{align*}
        Finally, we get 
        \[
            (\varphi_{\{E_1\}}^*f_{\Gamma_{\{E_1\}}})(t_1,t_2,t_3)= \frac{1-x_1^{12}}{(1-x_1^3)(1-x_1^4)}\cdot\frac{1-x_2^2}{1-x_2}\cdot\frac{1}{1-x_3}.
        \]
        \end{example}
        \begin{example}
            Let us consider the plumbing graph in Figure \ref{fig:splice_0_absorption_on_node} (the graph on the right). Denote by $E_1$ the left node, and with $E_2$ the right node.  First, consider $\varphi_{\{E_1\}}$. In this case $\theta_2$ is adjacent to $E_1$ in $\Gamma_{\{E_1\}}$, and 
    		  \[
    		X_1 = 11E_1+20E_2,\quad X_2=6E_1+11E_2,\quad Y_1=11E_1+9E_{\theta_2},\quad Y_2=6E_1+5E_{\theta_2}.
    		\]
    		Thus $\varphi_{\{E_1\}}(E_1) = E_1+E_2$ and $\varphi_{\{E_1\}}(E_{\theta_2}) = E_2$.
    		For $\varphi_{\{E_2\}}$, the vertex $\theta_1$ is adjacent to $E_2$ in $\Gamma_{\{E_2\}}$ and $Y_1 = 11E_{\theta_1}+20E_2$, $Y_2 =6E_{\theta_1}+11E_2 $, thus $\varphi_{\{E_2\}}(E_{\theta_1})=E_1$ and $\varphi_{\{E_2\}}(E_{2})=E_2$.
        \end{example}

        \subsection{The surgery formula}
    
            For convenience, in the case when $\mathcal{N}_{\mathrm{n}} = \emptyset$ (i.e. $|\mathcal{N}|=2$) and for the choice $I=\emptyset\subset \mathcal{N}_{\mathrm{e}}$ we set  
            \[
                (\varphi_{\emptyset}^*f_{\Gamma_{\emptyset}})(\mathbf{t}_{\mathcal{N}}) := \prod_{n\in\mathcal{N}_{\mathrm{e}}}\frac{1}{1-x_n},        
            \] 
            where $x_n=\mathbf{t}_{\mathcal{N}}^{X_n}$ as before.

            Then we get the following inclusion-exclusion-type surgery formula for the reduced functions in variables $\mathbf{t}_{\mathcal{N}}$.
       
        \begin{Thm}\label{Thm:zetaSurgery:ZHS}
        	\begin{equation}\label{eq:surgfunc}
        		\sum_{I\subseteq \mathcal{N}_{\mathrm{e}}}(-1)^{|I|}(\varphi_I^*f_{\Gamma_I})(\mathbf{t}_{\mathcal{N}}) = \prod_{n\in\mathcal{N}}P_n(x_n).
        	\end{equation}
        \end{Thm}
        \begin{proof} If $A_n, \, B_n$ ($n\in  \mathcal{N}_{\mathrm{e}}$) are formal variables then $\prod_{n\in\mathcal{N}_{\mathrm{e}}} (A_n+B_n)=\sum 
        _{I\subseteq \mathcal{N}_{\mathrm{e}}}\prod_{n\in I}A_n\prod_{n\not\in I}B_n$. Hence, for formal variables
        	  $V_n,W_n$ ($n\in \mathcal{N}_{\mathrm{e}}$) and $Z$ we get
        	\[
        		Z\sum_{I\subseteq \mathcal{N}_{\mathrm{e}}}(-1)^{|I|}\prod_{n\in I}(V_n-W_n)\prod_{n\in \mathcal{N}_{\mathrm{e}}\setminus I}V_n = Z\prod_{n\in\mathcal{N}_{\mathrm{e}}}W_n.
        	\]
        	 Then, substitute $Z := \prod_{n\in\mathcal{N}_{\mathrm{n}}}P_n(x_n)$, $V_n = 1/(1-x_n)$, $W_n = P_n(x_n)$, and use the expression
        	\[
        		(\varphi_I^*f_{\Gamma_I})(\mathbf{t}_{\mathcal{N}}) = \prod_{n\in\mathcal{N}_{\mathrm{n}}}P_n(x_n)\prod_{n\in I}\left(\frac{1}{1-x_n}-P_n(x_n)\right)\prod_{n\in \mathcal{N}_{\mathrm{e}}\setminus I}\frac{1}{1-x_n}, 
        	\]
           which is obtained from (\ref{eq:reducedzeta_xn}), (\ref{eq:Pn_EndNode}), (\ref{eq:Pn_NodeNode}).
        \end{proof}
        \begin{Cor}
            In particular, (\ref{eq:surgfunc}) implies the following recursion type expression:
            $$f_\Gamma(\mathbf{t}_\mathcal{N})=\sum_{I\subsetneq \mathcal{N}_{\mathrm{e}}}(-1)^{|\mathcal{N}_{\mathrm{e}}\setminus I|+1}(\varphi_I^*f_{\Gamma_I})(\mathbf{t}_{\mathcal{N}})+ (-1)^{|\mathcal{N}_{\mathrm{e}}|}\prod_{n\in\mathcal{N}}P_n(x_n).$$
        \end{Cor}
        \begin{example}
        Let $\Gamma$ be the graph in Figure \ref{fig:plumbing1}. Using the calculations of Example \ref{example:part:1}, it remains to consider the sets $I=\{E_3\}$ and $I=\emptyset$.  For $I = \{E_3\}$ the splice diagram and the corresponding plumbing tree can be seen in Figure \ref{fig:plumbing3}.
        \begin{figure}[h!]
            \centering
            \begin{tikzpicture}[roundnode/.style={circle, draw=black, fill=black, very thick, inner sep = 1},line width=1,x=\graphscale,y=\graphscale]
        	\node[roundnode] (1) {};
        	\path (1)++(1.5,0) node[roundnode] (2) {};
        	\path (1)++(3,0) node[roundnode] (3) {};
        	
        	\path (1)++(120:1.5) node (10) {};
        	\path (1)++(-120:1.5) node (11) {};
        	\path (2)++(90:1.5) node[roundnode] (20) {};
        	\path (3)++(60:1.5) node[roundnode] (30) {};
        	\path (3)++(-60:1.5) node[roundnode] (31) {};
        	
        	\draw  (1)--(3) (2)--(20) (30)--(3)--(31);
        	\begin{small}
        		\renewcommand{\r}{0.5}
        		\path (2)++(20:\r) node {$5$};
        		\path (2)++(160:\r) node {$7$};
        		\path (2)++(70:\r) node {$2$};	
        		\path (3)++(80:\r) node {$2$};
        		\path (3)++(-80:\r) node {$3$};
        		\path (3)++(160:\r) node {$19$};	
        	\end{small}	
        \end{tikzpicture}
        \begin{tikzpicture}[roundnode/.style={circle, draw=black, fill=black, very thick, inner sep = 1},line width=1,x=\graphscale,y=\graphscale]
        	\node(1) {};
        	\path (1)++(3,0) node[roundnode] (2) {};
        	\foreach \i/\j in {1/4,2/5,3/2,4/6,5/3}{
        		\path (1)++(\i,0) node[roundnode] (\j) {};
        	}
        	\path (1)++(120:1) node (8) {};
        	\path (8)++(180:1) node (7) {};
        	
        	\path (1)++(-120:1) node (11) {};
        	\foreach \i/\j in {1/10,2/9}{
        		\path (11)++(180:\i) node (\j) {};
        	}
        	\path (2)++(0,1) node[roundnode] (12) {};
        	\path (3)++(60:1) node[roundnode] (13) {};
        	\path (3)++(-60:1) node[roundnode] (14) {};
        	
        	\draw (4)--(3) (2)--(12) (13)--(3)--(14);
        	\begin{small}
        		\renewcommand{\r}{0.3}
        		\foreach \w/\i in {2/13,3/14,1/3}{
        			\path (\i)++(0:\r) node {$-\w$};
        		}
        		\foreach \w/\i in {2/4,4/5,1/2,11/6}{
        			\path (\i)++(-90:\r) node {$-\w$};
        		}
        		\foreach \w/\i in {2/12}{
        			\path (\i)++(90:\r) node {$-\w$};
        		}
        		\path (2)++({\r/tan(60)},\r) node {$E_2$};
        		\path (3)++({\r/tan(120)},\r) node {$E_3$};
        		\path (5)++(0,\r) node {$E_4$};
        		\path (6)++(0,\r) node {$E_5$};
        	\end{small}
        \end{tikzpicture}
            \caption{}
            \label{fig:plumbing3}
        \end{figure}
        By the formal variable change $t_4\mapsto t_1^{13}$, $t_2\mapsto t_2t_1^{-2}$, $t_3\mapsto t_3$ yields
        \[
        (\varphi_{\{E_3\}}^*f_{\Gamma_{\{E_3\}}})(t_1,t_2,t_3) = \frac{1}{1-x_1}\cdot\frac{1-x_2^2}{1-x_2}\cdot \frac{1-x_3^6}{(1-x_3^2)(1-x_3^3)}.
        \]
        For $I = \emptyset$ (see Figure \ref{fig:plumbing4}) 
        \begin{figure}[h!]
            \centering
            \begin{tikzpicture}[roundnode/.style={circle, draw=black, fill=black, very thick, inner sep = 1},line width=1,x=\graphscale,y=\graphscale]
        	\node[roundnode] (1) {};
        	\path (1)++(1.5,0) node[roundnode] (2) {};
        	\path (1)++(3,0) node[roundnode] (3) {};
        	\path (1)++(120:1.5) node (10) {};
        	\path (1)++(0,-0.85) node (11) {};
        	\path (2)++(90:1.5) node[roundnode] (20) {};
        	\path (3)++(60:1.5) node (30) {};
        	\draw  (1)--(3) (2)--(20);
        	\begin{small}
        		\renewcommand{\r}{0.5}
        		\path (2)++(20:\r) node {$5$};
        		\path (2)++(160:\r) node {$7$};
        		\path (2)++(70:\r) node {$2$};		
        	\end{small}
        \end{tikzpicture}
        \begin{tikzpicture}[roundnode/.style={circle, draw=black, fill=black, very thick, inner sep = 1},line width=1,x=\graphscale,y=\graphscale]
        	\node(1) {};
        	\path (1)++(3,0) node[roundnode] (2) {};
        	\foreach \i/\j in {1/4,2/5,3/2,4/6}{
        		\path (1)++(\i,0) node[roundnode] (\j) {};
        	}
        	\path (1)++(5,0) node (3) {};
        	\path (1)++(120:1) node (8) {};
        	\path (8)++(180:1) node (7) {};
        	
        	\path (1)++(-120:1) node (11) {};
        	\foreach \i/\j in {1/10,2/9}{
        		\path (11)++(180:\i) node (\j) {};
        	}
        	\path (2)++(0,1) node[roundnode] (12) {};
        	\path (3)++(60:1) node (13) {};
        	\path (3)++(-60:1) node (14) {};
        	
        	\draw (4)--(6) (2)--(12);
        	\begin{small}
        		\renewcommand{\r}{0.3}
        		\foreach \w/\i in {2/4,4/5,1/2,5/6}{
        			\path (\i)++(-90:\r) node {$-\w$};
        		}
        		\foreach \w/\i in {2/12}{
        			\path (\i)++(90:\r) node {$-\w$};
        		}
        		\path (2)++({\r/tan(60)},\r) node {$E_2$};
        		\path (5)++(0,\r) node {$E_4$};
        		\path (6)++(0,\r) node {$E_5$};
        	\end{small}
        \end{tikzpicture}
            \caption{}
            \label{fig:plumbing4}
        \end{figure}
        the formal variable change is given by $t_4\mapsto t_1^{13}$, $t_2\mapsto t_2t_1^{-2}t_3^{-1}$, $t_5\mapsto t_3^{11}$, which gives
        \[
        (\varphi_{\emptyset}^*f_{\Gamma_{\emptyset}})(t_1,t_2,t_3) = \frac{1}{1-x_1}\cdot\frac{1-x_2^2}{1-x_2}\cdot \frac{1}{1-x_3}.
        \]
        Combining all parts we have
        \[
            f_{\Gamma} -  (\varphi_{\{E_1\}}^*f_{\Gamma_{\{E_1\}}})- (\varphi_{\{E_3\}}^*f_{\Gamma_{\{E_3\}}})+ (\varphi_{\emptyset}^*f_{\Gamma_{\emptyset}}) = (x_1+x_1^2+x_1^5)\cdot(1+x_2)\cdot x_3 = P_1(x_1)P_2(x_2)P_3(x_3) 
        \]
        where $\varphi_{\{E_1,E_2,E_3\}}^*f_{\{E_1,E_2,E_3\}} = f_{\Gamma}$ and $P_{i}(t)$ ($i=1,2,3$) are the polynomials given in Example \ref{example:part:0}.
         \end{example}
    \section{Surgery formula for the polynomial part}

    \subsection{Polynomial part of the topological Poincar\'e series \cite{LNN19}}\label{ss:polpart}

    For the general definition of the polynomial part of a multivariable series (with certain regularity properties) see 
     \cite{LNN19,LNN20,LN14Erhart,LSz17}. Here, for our present situation
     (the 3-manifold is an integral homology sphere and the series is the Poincar\'e series)
    we adopt a simpler definition.
    
    For any two cycles $l_1, l_2\in L(\Gamma)$ we define $l_1\succ l_2$ if 
    $m_v(l_1)>m_v(l_2)$ for every $v$. 
    
        Let us consider first the truncation of  $Z(\mathbf{t})$ by considering only the monomials with exponents `no larger' than $Z_K-E$. More precisely, let $\operatorname{Supp}'(\Gamma)$ be the set of exponents $l$ of the form
        \[
            l=\sum_{e\in\mathcal{E}}a_eE_e^* + \sum_{v\in\mathcal{V}\setminus\mathcal{E}}a_vE_v^*\nsucc Z_K-E,
        \]
        where $0\le a_e$ for $e\in\mathcal{E}$ and $0\le a_v\le \delta_v-2$ if $v\in\mathcal{V}\setminus\mathcal{E}$, cf. \ref{ss:2.3}.  Note that $\Supp'(\Gamma)$ is a finite subset of the Lipman cone $\mathcal{S}=\Z_{\ge0}\langle E_v^*\rangle$. 
        
        \begin{Def}
            The polynomial part of the topological Poincar\'e series $Z(\mathbf{t})=\sum_{l}z(l)\mathbf{t}^l$ is defined as
            \[
                 \operatorname{Pol}_{\Gamma}(\mathbf{t}) := \mathbf{t}^{Z_K-E}\cdot \sum_{l\in\operatorname{Supp}'(\Gamma)} z(l)\mathbf{t}^{-l}.
            \]
        \end{Def}
        \noindent
        Note that the monomials of $\operatorname{Pol}_{\Gamma}(\mathbf{t})$
        are exactly those monomials of $ \mathbf{t}^{Z_K-E}\cdot  Z(\mathbf{t}^{-1})=
        \mathbf{t}^{Z_K-E}\cdot  Z(t_1^{-1}, \ldots, t_{|\mathcal{V}|}^{-1})$ which are
        $\nprec 0$.

        Similarly to Section \ref{sss:redseries}, we can define the reduced polynomial part for any $I\subseteq\mathcal{V}$ as $\operatorname{Pol}_{\Gamma}(\mathbf{t}_I):=\operatorname{Pol}_{\Gamma}(\mathbf{t})_{t_v\mapsto 1,v\notin I}$.

        \subsection{Polynomial part of formal Laurent series}
        Let $\mathcal{W}$ be a finite index set, 
        and consider  multi-variable formal Laurent series $S(\mathbf{t}_{\mathcal{W}})=\sum_{l}s(l)\mathbf{t}_{\mathcal{W}}^l\in\Z[[\mathbf{t}_{\mathcal{W}}^{-1}]][\mathbf{t}_{\mathcal{W}}]$, where $\mathbf{t}_{\mathcal{W}}^l=\prod_{w\in \mathcal{W}} t_w^{l_w}$ and $l=\sum_{w\in \mathcal{W}}l_wE_w\in \mathbb{Z}\langle E_w\mid w\in \mathcal{W}\rangle$
        is an element of the lattice associated with $\mathcal{W}$.
        For any  non-empty  $J\subseteq \mathcal{W}$,
        the \emph{polynomial part} of $S(\mathbf{t}_\mathcal{W})$ with respect to the variables $\mathbf{t}_{J}$ is defined as
        \[   \operatorname{PolPart}_J(S(\mathbf{t}_\mathcal{W}))=
            (\operatorname{PolPart}_JS)(\mathbf{t}_\mathcal{W})
            := \sum_{\restr{l}{J}\nprec0}s(l)\mathbf{t}_\mathcal{W}^l.
        \]
         In the case $J = \mathcal{W}$, for  simplicity we will omit the big index set from the notations and write e.g. $\operatorname{PolPart}S:=\operatorname{PolPart}_{\mathcal{W}}S$.
        
        We emphasize that the operator  $\operatorname{PolPart}_J$
        preserves the variables of $S$. However, 
        similarly to \ref{ss:polpart}, this  polynomial part  
        $\operatorname{PolPart}_JS(\mathbf{t}_{\mathcal{W}})$
        can be reduced to the variables $\mathbf{t}_{J'}$ for any $J'\subseteq
        \mathcal{W}$, if
         the  finiteness property with respect to $J'$ is satisfied, that is,   whenever
         the next set
         is finite
        \begin{equation}\label{eq:finiteprop}
        \{l \mid s(l)\neq 0 \ \mbox{and} \ \restr{l}{J'} \ \mbox{is fixed}\}.
        \end{equation}
        In this case we define  
        \[
            (\operatorname{PolPart}_JS)|_{J'}(\mathbf{t}_{J'}):=\big(\operatorname{PolPart}_{J}(S(\mathbf{t}_{\mathcal{W}}))\big)_{t_i\mapsto1,i\notin J'}.
        \]
       
        \begin{Rem}\label{Rem:IterPolPart}
            If $\emptyset\neq J_1\subseteq J_2\subseteq \mathcal{W}$ and $S$ has the finiteness property with respect to both $J_1$ and $J_2$, then 
            the `restriction operator' does not commute with the `polynomial part' operator:
            the restriction to $\mathbf{t}_{J_1}$ of 
            $\operatorname{PolPart}_{J_2}S$ usually does not equal to the  $\operatorname{PolPart}_{J_1}$
            of the restriction of $S$ to $\mathbf{t}_{J_1}$, that is, 
            $(\operatorname{PolPart}_{J_2}S)|_{J_1} \not= \operatorname{PolPart}_{J_1}(S|_{J_1})=(\operatorname{PolPart}_{J_1}S)|_{J_1}$. 

            On the other hand, 
            $\operatorname{PolPart}_{J_1}\big((\operatorname{PolPart}_{J_2}S)|_{J_1})\big)=\operatorname{PolPart}_{J_1}(S|_{J_1}) $. 
            (This will be relevant in the proof of Theorem \ref{Thm:oRedPolSurgery}.)
        \end{Rem}
        Now, we consider the setting of the present article and let $\Gamma$ be a negative definite tree with $\det\Gamma=1$ and $\mathcal{V}$ its set of vertices. Consider the associated rational function $f$ defined in (\ref{eq:zeta:def}). 
        Then 
         $f$ can be considered as an element in $\Z[[\mathbf{t}^{-1}]][\mathbf{t}]$
         (where $\mathbf{t}=\mathbf{t}_{\mathcal{V}}$)   
         in two different ways. Either we write as $\mathbf{t}^{Z_K-E}\cdot Z(\mathbf{t}^{-1})$, or, we expand the factors of type $1/(1-\mathbf {t}^a)$ of $f(\mathbf {t}) $  as 
        \[
            \frac{1}{1-\mathbf{t}^a} = -\frac{1}{\mathbf{t}^a}-\frac{1}{\mathbf{t}^{2a}}-\frac{1}{\mathbf{t}^{3a}}-\ldots \in \Z[[\mathbf{t}^{-1}]][\mathbf{t}] \ \ 
           \  ( a\succ0).
        \]
        Since the finiteness property (\ref{eq:finiteprop}) is satisfied,  $(\operatorname{PolPart}_Jf)(\mathbf{t}) $ is defined for any $J\subseteq \mathcal{V}$. 
        Note  that one has  
        $(\operatorname{PolPart}f)(\mathbf {t})=\operatorname{Pol}_{\Gamma}(\mathbf{t})$, where $\operatorname{Pol}_{\Gamma}(\mathbf{t})$ is the polynomial part of the topological Poincar\'e series introduced in \ref{ss:polpart}. This induces the same identity of reduced polynomials
        \begin{equation}\label{eq:pol1}
        \operatorname{Pol}_{\Gamma}(\mathbf{t}_J)=(\operatorname{PolPart}f)|_{J}(\mathbf{t}_{J})
        \end{equation}
        (where $\operatorname{Pol}_{\Gamma}(\mathbf{t}_{J})$ is defined in \ref{ss:polpart}). On the other hand, by definition, we also get 
        \begin{equation}\label{eq:pol2}
        (\operatorname{PolPart}_{J}f)|_{J}(\mathbf{t}_{J}) = (\operatorname{PolPart}_J(f|_{J}))(\mathbf{t}_J).
        \end{equation}
        Despite the general non-identity observed in Remark \ref{Rem:IterPolPart}, we still have the following:
        \begin{Lem}\label{lem:polpart}
            If we further assume that  $\mathcal{N}\subseteq J$ then we have $\operatorname{PolPart}f = \operatorname{PolPart}_{J}f$.
        \end{Lem}
        In particular, this, together with (\ref{eq:pol1}) and (\ref{eq:pol2}), implies that
        \begin{equation}\label{eq:polred}
            \operatorname{Pol}_{\Gamma}(\mathbf{t}_J) = (\operatorname{PolPart}_J(f|_{J}))(\mathbf{t}_J) \ \ \mbox{for any $\mathcal{N}\subseteq J$}.
        \end{equation}
        Lemma \ref{lem:polpart} follows from the fact that  $l\nprec0$ if and only if $\restr{l}{J}\nprec0$ for any  $l\in\Z\langle E_v^*\mid v\in\mathcal{N}\rangle\oplus \Z_{<0}\langle E_v^*\mid v\in\mathcal{E}\rangle$.
         This will follow from the next Lemma \ref{Lem:singDistribution}.
        \begin{Lem}\label{Lem:singDistribution}
             Let $J\sqcup J^* = \mathcal{V}$ be a partition of $\mathcal{V}$  such that $J,J^*$ are both non-empty. Then $\restr{l}{J}\prec0$ if and only if $l\prec0$ for any $l\in\Q\langle E_v^*\mid v\in J\rangle\oplus\Q_{\le0}\langle E_v^*\mid v\in J^*\rangle$.
        \end{Lem}
        \begin{proof} We prove that  $\restr{l}{J}\prec0$ implies  $l\prec0$. We may assume that $l\in L'(\Gamma)$, otherwise we can multiply $l$ with a positive integer $N$ such that $N\cdot l\in L'(\Gamma)$.
            Let $\widetilde{X}$ be the plumbed 4-manifold associated with $\Gamma$ and for any $v\in \mathcal{V}$ let $C_v$ be a generic disc
            (with boundary in $\partial \widetilde{X}$) transversal to $E_v$, oriented in such a way that $C_v+E^*_v=0$ in $H_2(\widetilde{X},\partial \widetilde{X},\mathbb{Z})$.  Note that this lattice can be identified with $L'$
            by Lefschetz duality. Write $l$ as $\sum_{v\in J}a_vE^*_v+\sum_{v\in J^*} c_vE^*_v$ with $c_v\leq 0$. 
            Then $l=-\sum_{v\in J}a_vC_v-\sum_{v\in J^*} c_vC_v$ in $H_2(\widetilde{X},\partial \widetilde{X},\mathbb{Z})$.
            Choose an arbitrary $u\in J^*$. Then $(C_v,E_u)\geq 0$ for $v\in J^*$,
            $(C_v,E_u)=0$ for $v\in J$, and $(l|_{E_v}, E_u)\leq 0$  for any $v\in J$ since 
            $\restr{l}{J}\prec0$. In fact, if $u$ is  a neighbour of some $v\in J$ then 
             $(l|_{E_v}, E_u)< 0$. Therefore, $(l|_{J^*}, E_u)\geq 0$ for any 
             $u\in J^*$, and at least one of the inequalities is strict. 
             Then $\restr{l}{J^*}\prec0$ by the negative definiteness of the intersection form. 
        \end{proof}
        \subsection{Polynomial part after variable change}
        From now on, we use again all the notations of section \ref{s:surgtopseries}. The goal of this Section is to obtain a surgery type formula (analog to Theorem \ref{Thm:zetaSurgery:ZHS}) for the polynomial $\operatorname{Pol}_{\Gamma}(\mathbf{t}_\mathcal{N})$ using the polynomials $\operatorname{Pol}_{\Gamma_{I}}(\mathbf{t}_{J_I})$ associated with the induced trees $\Gamma_I$ (see Section \ref{ss:inducedgraphs}) for any $I\subseteq \mathcal{N}_{\mathrm{e}}$. 
        For the case $|\mathcal{N}|=2$ (equivalently $\mathcal{N}_{\mathrm{n}}=\emptyset$), we denote $\varphi_{\emptyset}^*\operatorname{Pol}_{\Gamma_{\emptyset}} := 0$, 
        which is equal to $\operatorname{PolPart}\big((\varphi_{\emptyset}^*f_{\Gamma_{\emptyset}})(\mathbf{t}_{\mathcal{N}})\big)$.

        First we show that the change of variables map `preserves signs' 
        (at least on exponents in the Laurent series expansion).
        \begin{Lem}\label{lem:varchangeSign}
            Let $I\subseteq\mathcal{N}_{\mathrm{e}}$ and $l$ be an exponent in the Laurent series expansion of $f_{\Gamma_I}(\mathbf{t}_{J_I})$. Then $\varphi_I(l)\prec0$ if and only if $l\prec0$.
        \end{Lem}
        \begin{proof}
            By Lemma \ref{lem:induced_expr} the cycle  $l$ has the form 
            \[
                l =  \sum_{n\in\mathcal{N}_{\mathrm{n}}\cup I}\Big(a_n\omega_n-\sum_{\substack{e\in\mathcal{E}\\ n-_be}}a_e\omega_{e}\Big)Y_n - \sum_{n\in\mathcal{N}_{\mathrm{e}}\setminus I}a_nY_n,
            \]
            where $a_v\ge0$. Consider the (unrestricted) cycle $l_1$ given by
            \[
                l_1= \sum_{n\in\mathcal{N}_{\mathrm{n}}\cup I}\Big(a_nE_{n,\Gamma_I}^*-\sum_{\substack{e\in\mathcal{E}\\ n-_be}}a_eE_{e,\Gamma_I}^*\Big) - \sum_{n\in\mathcal{N}_{\mathrm{e}}\setminus I}a_nE_{e_n,\Gamma_I}^*
                \in L(\Gamma_I).
            \]
            By Lemma \ref{Lem:singDistribution} we have $l_1\prec0$ if and only if $\restr{l_1}{\mathcal{N}_{\mathrm{n}}\cup I}\prec0$.
            The equality $\restr{l_1}{J_I} = l$ implies that $l\prec0$ if and only if $\restr{l}{\mathcal{N}_{\mathrm{n}}\cup I}\prec0$.
            In a similar manner, we consider the cycle $l_2$ given by
            \[
                l_2 = \sum_{n\in\mathcal{N}_{\mathrm{n}}\cup I}\Big(a_nE_{n,\Gamma}^*-\sum_{\substack{e\in\mathcal{E}\\ n-_be}}a_eE_{e,\Gamma}^*\Big) - \sum_{n\in\mathcal{N}_{\mathrm{e}}\setminus I}\frac{a_n}{\omega_n}E_{n,\Gamma}^*
                \in L(\Gamma)\otimes\Q.
            \]
            Then $\restr{l_2}{\mathcal{N}} = \varphi_I(l)$. By Lemma \ref{Lem:singDistribution} we have that $l_2\prec0$ if and only if $\restr{l_2}{\mathcal{N}_{\mathrm{n}}\cup I}\prec0$.
            Thus $\varphi_I(l)\prec0$ if and only if $\restr{\varphi_I(l)}{\mathcal{N}_{\mathrm{n}}\cup I}\prec0$.
            Then by Lemma \ref{lem:varRestr} the proof is complete.
        \end{proof}
        \begin{Cor}\label{Lem:PolPartZHS}
            The $\operatorname{PolPart}$ and the change of variable operator $\varphi_I^*$ commute. More precisely
            For any $I\subseteq\mathcal{N}_{\mathrm{e}}$ the following equality holds
            \[
                \operatorname{PolPart}\big((\varphi_I^*f_{\Gamma_I})(\mathbf{t}_{\mathcal{N}})\big) = \varphi_I^*\big(\operatorname{Pol}_{\Gamma_I}(\mathbf{t}_{J_I})\big).
            \]
        \end{Cor}
        \noindent Note that on the left hand side of the identity one applies the $\operatorname{PolPart}:=\operatorname{PolPart}_{\mathcal{N}}$ operator with respect to the full index set $\mathcal{N}$, hence the obtained polynomial has variables $\mathbf{t}_{\mathcal{N}}$. Similarly, on the right hand the change of variable operator provides a polynomial in $\mathbf{t}_{\mathcal{N}}$.
        \begin{proof}
            By (\ref{eq:polred}), one needs to show that $\operatorname{PolPart}\big((\varphi_I^*f_{\Gamma_I})(\mathbf{t}_{\mathcal{N}})\big)= \varphi_I^*\big(\operatorname{PolPart}_{\color{blue}J_I}(f_{\Gamma_I}|_{J_I})(\mathbf{t}_{J_I})\big)$ for $I\subsetneq \mathcal{N}_{\mathrm{e}}$.  Suppose the Laurent expansion is given by
            \[
                (f_{\Gamma_I}|_{J_I})(\mathbf{t}_{J_I}) = \sum w(l)\mathbf{t}_{J_I}^{l} \in\Z[[\mathbf{t}_{J_I}^{-1}]][\mathbf{t}_{J_I}].
            \]
            Then the Laurent expansion of $\varphi_I^*f_{\Gamma_I}$ is given by
            \[
                 (\varphi_I^*f_{\Gamma_I})(\mathbf{t}_\mathcal{N}) = \sum w(l)\mathbf{t}_{\mathcal{N}}^{\varphi_I(l)} \in\Z[[\mathbf{t}_{\mathcal{N}}^{-1}]][\mathbf{t}_{\mathcal{N}}].
            \]
            Truncating the corresponding Laurent series expansions, we can write
        	\[
        		\operatorname{PolPart}_{J_I}(f_{\Gamma_I}|_{J_I})(\mathbf{t}_{J_I}) = \sum_{l\nprec0}w(l)\mathbf{t}_{J_I}^{l} \quad \mbox{and} \quad \operatorname{PolPart}(\varphi_I^*f_{\Gamma_I})(\mathbf{t}_{\mathcal{N}}) = \sum_{\varphi_I(l)\nprec0}w(l)\mathbf{t}_{\mathcal{N}}^{\varphi_I(l)}.
        	\]
            In the second identity, due to Lemma \ref{lem:varchangeSign} we can replace $\varphi_I(l)\nprec0$ with $l\nprec0$. Then applying $\varphi_I^*$ to the first identity completes the proof.
        \end{proof}
        Finally, by applying the $\operatorname{PolPart}$ operator to the formula in Theorem \ref{Thm:zetaSurgery:ZHS}, and then Corollary \ref{Lem:PolPartZHS} deduces the following surgery formula for the polynomial parts. 
        \begin{Thm}\label{Thm:PolSurgery}
        	\begin{equation}\label{eq:surgerypol}
        			(-1)^{|\mathcal{N}_{\mathrm{e}}|}\operatorname{Pol}_{\Gamma}(\mathbf{t}_{\mathcal{N}}) + \sum_{I\subsetneq \mathcal{N}_{\mathrm{e}}}(-1)^{|I|}(\varphi_I^*\operatorname{Pol}_{\Gamma_I})(\mathbf{t}_{\mathcal{N}}) = \prod_{n\in\mathcal{N}}P_n(x_n) =\prod_{n\in\mathcal{N}}P_n(\mathbf{t}_{\mathcal{N}}^{X_n}).
        	\end{equation}
        \end{Thm}

        \noindent Notice that for $I\subsetneq\mathcal{N}_{\mathrm{e}}$ the reduced polynomial part $\operatorname{Pol}_{\Gamma_I}(\mathbf{t}_{J_I})$ is not reduced to the nodes of the respective graphs, since extra variables from some legs are also included. This somehow corresponds to the `gluing' and ensures a `recursive connection' of the reduced polynomial parts appearing in the above formula.  

        \begin{example}
            Let us consider the splice diagram in Figure \ref{fig:plumbing3}. In this case $\mathcal{N} = \mathcal{N}_{\mathrm{e}} = \{E_2,E_3\}$ and the vertex obtained from the $0$-absorption is $E_5$.
            We only consider the cases $I = \{E_2\}$ or $I = \{E_3\}$, as by definition $(\varphi_{\emptyset}^*\operatorname{Pol_{\Gamma_{\emptyset}}})(t_2,t_3) = 0$.

            For $I = \{E_2\}$ the splice diagram is shown by the left hand side of Figure \ref{fig:plumbing4}, the corresponding change of variable map $\varphi_{\{E_2\}}$ is given by
            \[
                \varphi_{\{E_2\}}(E_2) = E_2-E_3\quad \mbox{and}\quad \varphi_{\{E_2\}}(E_5) = 11E_3.
            \]
            Moreover, one can calculate that $(\varphi_{\{E_2\}}^*\operatorname{Pol}_{\Gamma_{\{E_2\}}})(t_2,t_3) = t_2t_3^{-1} + t_2^{11}t_3^{11}$.

            For $I = \{E_3\}$ the splice diagram is given by Figure \ref{fig:plumbing5}, the corresponding change of variable map $\varphi_{\{E_3\}}$ is as follows 
            \[
                \varphi_{\{E_3\}}(E_3) = E_3-E_2,\quad \mbox{and}\quad \varphi_{\{E_3\}}(E_5) = 11E_2.  
            \]
            Hence, we deduce that $(\varphi_{\{E_3\}}^*\operatorname{Pol}_{\Gamma_{\{E_3\}}})(t_2,t_3) = t_2^{-1}t_3 + t_2^{4}t_3^{7} + t_2^{9}t_3^{13}$. The polynomials $P_2,P_3$ are expressed by
            \[
                P_2(t)=\frac{1}{1-t}-\frac{1-t^{14}}{(1-t^2)(1-t^7)}= t+t^3+t^5,\quad P_3(t)  = \frac{1}{1-t}-\frac{1-t^6}{(1-t^2)(1-t^3)}= t.
            \]
            Thus, we get the polynomial part reduced to the node variables as 
            \begin{align*}
                \operatorname{Pol}_{\Gamma}(t_2,t_3)&= (\varphi_{\{E_2\}}^*\operatorname{Pol}_{\Gamma_{\{E_2\}}})(t_2,t_3)+(\varphi_{\{E_3\}}^*\operatorname{Pol}_{\Gamma_{\{E_3\}}})(t_2,t_3)+P_2(t_2^{5}t_3^{6})P_3(t_2^{14}t_3^{19})\\ 
                &=t_2t_3^{-1} + t_2^{11}t_3^{11} +t_2^{-1}t_3 + t_2^{4}t_3^{7} + t_2^{9}t_3^{13}+ t_2^{19}t_3^{25} + t_2^{29}t_3^{37} + t_2^{39}t_3^{49}.
            \end{align*}
            \begin{figure}[h!]
                \centering
                \begin{tikzpicture}[roundnode/.style={circle, draw=black, fill=black, very thick, inner sep = 1},line width=1,x=\graphscale,y=\graphscale]
        	\node (1) {};
        	\path (1)++(1.5,0) node[roundnode] (2) {};
        	\path (1)++(3,0) node[roundnode] (3) {};
        	
        	\path (1)++(120:1.5) node (10) {};
        	\path (1)++(-120:1.5) node (11) {};
        	\path (2)++(90:1.5) node (20) {};
        	\path (3)++(60:1.5) node[roundnode] (30) {};
        	\path (3)++(-60:1.5) node[roundnode] (31) {};
        	
        	\draw  (2)--(3) (30)--(3)--(31);
        	\begin{small}
        		\renewcommand{\r}{0.5}
        		\path (3)++(80:\r) node {$2$};
        		\path (3)++(-80:\r) node {$3$};
        		\path (3)++(160:\r) node {$19$};	
        	\end{small}	
        \end{tikzpicture}
        \begin{tikzpicture}[roundnode/.style={circle, draw=black, fill=black, very thick, inner sep = 1},line width=1,x=\graphscale,y=\graphscale]
        	\node(1) {};
        	\path (1)++(3,0) node (2) {};
        	\foreach \i/\j in {2/5,3/2,4/6,5/3}{
        		\path (1)++(\i,0) node[roundnode] (\j) {};
        	}
        	\path (1)++(120:1) node (8) {};
        	\path (8)++(180:1) node (7) {};
        	
        	\path (1)++(-120:1) node (11) {};
        	\foreach \i/\j in {1/10,2/9}{
        		\path (11)++(180:\i) node (\j) {};
        	}
        	\path (3)++(60:1) node[roundnode] (13) {};
        	\path (3)++(-60:1) node[roundnode] (14) {};
        	
        	\draw (5)--(3) (13)--(3)--(14);
        	\begin{small}
        		\renewcommand{\r}{0.3}
        		\foreach \w/\i in {2/13,3/14,1/3}{
        			\path (\i)++(0:\r) node {$-\w$};
        		}
        		\foreach \w/\i in {2/5,2/2,7/6}{
        			\path (\i)++(-90:\r) node {$-\w$};
        		}
        		\path (3)++({\r/tan(120)},\r) node {$E_3$};
        		\path (6)++(0,\r) node {$E_5$};
        	\end{small}
        \end{tikzpicture}
                \caption{}
                \label{fig:plumbing5}
            \end{figure}
        \end{example}

        \section{Surgery formulae with over-reduced variables}\label{s:OverRedSurg}
        \subsection{The new setup with reduced variables}
            Let $\mathfrak{S}$ be a  connected, not necessarily minimal splice diagram. 
            As above,  $\mathcal{N}(\mathfrak{S})$ denotes the set of nodes of $\mathfrak{S}$, $\mathcal{N}_{\mathrm{e}}(\mathfrak{S})$ the set of  end-nodes of $\mathfrak{S}$, 
            and $\mathcal{N}_{\mathrm{n}}(\mathfrak{S}):=\mathcal{N}(\mathfrak{S})\setminus \mathcal{N}_{\mathrm{e}}(\mathfrak{S})$. For simplicity,  we also write $|\mathfrak{S}|:=|\mathcal{N}(\mathfrak{S})|$.
            
            We say that $\mathfrak{S}'$ is a {\it complete}  sub-diagram of $\mathfrak{S}$ if it is a connected sub-diagram of $\mathfrak{S}$ with at least one node, and each node in $\mathfrak{S}'$ has the same number of adjacent edges in $\mathfrak{S}'$ as in $\mathfrak{S}$.
            Equivalently,
            $\mathfrak{S}'$ satisfies  $\mathcal{N}(\mathfrak{S}')\neq\emptyset$
            and it is obtained by replacing each edge $u\sim v$
            with $u\in\mathcal{N}(\mathfrak{S}')$, $v\in\mathcal{N}(\mathfrak{S})\setminus\mathcal{N}(\mathfrak{S}')$ 
             by two legs (one for $u$ and one for $v$)
            and then taking that connected sub-diagram  whose nodes are $\mathcal{N}(\mathfrak{S}')$.         
            We denote this relation by  $\mathfrak{S}'\subseteq\mathfrak{S}$. We observe  that any complete sub-diagram of $\mathfrak{S}$ is uniquely determined by the set of its nodes.

            If $\mathfrak{S}'$ is a sub-diagram (not necessarily connected) of $\mathfrak{S}$, 
            then we denote by $\mathfrak{S}\setminus \mathfrak{S}'$ the disjoint union of 
            complete sub-diagrams whose nodes are 
            $\mathcal{N}(\mathfrak{S})\setminus \mathcal{N}(\mathfrak{S}')$.
            Note that if $u\sim v$, $u\in \mathcal{N}(\mathfrak{S})$ and 
            $v\in \mathcal{N}(\mathfrak{S}')$ then this edge is replaced in 
             $\mathfrak{S}\setminus \mathfrak{S}'$  by a leg connecting to $u$.

            We say that $\mathfrak{S}'\subseteq\mathfrak{S}$ is a \emph{peel} of $\mathfrak{S}$ if $\mathcal{N}_{\mathrm{n}}(\mathfrak{S})\subseteq\mathcal{N}(\mathfrak{S}')\subseteq\mathcal{N}(\mathfrak{S})$. In other words, $\mathfrak{S}'$ is a peel of $\mathfrak{S}$ if it is obtained by cutting away some end-nodes of $\mathfrak{S}$ (but not all if $|\mathfrak{S}|=2$). The relation $\mathfrak{S}'\hookrightarrow\mathfrak{S}$ denotes that $\mathfrak{S}'$ is a peel of $\mathfrak{S}$.

            Consider the plumbing graphs $\Gamma':=\Gamma(\mathfrak{S}')$ and $\Gamma:=\Gamma(\mathfrak{S})$ corresponding to $\mathfrak{S}'$ and $\mathfrak{S}$ respectively.
            For $\mathfrak{S}'\hookrightarrow\mathfrak{S}$, denote by $\varphi_{\mathfrak{S}',\mathfrak{S}}\colon \Q\langle E_{u,\Gamma'}\mid u\in \mathcal{N}(\mathfrak{S}')\cup \Theta_{\mathfrak{S}',\mathfrak{S}}\rangle \to \Q\langle E_{u,\Gamma}\mid u\in\mathcal{N}(\mathfrak{S})\rangle$ the change of variable map `from $\Gamma(\mathfrak{S}')$ to $\Gamma(\mathfrak{S})$',  as given in Section \ref{ss:varchange}, where $\Theta_{\mathfrak{S}',\mathfrak{S}}:=\{\theta_u\mid u\in\mathcal{N}(\mathfrak{S})\setminus\mathcal{N}(\mathfrak{S}')\}$.
            \begin{Rem}\label{Rem:PolSurgeryReform}
                Theorem \ref{Thm:PolSurgery} can be restated as follows
                \[
                    \sum_{\mathfrak{S}'\hookrightarrow\mathfrak{S}}(-1)^{|\mathfrak{S}|-|\mathfrak{S}'|}(\varphi_{\mathfrak{S}',\mathfrak{S}}^*\operatorname{Pol}_{\Gamma(\mathfrak{S}')})(\mathbf{t}_{\mathcal{N}(\mathfrak{S})}) = (-1)^{|\mathcal{N}_{\mathrm{e}}(\mathfrak{S})|}\prod_{n\in\mathcal{N}_{\mathrm{e}}(\mathfrak{S})}\frac{1-\Delta_n(x_n)}{1-x_n}\cdot\prod_{n\in\mathcal{N}_{\mathrm{n}}(\mathfrak{S})}\Delta_n(x_n).
                \]
                The above formula with $\operatorname{Pol}_{\Gamma(\mathfrak{S}')}$ replaced by $f_{\Gamma(\mathfrak{S}')}$, would give a reformulation of Theorem \ref{Thm:zetaSurgery:ZHS}.
            \end{Rem}
            
            Our goal is to use this reformulation to give splicing formulas for a reduced amount of variables, corresponding to some complete sub-diagram of $\mathfrak{S}$ (with at least two nodes). The general path of arguments is analogous to the one presented for the surgery formulae (\ref{eq:surgfunc}), (\ref{eq:surgerypol}).

            Fix a complete sub-diagram $\mathfrak{S}_0\subseteq\mathfrak{S}$, with at least two nodes. The nodes of $\mathfrak{S}_0$ will correspond to the variables to which we write the splice formulae, of the reduced series associated with $Z$ (see \ref{ss:5.3}) and the corresponding  polynomial parts (see \ref{ss:5.4}). The variables corresponding to $\mathcal{N}(\mathfrak{S}_0)$
            are the `over-reduced variables', appearing in the title of the section.
             
              First, we identify the right hand side (the correction term). It turns out that it is a polynomial, hence it will appear identically in both 
              surgery formulae. Similarly as in Theorem \ref{Thm:PolSurgery} (or Remark 
              \ref{Rem:PolSurgeryReform}) it is related with  Alexander polynomials of certain links. 
            \subsection{Normalized Alexander polynomials}
                For each $n\in\mathcal{N}(\mathfrak{S}_0)$ let $\mathfrak{S}_0(n)\subseteq\mathfrak{S}$ denote the complete sub-diagram 
                component of $\mathfrak{S}\setminus(\mathcal{N}(\mathfrak{S}_0)\setminus\{n\})$ that contains $n$.  Set $V_n := \{u\in\mathcal{N}(\mathfrak{S}_0)\mid u\sim n\}$ and let $\mathcal{L}_n$ be the graph link corresponding to $\mathfrak{S}_0(n)$, where there are arrows placed on all legs indexed by $V_n$. Let us denote the multivariable Alexander polynomial 
                of $\mathcal{L}_n$ (in the 3-manifold associated with $\mathfrak{S}_0(n)$) by
                $\Delta_{\mathcal{L}_n}(\mathbf{t})$ (in variables 
                $\{t_v\}_{v\in V_n}$). 
                Since all the arrows sit on the legs of the very same node $n$, 
                $\Delta_{\mathcal{L}_n}(\mathbf{t})$ can be expressed in terms of a one-variable polynomial, what we will denote by $\Delta_{\mathfrak{S}_0(n)}(t)$, and 
                we call  `normalized Alexander polynomial'.

                Indeed, 
                if $n\in\mathcal{N}_{\mathrm{e}}(\mathfrak{S}_0)$, then $|V_n|=1$,  and we define $\Delta_{\mathfrak{S}_0(n)}(t)$ as  the Alexander polynomial of $\mathcal{L}_n$. 
                In the case $n\in\mathcal{N}_{\mathrm{n}}(\mathfrak{S}_0)$
                we define $\Delta_{\mathfrak{S}_0(n)}(t)$ as 
                 \[
                    \Delta_{\mathfrak{S}_0(n)}(t) = \Delta_{\mathcal{L}_n'}(t)
                  \cdot   \frac{(1-t^{\omega_nd_n})^{|V_n|-1}}{1-t}, \quad \mbox{where} \ \ 
                    d_n = \prod_{\substack{u\sim n\\ u\in \mathcal{N}(\mathfrak{S})\setminus V_n}}d_{n,u},
                \] 
                 and $\Delta_{\mathcal{L}_n'}(t)$ is the Alexander polynomial of the graph link $\mathcal{L}_n'$ (with one component) given by removing all but one arrow-headed leg from the diagram representing $\mathcal{L}_n$.
                 Note that we do not need to specify the decoration of the leg which supports the arrow of $\mathcal{L}_n'$, since  $\Delta_{\mathcal{L}_n'}(t)$, by its very definition, is independent of this decoration.

                With this notation, one has (using the definition of the Alexander polynomials (\ref{eq:alexander}))
                $$\Delta_{\mathcal{L}_n}(\mathbf{t}_{V_n}) = \Delta_{\mathfrak{S}_0(n)}(t)\big|_{t\mapsto \mathbf{t}_{V_n}^{l_n}}, \ \ \  
                \mbox{where} \ \  
                    l_n = \sum_{v\in V_n} \Big(\prod_{\substack{u\sim n\\u\in V_n\setminus\{v\}}}d_{n,u}\Big)E_v.
                $$
                Next, the one-variable $\Delta_{\mathfrak{S}_0(n)}(t)$
                can be expressed in terms of the polynomials $\Delta_v(t)$ associated with  $\Gamma(\mathfrak{S})$ in Section \ref{ss:normalizedAlexander}.
                This emphasizes once more our strategy: the polynomials $\Delta_u(t) $
                behave as elementary `atoms' in several formulae (see e.g. Remark \ref{rem:2.3} for $f_\Gamma(\mathbf{t}_{\mathcal{N}})$, or identity 
                (\ref{eq:reducedzeta_xn}) for the reduced zeta functions after change of variables, or the surgery formula Theorem \ref{Thm:PolSurgery}). 
                Writing  $\Delta_{\mathfrak{S}_0(n)}(t)$ in terms of these atoms is the key step in the proof of the new general surgery formulae.

                \begin{Lem}\label{Lem:AlexanderMultipleFromSingle}
                    Fix an arbitrary node $n\in\mathcal{N}(\mathfrak{S}_0)$.
                    Denote by $\mathcal{N}_{\mathrm{n}}^{\mathfrak{S}_0(n)} := \mathcal{N}_{\mathrm{n}}(\mathfrak{S})\cap\mathcal{N}(\mathfrak{S}_0(n))$, $\mathcal{N}_{\mathrm{e}}^{\mathfrak{S}_0(n)} := \mathcal{N}_{\mathrm{e}}(\mathfrak{S})\cap\mathcal{N}(\mathfrak{S}_0(n))$. For each $u\in\mathcal{N}(\mathfrak{S}_0(n))$ denote
                    \[
                        m_{n,u} = \frac{w(\mathfrak{S}[n,u])}{\omega_u\cdot d_{n}'}, \ \ \ 
                        \mbox{where} \ \ \ 
                         d_n' := w(\mathfrak{S}[n,n])/(d_n\omega_n).
                    \]
                    Then
                    \[
                        \Delta_{\mathfrak{S}_0(n)}(t) = \prod_{u\in\mathcal{N}_{\mathrm{n}}^{\mathfrak{S}_0(n)}}\Delta_u(t^{m_{n,u}})\cdot\prod_{u\in\mathcal{N}_{\mathrm{e}}^{\mathfrak{S}_0(n)}}\frac{\Delta_u(t^{m_{n,u}})}{1-t^{m_{n,u}}}\cdot\begin{cases}
                            (1-t)& \ \mbox{if} \ \ n\in\mathcal{N}_{\mathrm{e}}(\mathfrak{S}_0)\\
                            1& \ \mbox{if} \ \ n\in\mathcal{N}_{\mathrm{n}}(\mathfrak{S}_0)
                        \end{cases}\ .
                    \]
                \end{Lem}
                \begin{proof}
                    For each $u\in\mathcal{N}(\mathfrak{S}_0(n))$ denote the component of the node $u$ by
                    \[
                        C_u(t) := (1-t^{w(\mathfrak{S}[n,u])/d_n'})^{\delta_u-2}\prod_{\substack{e\sim u\\e\in\mathcal{E}(\mathfrak{S})}}(1-t^{w(\mathfrak{S}[n,e])/d_n'})^{-1}.
                    \]
                    Then we can rewrite $\Delta_{\mathfrak{S}_0(n)}(t)$ as
                    \[
                        \Delta_{\mathfrak{S}_0(n)}(t) = \prod_{u\in\mathcal{N}(\mathfrak{S}_0(n))}C_u(t)\cdot\begin{cases}
                            (1-t)& \ \mbox{if} \ \ n\in\mathcal{N}_{\mathrm{e}}(\mathfrak{S}_0)\\
                            1& \ \mbox{if} \ \ n\in\mathcal{N}_{\mathrm{n}}(\mathfrak{S}_0)
                        \end{cases}\ .
                    \]
                    Note that $w(\mathfrak{S}[n,u])/d_n' =  \omega_u m_{n,u}$, and $w(\mathfrak{S}[n,e])/d_n' =  \frac{\omega_u}{d_{u,e}} m_{n,u} = \omega_em_{n,u}$.
                    Therefore $C_u(t) = \Delta_u(t^{m_{n,u}})$ if $u\in\mathcal{N}_{\mathrm{n}}^{\mathfrak{S}_0}$, and $C_u(t) = \Delta_u(t^{m_{n,u}})/(1-t^{m_{n,u}})$ if $u\in\mathcal{N}_{\mathrm{e}}^{\mathfrak{S}_0}$.
                \end{proof}
            \begin{example}\label{ex:Splice3:0}
                Consider the splice diagram $\mathfrak{S}$ as given in left side of Figure \ref{fig:OverSplice3}, with $\mathfrak{S}_0$ given by the nodes $E_2,E_3,E_4$. Note that $\mathcal{N}_{\mathrm{e}}(\mathfrak{S}_0) = \{E_2,E_4\}$. Then the splice diagrams $\mathfrak{S}_0(2)$, $\mathfrak{S}_0(3)$, $\mathfrak{S}_0(4)$ with arrows on the cut edges, corresponding to $\mathcal{L}_2$, $\mathcal{L}_3$, $\mathcal{L}_4$, in this order, are given on the right hand side of the same figure.
                The normalized Alexander polynomials are given by
                \begin{align*}
                    \Delta_{\mathfrak{S}_0(2)}(t)&= (1-t)\cdot \frac{1-t^2}{1-t}\cdot \frac{1-t^{12}}{(1-t^4)(1-t^6)} = (1-t)\Delta_2(t)\frac{\Delta_1(t^2)}{1-t^2},\\
                    \Delta_{\mathfrak{S}_0(3)}(t)&= (1-t^5)\cdot \frac{1-t^{6}}{(1-t^2)(1-t^3)} = \Delta_3(t^5)\frac{\Delta_5(t)}{1-t},\\
                    \Delta_{\mathfrak{S}_0(4)}(t)&= (1-t)\cdot \frac{1-t^2}{1-t}\cdot \frac{1-t^{12}}{(1-t^4)(1-t^6)} = (1-t)\Delta_4(t)\frac{\Delta_6(t^2)}{1-t^2}.
                \end{align*}
                \begin{figure}
                    \centering
                    \renewcommand{\graphscale}{1.25cm}
                    \begin{tikzpicture}[roundnode/.style={circle, draw=black, fill=black, very thick, inner sep = 1},line width=1,x=\graphscale,y=\graphscale]
                        \node[roundnode] (3) {};
                        \path (3)++(1,0) node[roundnode] (4) {};
                        \path (3)++(0,-1) node[roundnode] (5) {};
                        \path (3)++(-1,0) node[roundnode] (2) {};
                        \path (4)++(1,0) node[roundnode] (6) {};
                        \path (4)++(0,1) node[roundnode] (12) {};
                        \path (5)++(-60:1) node[roundnode] (11) {};
                        \path (5)++(-120:1) node[roundnode] (10) {};
                        \path (2)++(-1,0) node[roundnode] (1) {};
                        \path (2)++(0,1) node[roundnode] (9) {};
                        \path (6)++(60:1) node[roundnode] (13) {};
                        \path (6)++(-60:1) node[roundnode] (14) {};
                        \path (1)++(120:1) node[roundnode] (7) {};
                        \path (1)++(-120:1) node[roundnode] (8) {};
                        \draw (7)--(1)--(8) (1)--(6)--(14) (2)--(9) (10)--(5)--(11) (5)--(3) (4)--(12) (6)--(13);
                        
                        \begin{scriptsize}
                            \renewcommand{\r}{0.25}
                            \foreach \vertex/\angle/\weight in {%
                            1/-35/181,1/-150/3,1/150/2,
                            2/-145/1,2/120/2,2/35/15,
                            3/150/3,3/35/4,3/-120/5,
                            4/145/9,4/-35/1,4/65/2,
                            5/140/17,5/-150/2,5/-30/3,
                            6/-145/109,6/30/2,6/-30/3
                            }{
                            \path (\vertex) ++ (\angle:\r) node {$\weight$};
                            }
                            \renewcommand{\r}{0.3}
                            \foreach \vertex/\angle in {%
                                1/60,2/-80,3/-50,5/50,4/-120,6/120
                            }{
                                \path (\vertex) ++ (\angle:\r) node {$E_{\vertex}$};
                            }
                        \end{scriptsize}
                    \end{tikzpicture}
                    \hspace{1cm}
                    \begin{tikzpicture}[roundnode/.style={circle, draw=black, fill=black, very thick, inner sep = 1},line width=1,x=\graphscale,y=\graphscale]
                        \node[roundnode] (3) {};
                        \path (3)++(1.5,0) node[roundnode] (4) {};
                        \path (3)++(0,-1) node[roundnode] (5) {};
                        \path (3)++(-1.5,0) node[roundnode] (2) {};
                        \path (4)++(1,0) node[roundnode] (6) {};
                        \path (4)++(0,1) node[roundnode] (12) {};
                        \path (5)++(-60:1) node[roundnode] (11) {};
                        \path (5)++(-120:1) node[roundnode] (10) {};
                        \path (2)++(-1,0) node[roundnode] (1) {};
                        \path (2)++(0,1) node[roundnode] (9) {};
                        \path (6)++(60:1) node[roundnode] (13) {};
                        \path (6)++(-60:1) node[roundnode] (14) {};
                        \path (1)++(120:1) node[roundnode] (7) {};
                        \path (1)++(-120:1) node[roundnode] (8) {};
                        \draw (7)--(1)--(8) (1)--(2)--(9) (10)--(5)--(11) (5)--(3) (6)--(4)--(12) (14)--(6)--(13);
                        \draw[->] (2)-++(0.5,0);
                        \draw[->] (3)-++(-0.5,0);
                        \draw[->] (3)-++(0.5,0);
                        \draw[->] (4)-++(-0.5,0);
                        \begin{scriptsize}
                            \renewcommand{\r}{0.25}
                            \foreach \vertex/\angle/\weight in {%
                            1/-35/181,1/-150/3,1/150/2,
                            2/-145/1,2/120/2,2/35/15,
                            3/150/3,3/35/4,3/-120/5,
                            4/145/9,4/-35/1,4/65/2,
                            5/140/17,5/-150/2,5/-30/3,
                            6/-145/109,6/30/2,6/-30/3
                            }{
                            \path (\vertex) ++ (\angle:\r) node {$\weight$};
                            }
                            \renewcommand{\r}{0.3}
                            \foreach \vertex/\angle in {%
                                1/60,2/-80,3/-50,5/50,4/-120,6/120
                            }{
                                \path (\vertex) ++ (\angle:\r) node {$E_{\vertex}$};
                            }
                        \end{scriptsize}
                    \end{tikzpicture}
                    \caption{}
                    \label{fig:OverSplice3}
                \end{figure}
            \end{example}
            \subsection{Surgery formula for the over-reduced zeta functions}\label{ss:5.3}
            Similarly as in the previous sections, we fix a complete sub-diagram $\mathfrak{S}_0$ in $\mathfrak{S}$. Our goal is to prove a surgery formula
            for the zeta function $f_\Gamma$ reduced to the variables 
            $\mathbf{t}_{\mathcal{N}(\mathfrak{S}_0)}$. 
             \begin{Rem}
                 In the next definitions, lemmas and discussions, the indexing of the diagrams 
                (eg. $\mathfrak{S}_0, \mathfrak{S}_3$, etc) might be strange. However, it is controlled 
                by the following precise principle. $\mathfrak{S}_0$ is always the fixed complete sub-diagram as above.
                The other indices are guided by the notations of  Lemma \ref{Lem:reverseIE}, where different sub-diagrams appear, e.g.  $\mathfrak{S}_1$, $\mathfrak{S}_2$, $\mathfrak{S}_3$. This  lemma is a key step
                in the proof of the surgery formula for the polynomial $\operatorname{Pol}_{\Gamma}(\mathbf{t}_{\mathcal{N}(\mathfrak{S}_0)})$, basically, it organizes all the other steps.
                Keeping its notations we can follow easier the different steps and substitutions. 
            \end{Rem}
            
            For any peel $\mathfrak{S}_3\hookrightarrow \mathfrak{S}_0$ we define the 
            \emph{full extension} in $\mathfrak{S}$ of $\mathfrak{S}_3$ as the component of $\mathfrak{S}\setminus(\mathfrak{S}_0\setminus\mathfrak{S}_3)$ that has non-empty intersection with $\mathfrak{S}_0$ (equivalently that contains $\mathfrak{S}_3$). The full extension of $\mathfrak{S}_3$ is denoted by $\overline{\mathfrak{S}_3}$. Note that $\overline{\mathfrak{S}_0} = \mathfrak{S}$. Additionally, for any $\mathfrak{S}_3\hookrightarrow\mathfrak{S}_0$ denote by $\overline{\mathfrak{S}_3}'$ the complete sub-diagram given by $\mathcal{N}(\overline{\mathfrak{S}_3}') = \mathcal{N}(\overline{\mathfrak{S}_3})\cup\mathcal{N}(\mathfrak{S}_0)$. Observe that $\overline{\mathfrak{S}_3}\hookrightarrow\overline{\mathfrak{S}_3}'$.

            Denote $\Gamma_0 :=\Gamma(\mathfrak{S}_0)$ and let $U_n = \frac{1}{\omega_{n,\mathfrak{S}_0}}\restr{E_{n,\Gamma_0}^*}{\mathcal{N}(\mathfrak{S}_0)}$, where $\omega_{n,\mathfrak{S}_0}$ is the product of the leg determinants $d_{n,e}$ around the node $n$ in $\mathfrak{S}_0$.  Note that $d_n'\omega_{n,\mathfrak{S}_0} = w(\mathfrak{S}[n,n])$ for all $n\in\mathcal{N}(\mathfrak{S}_0)$.
            \begin{Not}
                In the statements of Theorem \ref{Thm:oRedZetaSurgery} and Theorem \ref{Thm:oRedPolSurgery} if $\mathfrak{S}_0$ has exactly two nodes, we extend "$\hookrightarrow$" to include the empty graph too, denoted by $\emptyset$. This way we treat these two cases: $|\mathfrak{S}_0|=2$ and $|\mathfrak{S}_0|>2$, simultaneously. To do this we set  $(\varphi_{\overline{\emptyset},\overline{\emptyset}'}^*f_{\Gamma(\overline{\emptyset})})|_{\mathcal{N}(\mathfrak{S}_0)}(\mathbf{t}_{\mathcal{N}(\mathfrak{S}_0)}) := \prod_{n\in\mathcal{N}(\mathfrak{S}_0)}(1-\mathbf{t}^{U_n})^{-1}$ 
                and   $(\varphi_{\overline{\emptyset},\overline{\emptyset}'}^*\operatorname{Pol}_{\Gamma(\overline{\emptyset})})(\mathbf{t}_{\mathcal{N}(\mathfrak{S}_0)}):=0$. Note that 
                \[
                    \operatorname{PolPart}_{\mathcal{N}(\mathfrak{S}_0)}\big((\varphi_{\overline{\emptyset},\overline{\emptyset}'}^*f_{\Gamma(\overline{\emptyset})})|_{\mathcal{N}(\mathfrak{S}_0)}(\mathbf{t}_{\mathcal{N}(\mathfrak{S}_0)})\big) = 0 = (\varphi_{\overline{\emptyset},\overline{\emptyset}'}^*\operatorname{Pol}_{\Gamma(\overline{\emptyset})})(\mathbf{t}_{\mathcal{N}(\mathfrak{S}_0)}).
                \]
            \end{Not}

            The surgery formula for the over-reduced rational function $f_{\Gamma}(\mathbf{t}_{\mathcal{N}(\mathfrak{S}_0)})$ using the rational functions $\big(\varphi_{\overline{\mathfrak{S}_3},\overline{\mathfrak{S}_3}'}^*f_{\Gamma(\overline{\mathfrak{S}_3})}\big)|_{\mathcal{N}(\mathfrak{S}_0)}(\mathbf{t}_{\mathcal{N}(\mathfrak{S}_0)})$ 
            (for all $\mathfrak{S}_3\hookrightarrow \mathfrak{S}$) is the following.
            \begin{Thm}\label{Thm:oRedZetaSurgery}
                \begin{align}\label{eq:oRedZetaSugery}
                    \sum_{\mathfrak{S}_3\hookrightarrow\mathfrak{S}_0}(-1)^{|\mathfrak{S}_0|-|\mathfrak{S}_3|}&\big(\varphi_{\overline{\mathfrak{S}_3},\overline{\mathfrak{S}_3}'}^*f_{\Gamma(\overline{\mathfrak{S}_3})}\big)|_{\mathcal{N}(\mathfrak{S}_0)}(\mathbf{t}_{\mathcal{N}(\mathfrak{S}_0)})=\\ &= (-1)^{|\mathcal{N}_{\mathrm{e}}(\mathfrak{S}_0)|}\prod_{n\in\mathcal{N}_{\mathrm{n}}(\mathfrak{S}_0)}\Delta_{\mathfrak{S}_0(n)}(\mathbf{t}^{U_n})\prod_{n\in\mathcal{N}_{\mathrm{e}}(\mathfrak{S}_0)}\frac{1-\Delta_{\mathfrak{S}_0(n)}(\mathbf{t}^{U_n})}{1-\mathbf{t}^{U_n}}.\nonumber
                \end{align}
            \end{Thm}
            The proof of Theorem \ref{Thm:oRedZetaSurgery} becomes a straightforward calculation, once we show a similar decomposition of $\big(\varphi_{\overline{\mathfrak{S}_3},\overline{\mathfrak{S}_3}'}^*f_{\Gamma(\overline{\mathfrak{S}_3})}\big)|_{\mathcal{N}(\mathfrak{S}_0)}(\mathbf{t}_{\mathcal{N}(\mathfrak{S}_0)})$ (in Lemma \ref{Lem:overreducedZetaExpression}) as the one given in Lemma \ref{lem:induced_expr}.
            \begin{Rem}
                Note that in the proof of the following Lemma \ref{Lem:overreducedZetaExpression} we also prove that $\overline{\mathfrak{S}_3}\subseteq\mathfrak{S}$ for $\mathfrak{S}_3\hookrightarrow\mathfrak{S}_0$ is given by $\mathcal{N}(\overline{\mathfrak{S}_3}) = \bigsqcup_{n\in\mathcal{N}(\mathfrak{S}_3)}\mathcal{N}(\mathfrak{S}_0(n))$.   
            \end{Rem}
            \begin{Lem}\label{Lem:overreducedZetaExpression}
                Fix an arbitrary $\mathfrak{S}_3\hookrightarrow\mathfrak{S}_0$, and let $I:=\mathcal{N}_{\mathrm{e}}(\mathfrak{S}_0)\cap\mathcal{N}(\mathfrak{S}_3)$. Then
                \[
                    \big(\varphi_{\overline{\mathfrak{S}_3},\overline{\mathfrak{S}_3}'}^*f_{\Gamma(\overline{\mathfrak{S}_3})}\big)|_{\mathcal{N}(\mathfrak{S}_0)}(\mathbf{t}_{\mathcal{N}(\mathfrak{S}_0)}) = \prod_{n\in\mathcal{N}_{\mathrm{n}}(\mathfrak{S}_0)}\Delta_{\mathfrak{S}_0(n)}(\mathbf{t}^{U_n})\prod_{n\in I}\frac{\Delta_{\mathfrak{S}_0(n)}(\mathbf{t}^{U_n})}{1-\mathbf{t}^{U_n}}\prod_{n\in \mathcal{N}_{\mathrm{e}}(\mathfrak{S}_0)\setminus I}\frac{1}{1-\mathbf{t}^{U_n}}.
                \]
            \end{Lem}
            \begin{proof}
                In this proof denote by $\Gamma_3:=\Gamma(\overline{\mathfrak{S}_3}')$, $J := \mathcal{N}_{\mathrm{e}}(\mathfrak{S}_0)\setminus I$ and for each $u\in\mathcal{N}(\Gamma_3)$ the cycles $X_u := \frac{1}{\omega_{u,\Gamma_3}}\restr{E_{u,\Gamma_3}^*}{\mathcal{N}(\Gamma_3)}$. By equations (\ref{eq:reducedzeta_xn}) and (\ref{eq:AlexanderSingleNode}) we can write that
                \[
                    \big(\varphi_{\overline{\mathfrak{S}_3},\overline{\mathfrak{S}_3}'}^*f_{\Gamma(\overline{\mathfrak{S}_3})}\big)(\mathbf{t}_{\mathcal{N}(\Gamma_3)})=\prod_{u\in\mathcal{N}_{\mathrm{n}}(\Gamma_3)}\Delta_u(\mathbf{t}^{X_u})\prod_{u\in\mathcal{N}_{\mathrm{e}}(\Gamma_3)\setminus J}\frac{\Delta_u(\mathbf{t}^{X_u})}{1-\mathbf{t}^{X_u}}\prod_{u\in J}\frac{1}{1-\mathbf{t}^{X_u}}.
                \]
                In {\bf Step 1} below we will show that $\mathcal{N}(\Gamma_3)\setminus J=\mathcal{N}(\overline{\mathfrak{S}_3}) = \bigsqcup_{n\in\mathcal{N}(\mathfrak{S}_3)}\mathcal{N}(\mathfrak{S}_0(n))$, as such, we can redistribute the above products as
                \[
                    \big(\varphi_{\overline{\mathfrak{S}_3},\overline{\mathfrak{S}_3}'}^*f_{\Gamma(\overline{\mathfrak{S}_3})}\big)(\mathbf{t}_{\mathcal{N}(\Gamma_3)})=\prod_{n\in\mathcal{N}(\mathfrak{S}_0)\setminus J}\Big(\prod_{u\in\mathcal{N}_{\mathrm{n}}^{\mathfrak{S}_0(n)}}\Delta_u(\mathbf{t}^{X_u})\prod_{u\in\mathcal{N}_{\mathrm{e}}^{\mathfrak{S}_0(n)}}\frac{\Delta_u(\mathbf{t}^{X_u})}{1-\mathbf{t}^{X_u}}\Big)\cdot\prod_{u\in J}\frac{1}{1-\mathbf{t}^{X_u}}.
                \]
                After which, due to Lemma \ref{Lem:AlexanderMultipleFromSingle} it suffices to show that $\restr{X_u}{\mathcal{N}(\mathfrak{S}_0)} = m_{n,u}U_n$ for each $u\in\mathfrak{S}_0(n)$ and $n\in\mathcal{N}(\mathfrak{S}_0)\setminus J$, and $\restr{X_n}{\mathcal{N}(\mathfrak{S}_0)} = U_n$ for $n\in J$ (see {\bf Step 2} below).

                {\bf Step 1.} We can write that
                \[
                    \mathcal{N}(\overline{\mathfrak{S}_3}) = \bigcup\Big\{\mathfrak{S}[u,v]\Big|u\in\mathcal{N}(\mathfrak{S}),v\in\mathcal{N}(\mathfrak{S}_3),\mathfrak{S}[u,v]\cap J=\emptyset\Big\}.
                \]
                Due to the tree-graph structure of $\mathfrak{S}$ and $\mathfrak{S}_0$ we can say that for any $u\in\mathcal{N}(\mathfrak{S})$, there is exactly one $n_u\in\mathcal{N}(\mathfrak{S}_0)$ such that $\mathfrak{S}[u,n_u]\cap\mathcal{N}(\mathfrak{S}_0) = \{n_u\}$. Note that for $u\in\mathcal{N}(\mathfrak{S}_0)$ we have $n_u=u$. With $\mathfrak{S}[u,v] = \mathfrak{S}[u,n_u]\cup\mathfrak{S}[n_u,v]$ and $\mathfrak{S}[u,v]\cap J =\mathfrak{S}[n_u,v]\cap J$ we can write that
                \[
                    \mathcal{N}(\overline{\mathfrak{S}_3}) = \bigcup\Big\{\mathfrak{S}[u,n_u]\Big|u\in\mathcal{N}(\mathfrak{S}),n_u\in\mathcal{N}(\mathfrak{S}_3),\mathfrak{S}[u,n_u]\cap J=\emptyset\Big\}.
                \]
                The condition $\mathfrak{S}[u,n_u]\cap J = \emptyset$ is equivalent to $\{n_u\}\cap J=\emptyset$, which is true for any $n_u\in\mathcal{N}(\mathfrak{S}_3)$. Thus we can further deduce that
                \[
                    \mathcal{N}(\overline{\mathfrak{S}_3}) = \bigcup\Big\{\mathfrak{S}[u,n]\Big|u\in\mathcal{N}(\mathfrak{S}),n\in\mathcal{N}(\mathfrak{S}_3),\mathfrak{S}[u,n]\cap\mathcal{N}(\mathfrak{S}_0)=\{n\}\Big\}.
                \]
                Bringing out $n\in \mathcal{N}(\mathfrak{S}_3)$ to under the union, we have
                \[
                    \mathcal{N}(\overline{\mathfrak{S}_3}) = \bigcup_{n\in\mathcal{N}(\mathfrak{S}_3)}\Big(\bigcup\Big\{\mathfrak{S}[u,n]\Big|u\in\mathcal{N}(\mathfrak{S}),\mathfrak{S}[u,n]\cap\mathcal{N}(\mathfrak{S}_0)=\{n\}\Big\}\Big)=\bigcup_{n\in\mathcal{N}(\mathfrak{S}_3)}\mathcal{N}(\mathfrak{S}_0(n)).
                \]
                Note that the sets $\{\mathcal{N}(\mathfrak{S}_0(n))\}_{n\in\mathcal{N}(\mathfrak{S}_0)}$ are pairwise disjoint. Otherwise for $n_1\neq n_2$ and $u \in \mathcal{N}(\mathfrak{S}_0(n_1))\cap \mathcal{N}(\mathfrak{S}_0(n_2))$ the set $\mathfrak{S}[u,n_1]\cup\mathfrak{S}[u,n_2]\cup\mathfrak{S}_0[n_1,n_2]$ would represent a cycle in $\mathfrak{S}$.
                
                {\bf Step 2.} Let $n\in\mathcal{N}(\mathfrak{S}_0)\setminus J$, and fix an arbitrary $v\in\mathcal{N}(\mathfrak{S}_0)$, and $u\in\mathcal{N}(\mathfrak{S}_0(n))$. Note that $n\in\mathfrak{S}[v,u]$. Recall that for $n\not\sim u$ the number $d_{n,u}$ is equal to $d_{n,w}$ where $n\sim w$ and $w\in\mathfrak{S}[n,u]$ (see section \ref{ss:splice}). Similarly for $d_{n,v}$.
                We can write that
                \begin{align*}
                    w(\mathfrak{S}[v,u]) &= w(\mathfrak{S}[v,n))\cdot w(\mathfrak{S}(n,u])\cdot\frac{w(\mathfrak{S}[n,n])}{d_{n,u}d_{n,v}}\\
                        &=\frac{w(\mathfrak{S}[v,n])}{w(\mathfrak{S}[n,n])/d_{n,v}}\cdot\frac{w(\mathfrak{S}[n,u])}{w(\mathfrak{S}[n,n])/d_{n,u}}\cdot\frac{w(\mathfrak{S}[n,n])}{d_{n,u}d_{n,v}}\\
                        &=\frac{1}{w(\mathfrak{S}[n,n])}w(\mathfrak{S}[v,n])w(\mathfrak{S}[n,u]).
                \end{align*}  
                Using this identity it follows that
                \begin{align*}
                    m_v(X_u) &= \frac{w(\mathfrak{S}[v,u])}{\omega_u} = m_{n,u}d_n'\cdot\frac{w(\mathfrak{S}[v,u])}{w(\mathfrak{S}[u,n])} \\&=m_{n,u} d_n'\frac{w(\mathfrak{S}[v,n])}{w(\mathfrak{S}[n,n])} = m_{n,u}\frac{w(\mathfrak{S}[v,n])}{\omega_{n,\mathfrak{S}_0}}= m_{n,u}m_v(U_n).
                \end{align*}
                For $n\in J\subseteq\mathcal{N}_{\mathrm{e}}(\mathfrak{S}_0)\cap\mathcal{N}_{\mathrm{e}}(\Gamma_3)$ and $v\in\mathcal{N}(\mathfrak{S}_0)\setminus\{n\}$ we have $ m_v(X_n) = w(\mathfrak{S}[v,n)) = m_v(U_n)$; otherwise $m_n(X_n) = d_{n,w_n} = m_n(U_n)$, where $w_n\sim n$ and $w_n\in\mathcal{N}(\mathfrak{S}_0)$.
            \end{proof}
            \begin{example}\label{ex:Splice3:1}
                Let us continue with the context of Example \ref{ex:Splice3:0} (see Figure \ref{fig:OverSplice3}). Let $\mathfrak{S}_3:=\mathfrak{S}_{\{2,3\}}\hookrightarrow\mathfrak{S}_0$ be given by the nodes $\{E_2,E_3\}$; 
                thus $\overline{\mathfrak{S}_{3}}$ and $\overline{\mathfrak{S}_{3}}'$ are given by $\{E_1,E_2,E_3,E_5\}$ and $\{E_1,E_2,E_3,E_4,E_5\}$ respectively, see Figure \ref{fig:OverSplice3:1}. Using the notations of Lemma \ref{Lem:overreducedZetaExpression} we can write that
                \begin{align*}
                    \big(\varphi_{\overline{\mathfrak{S}_3},\overline{\mathfrak{S}_3}'}^*f_{\Gamma(\overline{\mathfrak{S}_3})}\big)(\mathbf{t}_{\mathcal{N}(\Gamma_3)}) = \frac{1-\mathbf{t}^{6X_1}}{(1-\mathbf{t}^{2X_1})(1-\mathbf{t}^{3X_1})}\frac{1-\mathbf{t}^{2X_2}}{1-\mathbf{t}^{X_2}}(1-\mathbf{t}^{X_3})\frac{1-\mathbf{t}^{6X_5}}{(1-\mathbf{t}^{2X_5})(1-\mathbf{t}^{3X_5})}\frac{1}{1-\mathbf{t}^{X_4}},
                \end{align*}
                where 
                \[
                    \begin{bmatrix}
                        X_1\\X_2\\X_3\\X_4\\X_5
                    \end{bmatrix} = \begin{bmatrix}
                        181&30&40&20&48\\
                        90&15&20&10 &24\\
                        120&40&60&30&72\\
                        60& 10&15&9&18\\
                        48&8&12&6&17
                    \end{bmatrix}\cdot\begin{bmatrix}
                        E_1\\E_2\\E_3\\E_4\\E_5
                    \end{bmatrix},\quad\mbox{and}\quad \begin{bmatrix}
                        U_2\\U_3\\U_4
                    \end{bmatrix} = \begin{bmatrix}
                        15&20&10\\
                        8&12&6\\
                        10&15&9
                    \end{bmatrix}\cdot\begin{bmatrix}
                        E_2\\E_3\\E_4
                    \end{bmatrix}.
                \]
                Note that $\restr{X_1}{\mathcal{N}(\mathfrak{S}_0)} = 2U_2$, $\restr{X_2}{\mathcal{N}(\mathfrak{S}_0)} = U_2$, $\restr{X_3}{\mathcal{N}(\mathfrak{S}_0)} = 5U_3$, $\restr{X_4}{\mathcal{N}(\mathfrak{S}_0)} = U_4$, $\restr{X_5}{\mathcal{N}(\mathfrak{S}_0)} = U_3$. Thus, as seen in Lemma \ref{Lem:AlexanderMultipleFromSingle} and the polynomials in Example \ref{ex:Splice3:0}, we have that
                \begin{align*}
                    f_{\{2,3\}}&(t_2,t_3,t_4) := \big(\varphi_{\overline{\mathfrak{S}_3},\overline{\mathfrak{S}_3}'}^*f_{\Gamma(\overline{\mathfrak{S}_3})}\big)|_{\mathcal{N}(\mathfrak{S}_0)}(\mathbf{t}_{\mathcal{N}(\mathfrak{S}_0)})= \\
                    &= \frac{1-\mathbf{t}^{12U_2}}{(1-\mathbf{t}^{4U_2})(1-\mathbf{t}^{6U_2})}\frac{1-\mathbf{t}^{2U_2}}{1-\mathbf{t}^{U_2}}(1-\mathbf{t}^{5U_3})\frac{1-\mathbf{t}^{6U_3}}{(1-\mathbf{t}^{2U_3})(1-\mathbf{t}^{3U_3})}\frac{1}{1-\mathbf{t}^{U_4}}\\
                    &=\frac{\Delta_{\mathfrak{S}_0(2)}(\mathbf{t}^{U_2})}{1-\mathbf{t}^{U_2}}\Delta_{\mathfrak{S}_0(3)}(\mathbf{t}^{U_3})\frac{1}{1-\mathbf{t}^{U_4}}.
                \end{align*}
                \begin{figure}
                    \centering
                    \renewcommand{\graphscale}{1.25cm}
                    \begin{tikzpicture}[roundnode/.style={circle, draw=black, fill=black, very thick, inner sep = 1},line width=1,x=\graphscale,y=\graphscale]
                        \node[roundnode] (3) {};
                        \path (3)++(1,0) node[roundnode] (4) {};
                        \path (3)++(0,-1) node[roundnode] (5) {};
                        \path (3)++(-1,0) node[roundnode] (2) {};
                        \path (5)++(-60:1) node[roundnode] (11) {};
                        \path (5)++(-120:1) node[roundnode] (10) {};
                        \path (2)++(-1,0) node[roundnode] (1) {};
                        \path (2)++(0,1) node[roundnode] (9) {};
                        \path (1)++(120:1) node[roundnode] (7) {};
                        \path (1)++(-120:1) node[roundnode] (8) {};
                        \draw (7)--(1)--(8) (1)--(4) (2)--(9) (10)--(5)--(11) (5)--(3);
                        
                        \begin{scriptsize}
                            \renewcommand{\r}{0.25}
                            \foreach \vertex/\angle/\weight in {%
                            1/-35/181,1/-150/3,1/150/2,
                            2/-145/1,2/120/2,2/35/15,
                            3/150/3,3/35/4,3/-120/5,
                            5/140/17,5/-150/2,5/-30/3
                            }{
                            \path (\vertex) ++ (\angle:\r) node {$\weight$};
                            }
                            \renewcommand{\r}{0.3}
                            \foreach \vertex/\angle in {%
                                1/60,2/-80,3/-50,5/50
                            }{
                                \path (\vertex) ++ (\angle:\r) node {$E_{\vertex}$};
                            }
                           
                        \end{scriptsize}
                         \path (10)++(0,-0.5) node {$\overline{\mathfrak{S}_3}$};
                    \end{tikzpicture}
                    \hspace{1cm}
                    \begin{tikzpicture}[roundnode/.style={circle, draw=black, fill=black, very thick, inner sep = 1},line width=1,x=\graphscale,y=\graphscale]
                        \node[roundnode] (3) {};
                        \path (3)++(1,0) node[roundnode] (4) {};
                        \path (3)++(0,-1) node[roundnode] (5) {};
                        \path (3)++(-1,0) node[roundnode] (2) {};
                        \path (4)++(1,0) node[roundnode] (6) {};
                        \path (4)++(0,1) node[roundnode] (12) {};
                        \path (5)++(-60:1) node[roundnode] (11) {};
                        \path (5)++(-120:1) node[roundnode] (10) {};
                        \path (2)++(-1,0) node[roundnode] (1) {};
                        \path (2)++(0,1) node[roundnode] (9) {};
                        \path (1)++(120:1) node[roundnode] (7) {};
                        \path (1)++(-120:1) node[roundnode] (8) {};
                        \draw (7)--(1)--(8) (1)--(6) (2)--(9) (10)--(5)--(11) (5)--(3) (4)--(12);
                        
                        \begin{scriptsize}
                            \renewcommand{\r}{0.25}
                            \foreach \vertex/\angle/\weight in {%
                            1/-35/181,1/-150/3,1/150/2,
                            2/-145/1,2/120/2,2/35/15,
                            3/150/3,3/35/4,3/-120/5,
                            4/145/9,4/-35/1,4/65/2,
                            5/140/17,5/-150/2,5/-30/3
                            }{
                            \path (\vertex) ++ (\angle:\r) node {$\weight$};
                            }
                            \renewcommand{\r}{0.3}
                            \foreach \vertex/\angle in {%
                                1/60,2/-80,3/-50,5/50,4/-120
                            }{
                                \path (\vertex) ++ (\angle:\r) node {$E_{\vertex}$};
                            }
                            
                        \end{scriptsize}
                        \path (10)++(0,-0.5) node {$\overline{\mathfrak{S}_3}'$};
                    \end{tikzpicture}
                    \caption{}
                    \label{fig:OverSplice3:1}
                \end{figure}
                
            \end{example}
            \begin{proof}[Proof of Theorem \ref{Thm:oRedZetaSurgery}]
                By Lemma \ref{Lem:overreducedZetaExpression} we have 
                \begin{align*}
                    \sum_{\mathfrak{S}_3\hookrightarrow\mathfrak{S}_0}(-1)^{|\mathfrak{S}_3|}&\big(\varphi_{\overline{\mathfrak{S}_3},\overline{\mathfrak{S}_3}'}^*f_{\Gamma(\overline{\mathfrak{S}_3})}\big)|_{\mathcal{N}(\mathfrak{S}_0)}(\mathbf{t}_{\mathcal{N}(\mathfrak{S}_0)})= \\&= (-1)^{|\mathcal{N}_{\mathrm{n}}(\mathfrak{S}_0)|} \prod_{n\in\mathcal{N}_{\mathrm{n}}(\mathfrak{S}_0)}\Delta_{\mathfrak{S}_0(n)}(\mathbf{t}^{U_n})\prod_{n\in\mathcal{N}_{\mathrm{e}}(\mathfrak{S}_0)}\frac{1}{1-\mathbf{t}^{U_n}}\sum_{I\subseteq\mathcal{N}_{\mathrm{e}}(\mathfrak{S}_0)}(-1)^{|I|}\prod_{n\in I}\Delta_{\mathfrak{S}_0(n)}(\mathbf{t}^{U_n})\\
                    &=(-1)^{|\mathcal{N}_{\mathrm{n}}(\mathfrak{S}_0)|} \prod_{n\in\mathcal{N}_{\mathrm{n}}(\mathfrak{S}_0)}\Delta_{\mathfrak{S}_0(n)}(\mathbf{t}^{U_n})\prod_{n\in\mathcal{N}_{\mathrm{e}}(\mathfrak{S}_0)}\frac{1}{1-\mathbf{t}^{U_n}}\prod_{n\in\mathcal{N}_{\mathrm{e}}(\mathfrak{S}_0)}(1-\Delta_{\mathfrak{S}_0(n)}(\mathbf{t}^{U_n})).
                \end{align*}
            \end{proof}
            \begin{example}
                We continue with Example \ref{ex:Splice3:0}. Performing similar computations as in Example \ref{ex:Splice3:1} we can write the following. For $\mathfrak{S}_3:=\mathfrak{S}_{\{3,4\}}\hookrightarrow\mathfrak{S}_0$ given by $\{E_3,E_4\}$ we  have $\mathcal{N}(\overline{\mathfrak{S}_3}) = \{3,4,5,6\}$, $\mathcal{N}(\overline{\mathfrak{S}_3'}) = \{2,3,4,5,6\}$ and
                \[
                     f_{\{3,4\}}(t_2,t_3,t_4) :=\big(\varphi_{\overline{\mathfrak{S}_3},\overline{\mathfrak{S}_3}'}^*f_{\Gamma(\overline{\mathfrak{S}_3})}\big)|_{\mathcal{N}(\mathfrak{S}_0)}(\mathbf{t}_{\mathcal{N}(\mathfrak{S}_0)}) = \frac{1}{1-\mathbf{t}^{U_2}}\Delta_{\mathfrak{S}_0(3)}(\mathbf{t}^{U_3})\frac{\Delta_{\mathfrak{S}_0(4)}(\mathbf{t}^{U_4})}{1-\mathbf{t}^{U_4}}.
                \]
                For $\mathfrak{S}_3:=\mathfrak{S}_{\{3\}}\hookrightarrow\mathfrak{S}_0$ given by $\{E_3\}$ we have $\mathcal{N}(\overline{\mathfrak{S}_3}) = \{3,5\}$, $\mathcal{N}(\overline{\mathfrak{S}_3}') = \{2,3,4,5\}$ and
                \[
                    f_{\{3\}}(t_2,t_3,t_4) := \big(\varphi_{\overline{\mathfrak{S}_3},\overline{\mathfrak{S}_3}'}^*f_{\Gamma(\overline{\mathfrak{S}_3})}\big)|_{\mathcal{N}(\mathfrak{S}_0)}(\mathbf{t}_{\mathcal{N}(\mathfrak{S}_0)})= \frac{1}{1-\mathbf{t}^{U_2}}\Delta_{\mathfrak{S}_0(3)}(\mathbf{t}^{U_3})\frac{1}{1-\mathbf{t}^{U_4}}.
                \]
                And finally for $\mathfrak{S}_3:=\mathfrak{S}_{\{2,3,4\}} =\mathfrak{S}_0$ given by the nodes $\{E_2,E_3,E_4\}$, we have $\overline{\mathfrak{S}_3}=\overline{\mathfrak{S}_3}'= \mathfrak{S}$, and
                \begin{align*}
                     f_{\{2,3,4\}}(t_2,t_3,t_4) := \big(\varphi_{\overline{\mathfrak{S}_3},\overline{\mathfrak{S}_3}'}^*f_{\Gamma(\overline{\mathfrak{S}_3})}\big)|_{\mathcal{N}(\mathfrak{S}_0)}(\mathbf{t}_{\mathcal{N}(\mathfrak{S}_0)})&=f_{\Gamma(\mathfrak{S})}(\mathbf{t}_{\mathcal{N}(\mathfrak{S}_0)}) \\
                     &= \frac{\Delta_{\mathfrak{S}_0(2)}(\mathbf{t}^{U_2})}{1-\mathbf{t}^{U_2}}\Delta_{\mathfrak{S}_0(3)}(\mathbf{t}^{U_3})\frac{\Delta_{\mathfrak{S}_0(4)}(\mathbf{t}^{U_4})}{1-\mathbf{t}^{U_4}}.
                \end{align*}
                Thus
                \begin{align*}
                    f_{\{2,3,4\}}(\mathbf{t}_{\mathcal{N}(\mathfrak{S}_0)})&-f_{\{2,3\}}(\mathbf{t}_{\mathcal{N}(\mathfrak{S}_0)})-f_{\{3,4\}}(\mathbf{t}_{\mathcal{N}(\mathfrak{S}_0)})+f_{\{3\}}(\mathbf{t}_{\mathcal{N}(\mathfrak{S}_0)})=\\
                    &\qquad=\frac{1-\Delta_{\mathfrak{S}_0(2)}(\mathbf{t}^{U_2})}{1-\mathbf{t}^{U_2}}\Delta_{\mathfrak{S}_0(3)}(\mathbf{t}^{U_3})\frac{1-\Delta_{\mathfrak{S}_0(4)}(\mathbf{t}^{U_4})}{1-\mathbf{t}^{U_4}}
                \end{align*}
                The same formula holds for the over-reduced polynomial part, as will be shown in Theorem \ref{Thm:oRedPolSurgery}.
            \end{example}
            \subsection{Surgery formula for the over-reduced polynomial parts}\label{ss:5.4}
            We use the same notations as in the previous section.
            
            Our goal  is to apply the operator  $\operatorname{PolPart}_{\mathcal{N}(\mathfrak{S}_0)}$ to the formula (\ref{eq:oRedZetaSugery}).
            However, unlike in the proof of Theorem \ref{Thm:PolSurgery},  the equality
            \[
                \operatorname{PolPart}_{\mathcal{N}(\mathfrak{S}_0)}\big((\varphi_{\overline{\mathfrak{S}_3},\overline{\mathfrak{S}_3}'}^*f_{\Gamma(\overline{\mathfrak{S}_3})})|_{\mathcal{N}(\mathfrak{S}_0)}(\mathbf{t}_{\mathcal{N}(\mathfrak{S}_0)})\big)=  \big(\varphi_{\overline{\mathfrak{S}_3},\overline{\mathfrak{S}_3}'}^*\operatorname{Pol}_{\Gamma(\overline{\mathfrak{S}_3})}\big)(\mathbf{t}_{\mathcal{N}(\mathfrak{S}_0)}),
            \]
            fails to hold, as $\mathcal{N}(\mathfrak{S}_0)$ can be a strict subset of $\mathcal{N}(\overline{\mathfrak{S}_3}')$, thus the right hand side may have more monomials than the left hand side. Nevertheless, their signed sum (like the left side of (\ref{eq:oRedZetaSugery})) with $\mathfrak{S}_3\hookrightarrow\mathfrak{S}_0$ end up being equal as it will be shown in this section in Theorem \ref{Thm:oRedPolSurgery}. The idea is based on Remark \ref{Rem:IterPolPart}. This alone is not enough, since the right hand side (the sum of the polynomial parts, reduced) will be have the operator $\operatorname{PolPart}_{\mathcal{N}(\mathfrak{S}_0)}$ applied to it. We need to show that the sum of the polynomial parts (reduced to $\mathcal{N}(\mathfrak{S}_0)$ after the respective variable changes) is still a polynomial, i.e., an element of $\Z[\mathbf{t}_{\mathcal{N}(\mathfrak{S}_0)}]$. To do this, we will use the reformulation of Theorem \ref{Thm:PolSurgery} mentioned in Remark \ref{Rem:PolSurgeryReform} and an inversion formula for the sums of type $\sum_{\mathfrak{S}'\hookrightarrow\mathfrak{S}} (-1)^{|\mathfrak{S}'|}A_{\mathfrak{S}'}$, where $A_{\mathfrak{S}'}$ are formal variables associated for each $\mathfrak{S}'\subseteq\mathfrak{S}$.

            \vspace{2mm}

            We say that $\mathfrak{S}_2\subseteq\mathfrak{S}$ is \emph{extendable} (with respect to $\mathfrak{S}_0$) if $\mathcal{N}(\mathfrak{S}_0)\setminus\mathcal{N}(\mathfrak{S}_2)\subseteq\mathcal{N}_{\mathrm{e}}(\mathfrak{S}_0)$ and $\mathcal{N}(\mathfrak{S}_2)\cap\mathcal{N}(\mathfrak{S}_0)\neq\emptyset$.
            \begin{Lem}\label{Lem:reverseIE}
                We have the following inclusion-exclusion-inversion type formula
                \[
                    \sum_{\mathfrak{S}_0\subseteq \mathfrak{S}_1\subseteq\mathfrak{S}}(-1)^{|\mathfrak{S}_1|}\sum_{\mathfrak{S}_2\hookrightarrow\mathfrak{S}_1}(-1)^{|\mathfrak{S}_2|}A_{\mathfrak{S}_2} = \sum_{\mathfrak{S}_3\hookrightarrow\mathfrak{S}_0} \ (-1)^{|\mathfrak{S}_0|-|\mathfrak{S}_3|}A_{\overline{\mathfrak{S}_3}}.
                \]
            \end{Lem}
            
            \begin{proof}
                Note that $A_{\mathfrak{S}_2}$ appears in the left sum exactly when $\mathfrak{S}_2$ is extendable.
                The coefficient of $A_{\mathfrak{S}_2}$ for $\mathfrak{S}_2\subseteq\mathfrak{S}$ is given by the following summation over $\mathfrak{S}_1$ under the specified conditions 
                \[
                    \sum\{(-1)^{|\mathfrak{S}_1|+|\mathfrak{S}_2|}\mid \mathfrak{S}_0\subseteq\mathfrak{S}_1\subseteq\mathfrak{S}, \mathfrak{S}_2\hookrightarrow \mathfrak{S}_1\}.
                \]                
                Let $\mathfrak{S}_2'\subseteq\mathfrak{S}$ be the largest such that $\mathfrak{S}_2\hookrightarrow\mathfrak{S}_2'$. Then $\mathfrak{S}_0\subseteq \mathfrak{S}_2'$, and the nodes of $\mathfrak{S}_2'$ can be expressed as $\mathcal{N}(\mathfrak{S}_2') = \mathcal{N}(\mathfrak{S}_2)\cup(\mathcal{N}(\mathfrak{S}_0)\setminus\mathcal{N}(\mathfrak{S}_2))\cup K$. Then $\mathfrak{S}_1$ contributes to the coefficient of $\mathfrak{S}_2$ if and only if
                $\mathcal{N}(\mathfrak{S}_1) = \mathcal{N}(\mathfrak{S}_2)\cup(\mathcal{N}(\mathfrak{S}_0)\setminus\mathcal{N}(\mathfrak{S}_2))\cup J$, for some $J\subseteq K$. Thus the coefficient of $A_{\mathfrak{S}_2}$ is given by
                \[
                    (-1)^{|\mathfrak{S}_2|+|\mathfrak{S}_2|+|\mathcal{N}(\mathfrak{S}_0)\setminus\mathcal{N}(\mathfrak{S}_2)|}\sum_{J\subseteq K}(-1)^{|J|} = (-1)^{|\mathcal{N}(\mathfrak{S}_0)\setminus\mathcal{N}(\mathfrak{S}_2)|}\cdot\begin{cases}
                        1& \mbox{if} \ K=\emptyset\\
                        0& \mbox{if} \  K\neq\emptyset.
                    \end{cases}
                \]
                Let $\mathfrak{S}_3\hookrightarrow\mathfrak{S}_0$ be given by $\mathcal{N}(\mathfrak{S}_3)=\mathcal{N}(\mathfrak{S}_0)\cap\mathcal{N}(\mathfrak{S}_2)$.
                Note that $K=\emptyset$ if and only if $\mathfrak{S}_2 = \overline{\mathfrak{S}_3}$. The coefficient is 
                \[
                    (-1)^{|\mathcal{N}(\mathfrak{S}_0)\setminus\mathcal{N}(\mathfrak{S}_2)|} = (-1)^{|\mathcal{N}(\mathfrak{S}_0)\setminus\mathcal{N}(\mathfrak{S}_3)|} = (-1)^{|\mathfrak{S}_0| - |\mathfrak{S}_3|}.
                \]
            \end{proof}
            \begin{Rem}
                To actually use this inversion formula, with $A_{\mathfrak{S}_2} :=\big(\varphi_{\mathfrak{S}_2,\mathfrak{S}_1}^*\operatorname{Pol}_{\Gamma(\mathfrak{S}_2)}\big)(\mathbf{t}_{\mathcal{N}(\mathfrak{S}_0)})$, we need to show that this substitution is well defined, that is $\big(\varphi_{\mathfrak{S}_2,\mathfrak{S}_1}^*\operatorname{Pol}_{\Gamma(\mathfrak{S}_2)}\big)(\mathbf{t}_{\mathcal{N}(\mathfrak{S}_0)})$ is independent of the choice of $\mathfrak{S}_1$, where $\mathfrak{S}_2\hookrightarrow\mathfrak{S}_1$. This will 
                guarantee that similar cancellations will occur for these polynomials as in Lemma \ref{Lem:reverseIE}.

                For instance consider the left diagram in Figure \ref{fig:OverSplice3} and 
                $\mathfrak{S}_2$ given by the nodes $\mathcal{N}(\mathfrak{S}_2)=\{2,3\}$. It has extensions $\mathfrak{S}_1$ which correspond to $\{2,3,4\}$, $\{1,2,3,4\}$, $\{2,3,4,5\}$, and $\{1,2,3,4,5\}$. These yield four different change of variable maps which applied to $\operatorname{Pol}_{\Gamma(\mathfrak{S}_2)}(\mathbf{t}_{\mathcal{V}(\Gamma(\mathfrak{S}_2))})$, yield series in  variables $\mathbf{t}_{\{2,3,4\}}$, $\mathbf{t}_{\{1,2,3,4\}}$, $\mathbf{t}_{\{2,3,4,5\}}$, and $\mathbf{t}_{\{1,2,3,4,5\}}$, respectively. Thus in order to have the cancellations as in Lemma \ref{Lem:reverseIE}, we need to show that further restricting these series to the nodes $\{2,3,4\}$ yields the same series.
                
            \end{Rem}
            Let $\mathfrak{S}_2\subseteq\mathfrak{S}$ be extendable with respect to $\mathfrak{S}_0$. Also let $\mathfrak{S}_1\subseteq\mathfrak{S}$ such that $\mathfrak{S}_0\subseteq\mathfrak{S}_1$ and $\mathfrak{S}_2$ is a peel of $\mathfrak{S}_1$. Note that if $\mathfrak{S}_2 = \overline{\mathfrak{S}_3}$ for some $\mathfrak{S}_3\hookrightarrow\mathfrak{S}_0$, then there is only one choice for of $\mathfrak{S}_1$ and that is exactly $\overline{\mathfrak{S}_3}'$.
            Let $\mathcal{V}_2$ be the vertex set of the plumbing tree associated with $\mathfrak{S}_2$.
            \begin{Lem}\label{Lem:PolIndependence}
                 Let $S\in \Z[[\mathbf{t}_{\mathcal{V}_2}^{-1}]][\mathbf{t}_{\mathcal{V}_2}]$. Then $(\varphi_{\mathfrak{S}_2,\mathfrak{S}_1}^*S)|_{\mathcal{N}(\mathfrak{S}_0)}(\mathbf{t}_{\mathcal{N}(\mathfrak{S}_0)})$ is independent of $\mathfrak{S}_1$. Additionally $(\varphi_{\overline{\mathfrak{S}_3},\overline{\mathfrak{S}_3}'}^*S)|_{\mathcal{N}(\mathfrak{S}_0)}(\mathbf{t}_{\mathcal{N}(\mathfrak{S}_0)}) = (\varphi_{\mathfrak{S}_3,\mathfrak{S}_0}^*S)(\mathbf{t}_{\mathcal{N}(\mathfrak{S}_0)})$ for any $\mathfrak{S}_3\hookrightarrow\mathfrak{S}_0$.
            \end{Lem}
            In the statement in order, to apply $\varphi_{\mathfrak{S}_2,\mathfrak{S}_1}^*$, first we reduce $S(\mathbf{t}_{\mathcal{V}_2})$ to the variables determined by the domain of $\varphi_{\mathfrak{S}_2,\mathfrak{S}_1}$ (assuming this is well defined). After applying $\varphi_{\mathfrak{S}_2,\mathfrak{S}_1}^*$, we further reduce to the nodes of $\mathfrak{S}_0$. The series obtained this way is independent of the choice of $\mathfrak{S}_1$.
            \begin{proof}
                Suppose $\mathfrak{S}_1'\subseteq\mathfrak{S}$ is such that $\mathfrak{S}_2$ is a peel of $\mathfrak{S}_1'$ and $\mathfrak{S}_0\subseteq\mathfrak{S}_1'$ with corresponding change of variable map $\varphi_{\mathfrak{S}_2,\mathfrak{S}_1'}$.        
                Denote by $\Theta := \{\theta_u\mid u\in\mathcal{N}(\mathfrak{S}_0)\setminus\mathcal{N}(\mathfrak{S}_2)\}$. Observe that $\Theta\subseteq \Theta_{\mathfrak{S}_2,\mathfrak{S}_1}\cap \Theta_{\mathfrak{S}_2,\mathfrak{S}_1'}$.
                Suppose $l$ is an exponent in $S$. Then the corresponding exponents after the variable change and reduction are
                \[
                    l_1=\restr{\big(\varphi_{\mathfrak{S}_2,\mathfrak{S}_1}(\restr{l}{\mathcal{N}(\mathfrak{S}_2)\cup \Theta_{\mathfrak{S}_2,\mathfrak{S}_1}})\big)}{\mathcal{N}(\mathfrak{S}_0)},\quad l_1'=\restr{\big(\varphi_{\mathfrak{S}_2,\mathfrak{S}_1'}(\restr{l}{\mathcal{N}(\mathfrak{S}_2)\cup \Theta_{\mathfrak{S}_2,\mathfrak{S}_1'}})\big)}{\mathcal{N}(\mathfrak{S}_0)} .
                \]
                By Lemma \ref{lem:varRestr} we have that $\restr{l_1}{\mathcal{N}(\mathfrak{S}_2)} = \restr{l_1'}{\mathcal{N}(\mathfrak{S}_2)}$. 
                
                Fix an arbitrary $\theta_n\in\Theta$ and consider the constants $\alpha_n,\beta_n$ corresponding to $\varphi_{\mathfrak{S}_2,\mathfrak{S}_1}$ and $\alpha_n',\beta_n'$ corresponding to $\varphi_{\mathfrak{S}_2,\mathfrak{S}_1'}$, as given in Lemma \ref{Lem:legRestrs}. They are determined by the weights surrounding the node $n$ and $n\sim v_n\in\mathcal{N}(\mathfrak{S}_0)$ in the splice diagrams $\mathfrak{S}_1$ and $\mathfrak{S}_1'$ respectively. These weights are the same since $n,v_n\in\mathcal{N}(\mathfrak{S}_0)$ and $\mathfrak{S}_0\subseteq \mathfrak{S}_1,\mathfrak{S}_1'$. Therefore $\alpha_n = \alpha_n'$ and $\beta_n=\beta_n'$, which produces 
                \[
                    m_{n}(l_1) = \frac{m_{\theta_n}(l)-\beta_nm_{v_n}(l)}{\alpha_n}=\frac{m_{\theta_n}(l)-\beta_n'm_{v_n}(l)}{\alpha_n'} = m_n(l_1').
                \]
                The proof of the second part is analogous to the first part, as $\mathcal{N}(\mathfrak{S}_3)\subseteq\mathcal{N}(\overline{\mathfrak{S}_3})$ and $\Theta_{\overline{\mathfrak{S}_3},\overline{\mathfrak{S}_3}'}=\Theta_{\mathfrak{S}_3,\mathfrak{S}_0}$.
            \end{proof} 
            Now we can state and prove the surgery formula for the over-reduced polynomial part $\operatorname{Pol}_{\Gamma}(\mathbf{t}_{\mathfrak{S}_0})$.         
            \begin{Thm}\label{Thm:oRedPolSurgery}
                \begin{align}\label{eq:oRedPolSurgery}
                    \sum_{\mathfrak{S}_3\hookrightarrow\mathfrak{S}_0}(-1)^{|\mathfrak{S}_0|-|\mathfrak{S}_3|}&\big(\varphi_{\overline{\mathfrak{S}_3},\overline{\mathfrak{S}_3}'}^*\operatorname{Pol}_{\Gamma(\overline{\mathfrak{S}_3})}\big)(\mathbf{t}_{\mathcal{N}(\mathfrak{S}_0)})= \\&= (-1)^{|\mathcal{N}_{\mathrm{e}}(\mathfrak{S}_0)|}\prod_{n\in\mathcal{N}_{\mathrm{n}}(\mathfrak{S}_0)}\Delta_{\mathfrak{S}_0(n)}(\mathbf{t}^{U_n})\prod_{n\in\mathcal{N}_{\mathrm{e}}(\mathfrak{S}_0)}\frac{1-\Delta_{\mathfrak{S}_0(n)}(\mathbf{t}^{U_n})}{1-\mathbf{t}^{U_n}}.\nonumber
                \end{align}
            \end{Thm}
            \begin{proof}
                To obtain this surgery formula, we will apply the $\operatorname{PolPart}_{\mathcal{N}(\mathfrak{S}_0)}$ operator to the surgery formula in Theorem \ref{Thm:oRedZetaSurgery}. For this let $\mathfrak{S}_3\hookrightarrow\mathfrak{S}_0$. Since $\mathcal{N}(\mathfrak{S}_0)\subseteq\mathcal{N}(\overline{\mathfrak{S}_3}')$ we can write that
                \begin{align*}
                    \operatorname{PolPart}_{\mathcal{N}(\mathfrak{S}_0)}&\Big(\big(\varphi_{\overline{\mathfrak{S}_3},\overline{\mathfrak{S}_3}'}^*f_{\Gamma(\overline{\mathfrak{S}_3})}\big)|_{\mathcal{N}(\mathfrak{S}_0)}
                (\mathbf{t}_{\mathcal{N}(\mathfrak{S}_0)})\Big)=\\ &= \operatorname{PolPart}_{\mathcal{N}(\mathfrak{S}_0)}\Big(\big(\operatorname{PolPart}_{\mathcal{N}(\overline{\mathfrak{S}_3}')}(\varphi_{\overline{\mathfrak{S}_3},\overline{\mathfrak{S}_3}'}^*f_{\Gamma(\overline{\mathfrak{S}_3})}
                )\big)|_{\mathcal{N}(\mathfrak{S}_0)}(\mathbf{t}_{\mathcal{N}(\mathfrak{S}_0)}) \Big)\\
                &=\operatorname{PolPart}_{\mathcal{N}(\mathfrak{S}_0)}\Big(\big(\varphi_{\overline{\mathfrak{S}_3},\overline{\mathfrak{S}_3}'}^*\operatorname{Pol}_{\Gamma(\overline{\mathfrak{S}_3})}\big)(\mathbf{t}_{\mathcal{N}(\mathfrak{S}_0)})\Big),
                \end{align*}
                where in the second equality we used Corollary \ref{Lem:PolPartZHS}.
                Theorem \ref{Thm:oRedZetaSurgery} implies that 
                \[
                \operatorname{PolPart}_{\mathcal{N}(\mathfrak{S}_0)}\Big(\sum_{\mathfrak{S}_3\hookrightarrow\mathfrak{S}_0}(-1)^{|\mathfrak{S}_0|-|\mathfrak{S}_3|}\big(\varphi_{\overline{\mathfrak{S}_3},\overline{\mathfrak{S}_3}'}^*\operatorname{Pol}_{\Gamma(\overline{\mathfrak{S}_3})}\big)(\mathbf{t}_{\mathcal{N}(\mathfrak{S}_0)})\Big)
                \]
                is equal to the right hand side of the equation (\ref{eq:oRedPolSurgery}).  In the case $|\mathfrak{S}_0|=2$, the case $\emptyset\hookrightarrow\mathfrak{S}_0$ does not contribute to the sum, as
                the corresponding polynomial part is zero. Now we prove that in the above expression, the series inside $\operatorname{PolPart}_{\mathcal{N}(\mathfrak{S}_0)}$ is a polynomial, more precisely an element of $\Z[\mathbf{t}_{\mathcal{N}(\mathfrak{S}_0)}]$, which completes the proof.

                For each $\mathfrak{S}_1\subseteq\mathfrak{S}$ for which $\mathfrak{S}_0\subseteq\mathfrak{S}_1$, we have by Remark \ref{Rem:PolSurgeryReform} that
                \[
                \sum_{\mathfrak{S}_2\hookrightarrow\mathfrak{S}_1}(-1)^{|\mathfrak{S}_2|}\big(\varphi_{\mathfrak{S}_2,\mathfrak{S}_1}^*\operatorname{Pol}_{\Gamma(\mathfrak{S}_2)}\big)(\mathbf{t}_{\mathcal{N}(\mathfrak{S}_1)})\in\Z[\mathbf{t}_{\mathcal{N}(\mathfrak{S}_1)}].
                \]
                Further reducing to the nodes of $\mathfrak{S}_0$, we can write that
                \[
                    \sum_{\mathfrak{S}_2\hookrightarrow\mathfrak{S}_1}(-1)^{|\mathfrak{S}_2|}A_{\mathfrak{S}_2}\in\Z[\mathbf{t}_{\mathcal{N}(\mathfrak{S}_0)}],
                \]
                where by Lemma \ref{Lem:PolIndependence} we have that
                $A_{\mathfrak{S}_2} := \big(\varphi_{\mathfrak{S}_2,\mathfrak{S}_1}^*\operatorname{Pol}_{\Gamma(\mathfrak{S}_2)}\big)(\mathbf{t}_{\mathcal{N}(\mathfrak{S}_0)})$ is independent of the choice of $\mathfrak{S}_1$. By Lemma \ref{Lem:reverseIE} we have that
                \begin{align*}
                    \sum_{\mathfrak{S}_3\hookrightarrow\mathfrak{S}_0}(-1)^{|\mathfrak{S}_0|-|\mathfrak{S}_3|}\big(\varphi_{\overline{\mathfrak{S}_3},\overline{\mathfrak{S}_3}'}^*\operatorname{Pol}_{\Gamma(\overline{\mathfrak{S}_3})}\big)(\mathbf{t}_{\mathcal{N}(\mathfrak{S}_0)})=\sum_{\mathfrak{S}_0\subseteq \mathfrak{S}_1\subseteq\mathfrak{S}}(-1)^{|\mathfrak{S}_1|}\sum_{\mathfrak{S}_2\hookrightarrow\mathfrak{S}_1}(-1)^{|\mathfrak{S}_2|}A_{\mathfrak{S}_2}\in\Z[\mathbf{t}_{\mathcal{N}(\mathfrak{S}_0)}],
                \end{align*}
                as $\big(\varphi_{\overline{\mathfrak{S}_3},\overline{\mathfrak{S}_3}'}^*\operatorname{Pol}_{\Gamma(\overline{\mathfrak{S}_3})}\big)(\mathbf{t}_{\mathcal{N}(\mathfrak{S}_0)}) = A_{\overline{\mathfrak{S}_3}}$.
            \end{proof}
            Expanding the notation in (\ref{eq:oRedPolSurgery}) we have
            \[
                \big(\varphi_{\overline{\mathfrak{S}_3},\overline{\mathfrak{S}_3}'}^*\operatorname{Pol}_{\Gamma(\overline{\mathfrak{S}_3})}\big)(\mathbf{t}_{\mathcal{N}(\mathfrak{S}_0)}) = \Big(\varphi_{\overline{\mathfrak{S}_3},\overline{\mathfrak{S}_3}'}^*\big(\operatorname{Pol}_{\Gamma(\overline{\mathfrak{S}_3})}(\mathbf{t}_{\mathcal{N}(\overline{\mathfrak{S}_3})\cup\Theta_{\overline{\mathfrak{S}_3},\overline{\mathfrak{S}_3}'}})\big)\Big)\Big|_{t_v\mapsto 1,v\notin\mathcal{N}(\mathfrak{S}_0)}.
            \]
            In other words $\big(\varphi_{\overline{\mathfrak{S}_3},\overline{\mathfrak{S}_3}'}^*\operatorname{Pol}_{\Gamma(\overline{\mathfrak{S}_3})}\big)(\mathbf{t}_{\mathcal{N}(\mathfrak{S}_0)})$ is the polynomial part of $f_{\Gamma(\overline{\mathfrak{S}_3})}(\mathbf{t}_{\mathcal{V}(\Gamma(\overline{\mathfrak{S}_3}))})$ reduced to the variables given by the domain of $\varphi_{\overline{\mathfrak{S}_3},\overline{\mathfrak{S}_3}'}$, to which the corresponding variable change is applied to; and finally we reduce to the variables $\mathcal{N}(\mathfrak{S}_0)$. This, in general, is not equal to the polynomial part of the over-reduced rational function 
            \[
            (\varphi_{\overline{\mathfrak{S}_3},\overline{\mathfrak{S}_3}'}^*
            f_{\Gamma(\overline{\mathfrak{S}_3})})|_{\mathcal{N}(\mathfrak{S}_0)}(\mathbf{t}_{\mathcal{N}(\mathfrak{S}_0)}) = \Big(\varphi_{\overline{\mathfrak{S}_3},\overline{\mathfrak{S}_3}'}^*\big(f_{\Gamma(\overline{\mathfrak{S}_3})}(\mathbf{t}_{\mathcal{N}(\overline{\mathfrak{S}_3})\cup\Theta_{\overline{\mathfrak{S}_3},\overline{\mathfrak{S}_3}'}})\big)\Big)\Big|_{t_v\mapsto 1,v\notin\mathcal{N}(\mathfrak{S}_0)};
            \]
            as stated in the beginning of this section.
            
            \begin{Rem}
                If $\mathfrak{S} = \mathfrak{S}_0$, for any $\mathfrak{S}_3\hookrightarrow\mathfrak{S}_0 = \mathfrak{S}$ we have $\overline{\mathfrak{S}_3} = \mathfrak{S}_3$ and $\overline{\mathfrak{S}_3}' = \mathfrak{S}$. Therefore, in this case,  Theorems \ref{Thm:PolSurgery} and  \ref{Thm:zetaSurgery:ZHS} are identical with  Theorems \ref{Thm:oRedPolSurgery} and \ref{Thm:oRedZetaSurgery} respectively.
            \end{Rem}
            \subsubsection{\bf Splicing formula for two variables}\label{ss:two}
            Let $\mathfrak{S}$ be a splice diagram. We chose two connected nodes $n_1$, $n_2$ of $\mathfrak{S}$ and let $\mathfrak{S}_0$ be the complete sub-diagram given by $\{n_1,n_2\}$. Furthermore, let $\mathfrak{S}_1$ and $\mathfrak{S}_2$ be the splice sub-diagrams of $\mathfrak{S}$ obtained by cutting the edge $n_1\sim n_2$, i.e., they contain $n_1,n_2$ respectively. 
            That is, $\mathfrak{S}$  is constructed from $\mathfrak{S}_1$ and $\mathfrak{S}_2$
            by the `classical splicing' along the edge  $n_1\sim n_2$. 
            Note that $\mathfrak{S}_1$ and $\mathfrak{S}_2$ correspond to the full extensions (with respect to $\mathfrak{S}_0$) of the complete sub-diagrams with the only node $n_1$, $n_2$ respectively. In particular $\mathfrak{S}_k = \mathfrak{S}_0(n_k)$, $k=1,2$. Consider the Alexander polynomials $\Delta_1,\Delta_2$ of the graph links given by $\mathfrak{S}_1$, $\mathfrak{S_2}$ respectively, with an arrow placed on the leg obtained from $n_1\sim n_2$. A direct application of Theorem \ref{Thm:oRedPolSurgery} yields the following.
            \begin{Cor}
                We have 
                \begin{align*}
                    \operatorname{Pol}_{\Gamma(\mathfrak{S})}(t_{n_1},t_{n_2}) - \big(\varphi_{\mathfrak{S}_1,\mathfrak{S}_1'}^*\operatorname{Pol}_{\Gamma(\mathfrak{S}_1)}\big)(t_{n_1},t_{n_2})&-\big(\varphi_{\mathfrak{S}_2,\mathfrak{S}_2'}^*\operatorname{Pol}_{\Gamma(\mathfrak{S}_2)}\big)(t_{n_1},t_{n_2})=\\ &= \frac{1-\Delta_1(\mathbf{t}^{U_1})}{1-\mathbf{t}^{U_1}}\cdot\frac{1-\Delta_2(\mathbf{t}^{U_2})}{1-\mathbf{t}^{U_2}},
                \end{align*}
                where $\mathfrak{S}_1',\mathfrak{S}_2'\subseteq\mathfrak{S}$ are given by $\mathcal{N}(\mathfrak{S}_k') = \mathcal{N}(\mathfrak{S}_k)\cup\mathcal{N}(\mathfrak{S}_0)$, and
                \[
                    U_1 = d_{n_1,n_2}E_{n_1}+w(\mathfrak{S}[n_1,n_2))E_{n_2},\quad U_2 = w(\mathfrak{S}[n_2,n_1))E_{n_1}+d_{n_2,n_1}E_{n_2}.
                \]
            \end{Cor}  
        
        \begin{example}
    		Let us consider again the splice diagram in the left hand side of figure \ref{fig:OverSplice3}, with $\mathfrak{S}_0$ given by the nodes $\{E_2,E_3\}$. Thus the diagrams $\mathfrak{S}_1$, $\mathfrak{S}_2$ are obtained by cutting the edge $2\sim3$. Hence,  $\mathfrak{S}_1$ contains  the nodes $\{E_1,E_2\}$, and $\mathfrak{S}_2$ the rest of the nodes, as in figure \ref{fig:OverSplice3:2}.
    		The normalized Alexander polynomials are
    		\begin{align*}
    			\Delta_1(t)=\Delta_{\mathfrak{S}_0(2)}(t) &= (1-t)\frac{1-t^2}{1-t}\frac{1-t^{12}}{(1-t^4)(1-t^6)}\\
    			\Delta_2(t)=\Delta_{\mathfrak{S}_0(3)}(t) &= (1-t)(1-t^{20})\frac{1-t^{24}}{(1-t^8)(1-t^{12})}\frac{1-t^{10}}{1-t^5}\frac{1-t^{60}}{(1-t^{30})(1-t^{20})}.
    		\end{align*}
    		Thus
    		\begin{align*}
    			\operatorname{Pol}_{\Gamma(\mathfrak{S})}(t_{2},t_{3}) - \big(\varphi_{\mathfrak{S}_1,\mathfrak{S}_1'}^*\operatorname{Pol}_{\Gamma(\mathfrak{S}_1)}\big)(t_{2},t_{3})&-\big(\varphi_{\mathfrak{S}_2,\mathfrak{S}_2'}^*\operatorname{Pol}_{\Gamma(\mathfrak{S}_2)}\big)(t_{2},t_{3})=\\ &= \frac{1-\Delta_1(\mathbf{t}^{U_1})}{1-\mathbf{t}^{U_1}}\cdot\frac{1-\Delta_2(\mathbf{t}^{U_2})}{1-\mathbf{t}^{U_2}};
    		\end{align*}
    		where $U_1 = 15E_2+20E_3$, $U_2=2E_2+3E_3$. Note that
            $1-\Delta_1(t) =(1-t)(t^2+t^3)$, and $1-\Delta_2(t) = (1-t)(t + t^2 + t^3 + t^4 + t^6 + t^7 + t^9 + t^{10} + t^{11} + t^{14} + t^{15} + t^{18} + t^{19} + t^{22} + t^{23} + t^{26} + t^{27} + t^{31} + t^{34} + t^{39})$.
    		\begin{figure}
    			\centering
    			\renewcommand{\graphscale}{1.25cm}
    			\begin{tikzpicture}[roundnode/.style={circle, draw=black, fill=black, very thick, inner sep = 1},line width=1,x=\graphscale,y=\graphscale]
    				\node[roundnode] (3) {};
    				\path (3)++(1,0) node[roundnode] (4) {};
    				\path (3)++(0,-1) node[roundnode] (5) {};
    				\path (3)++(-3,0) node[roundnode] (2) {};
    				\path (4)++(1,0) node[roundnode] (6) {};
    				\path (4)++(0,1) node[roundnode] (12) {};
    				\path (5)++(-60:1) node[roundnode] (11) {};
    				\path (5)++(-120:1) node[roundnode] (10) {};
    				\path (2)++(-1,0) node[roundnode] (1) {};
    				\path (2)++(0,1) node[roundnode] (9) {};
    				\path (6)++(60:1) node[roundnode] (13) {};
    				\path (6)++(-60:1) node[roundnode] (14) {};
    				\path (1)++(120:1) node[roundnode] (7) {};
    				\path (1)++(-120:1) node[roundnode] (8) {};
    				\draw (7)--(1)--(8) (1)--(2)--(9) (10)--(5)--(11) (5)--(3)--(6) (4)--(12) (14)--(6)--(13);
    				\draw (2)-++(1,0) node[roundnode] {};
    				\draw (3)-++(-1,0) node[roundnode] {};
    
    				\begin{scriptsize}
    					\renewcommand{\r}{0.25}
    					\foreach \vertex/\angle/\weight in {%
    						1/-35/181,1/-150/3,1/150/2,
    						2/-145/1,2/120/2,2/35/15,
    						3/150/3,3/35/4,3/-120/5,
    						4/145/9,4/-35/1,4/65/2,
    						5/140/17,5/-150/2,5/-30/3,
    						6/-145/109,6/30/2,6/-30/3
    					}{
    						\path (\vertex) ++ (\angle:\r) node {$\weight$};
    					}
    					\renewcommand{\r}{0.3}
    					\foreach \vertex/\angle in {%
    						1/60,2/-80,3/-50,5/50,4/-120,6/120
    					}{
    						\path (\vertex) ++ (\angle:\r) node {$E_{\vertex}$};
    					}
    				\end{scriptsize}
    			\end{tikzpicture}
    			\caption{}
    			\label{fig:OverSplice3:2}
    		\end{figure} 
	    \end{example}
        
        \phantom{\cite{*}}

      \section{Application to the periodic constant}
        \subsection{\bf A short detour on the periodic constant}\label{ss:6.1}
    \cite{LNN19,LN14Erhart}  First
     we give a short history of the {\it periodic constant} of series.

  \subsubsection{}\label{sss:6.1.1}  {\bf The one variable case.} 
    The periodic constant of a formal power series in one variable was introduced in \cite{NO09,O08} as follows. Assume that for $S = \sum c_lt^l\in \Z[[t]]$ there exists a positive integer $p$ such that $Q_p(n)=\sum_{l< np}c_l$ is a polynomial in $n$. Then the constant term $\operatorname{pc}(S)=Q_p(0)$ is independent of $p$ and is called the periodic constant of $S$. For example, if $S(t)\in\Z[t]$ then for  $p\gg 1$ 
    the sum 
     $Q_p(n)$ is the constant $S(1)$, hence $\operatorname{pc}(S)=S(1)$. 

    The intuitive meaning of the periodic constant is well shown  by the following example. Let $S$ be the Hilbert series of a graded algebra/vector space $A = \bigoplus_{l\ge0} A_l$, that is, $c_l = \dim A_l$. Assume that $S$ admits a Hilbert quasi-polynomial $Q(l)$, meaning that $c_l = Q(l)$ for large enough $l$. Then (by a computation) 
    the periodic constant of the regularized series $S_{\operatorname{reg}} = \sum Q(l)t^l$ is zero, thus $\operatorname{pc}S$ measures the difference  between $S$ and $S_{\operatorname{reg}}$ (which is a polynomial), that is $\operatorname{pc}(S) = (S(t) - S_{\operatorname{reg}}(t))\mid_{t=1}$.
    
    In   \cite{BN10} one can find the next interpretation of the periodic constant too. 
    Assume  that the series $S(t)$ can be expressed as a rational function $B(t)/A(t)$ in $t$, where $A(t) = \prod_i(1-t^{a_i})$. Rewriting the series as $C(t) + D(t)/A(t)$, where the degree of $D(t)$ is smaller than the degree of  $A(t)$ we get  $S_{\operatorname{reg}} = D(t)/A(t)$ and   $S-S_{\operatorname{reg}}=C(t)$, 
    hence $\operatorname{pc}S=C(1)$.

 \subsubsection{}   {\bf The multivariable case.} 
    The periodic constant was generalized to the multivariable case in \cite{LN14Erhart}. Here we focus on the Taylor expansion  $Z(\mathbf{t})=\sum_l z(l)\mathbf{t}^l$ of (\ref{eq:zeta:def}) (hence we use the terminology from Section \ref{ss:2.1}). 
    Assume that there is a real cone $\mathcal{K}\subset L\otimes\R$, with affine closure of top dimension,  and also   a sublattice $\widetilde{L}\subseteq L$ with finite index and $l_*\in\mathcal{K}$ such that $Q(l)=\sum_{\bar{l}\not\geq l}z(\bar{l})$ is equal to a quasi-polynomial $\mathcal{P}^{\mathcal{K}}(l)$ for all $l\in \widetilde{L}\cap(l_*+\mathcal{K})$. Then the periodic constant of $Z(\mathbf{t})$ associated with the cone $\mathcal{K}$ is 
    \[
        \operatorname{pc}^{\mathcal{K}}(Z) = \mathcal{P}^{\mathcal{K}}(0).
    \]
    This definition does not depend on $\widetilde{L}$ (its choice is the analogue for the choice of $p$ in the single variable case).

    In the case of $Z(\mathbf{t})=Z_\Gamma(\mathbf{t})$ there is a distinguished cone, the {\it real Lipman cone}, defined as  $\mathcal{K}=\mathcal{S}_{\R} =\mathcal{S}\otimes\R$, where 
    $\mathcal{S}:=\{l\in L\,:\, (l,E_v)\leq 0 \ \mbox{for all $v$}\}$ as in Section \ref{ss:2.3}. 
    In \cite{LN14Erhart} it was  proved  that for $\mathcal{K}=\mathcal{S}_{\R}$ the periodic constant (denoted also by $\operatorname{pc}(Z(\mathbf{t}))$)  is well-defined  and 
        \[
            \operatorname{pc}(Z_\Gamma(\mathbf{t})) = \operatorname{Pol}_{\Gamma}({\mathbf{t}})|_{t_v=1 \ \mbox{for all $v$ }}.
        \]   
    Next,
    we can consider the series $Z_\Gamma(\mathbf{t}_{\mathcal{N}})$, $Z_\Gamma(\mathbf {t})$
    reduced to the variables of the nodes $\{t_n\}_{n\in\mathcal{N}(\Gamma)}$. Then, by \cite{LN14Erhart}, 
    the periodic constant of  $Z_\Gamma(\mathbf {t}_{\mathcal{N}})$
     associated with the cone $\mathcal{S}_{\R}\cap \R\langle E_n, \ n\in\mathcal{N}(\Gamma)\rangle$
     is well-defined and it also equals  $\operatorname{pc}(Z_\Gamma(\mathbf{t})) = 
     \operatorname{Pol}_{\Gamma}(1)$. 

    All these facts can be applied for the graphs $\Gamma_I$ and the 
    corresponding series $Z_{\Gamma_I}$ (considered in section
    \ref{s:surgtopseries}) as well. In particular, 
    $\operatorname{pc}(Z_{\Gamma_I}) = \operatorname{Pol}_{\Gamma_I}(1)$.

    Finally notice  that, since $\varphi_I^*$ is an invertible linear map, one also has 
    \[
        (\varphi^*_I\operatorname{Pol}_{\Gamma_I}) (\mathbf{t}_{\mathcal{N}(\Gamma)})\big|_{t_v=1, \ v\in \mathcal{N}(\Gamma)}= 
    \operatorname{Pol}_{\Gamma_I} (\mathbf{t}_{\mathcal{N}(\Gamma_I)})\big|_{t_v=1, \ v\in \mathcal{N}(\Gamma_I)}=  \operatorname{pc}(Z_{\Gamma_I}).
    \]
    Therefore,  if we substitute $\mathbf{t}_{\mathcal{N}(\Gamma)}\mapsto 1$ in (\ref{eq:surgerypol}), we get the next {\it surgery formula for the periodic constant of the topological Poincar\'e series} involving all the graphs $\{\Gamma_I\}_{I\subset \mathcal{N}_{\mathrm{e}}}$.
    \begin{Thm}\label{thm:6a}
     
            \begin{equation}\label{eq:surgerypc}
                \operatorname{pc}(Z_{\Gamma}) +\sum_{I\subsetneq \mathcal{N}_{\mathrm{e}}}(-1)^{|\mathcal{N}_{\mathrm{e}}\setminus I|}\operatorname{pc}(Z_{\Gamma_I})  = (-1)^{|\mathcal{N}_{\mathrm{e}}|} \prod_{n\in\mathcal N}P_n(1).
            \end{equation}
    \end{Thm}
    Similarly, in the context of Section \ref{s:OverRedSurg}, Theorem \ref{Thm:oRedPolSurgery} implies
    \begin{Thm}\label{thm:6b}
        \begin{equation}\label{eq:surgerypc_b}
                    \sum_{\mathfrak{S}_3\hookrightarrow\mathfrak{S}_0}(-1)^{|\mathfrak{S}_0|-|\mathfrak{S}_3|}\operatorname{pc}(Z_{\Gamma(\overline{\mathfrak{S}_3})})= (-1)^{|\mathcal{N}_{\mathrm{e}}(\mathfrak{S}_0)|}\prod_{n\in\mathcal{N}_{\mathrm{n}}(\mathfrak{S}_0)}\Delta_{\mathfrak{S}_0(n)}(1)\cdot \prod_{n\in\mathcal{N}_{\mathrm{e}}(\mathfrak{S}_0)}\Delta'_{\mathfrak{S}_0(n)}(1).
        \end{equation}
    \end{Thm}
Using Lemma \ref{Lem:AlexanderMultipleFromSingle}, the right hand side of (\ref{eq:surgerypc_b})
can be expressed (after a longer computation) in terms of the
decorations of the splice diagram. However, we decided to keep (\ref{eq:surgerypc_b})
in the above short and compact form (and not to add any additional technical reinterpretation).

\subsubsection{\bf Connection with  the Casson's invariants and the $(Z_K^2+|\mathcal{V}|)$ invariant 
of the link}
     
     Let $\lambda(M)$ be the Casson's invariant of the 3-manifold $M$, see eg. \cite{A14casson}. 
       
    In our case, when $M=M(\Gamma)$ is a negative definite  plumbed 3--manifold, it is given in terms of $\Gamma$ by 
    (see eg. \cite{nemethi2022normal}) 
            \[
                -24\lambda(M(\Gamma)) = \sum_{v\in\mathcal{V}}(3+E_v^2) - (\delta_v-2)(E_v^*,E_v^*).
            \]
     In   \cite{A14casson,Ls96global} (see also \cite{nemethi2022normal}), whenever $M$ is an $\Z HS^3$, it is proved that 
            \begin{equation}\label{eq:lambda}
             {\rm pc} (Z_\Gamma)=  
             -\lambda(M(\Gamma))-\frac{Z_K^2+|\mathcal{V}(\Gamma)|}{8}.
            \end{equation}
    If we substitute this identity into (\ref{eq:surgerypc}) and (\ref{eq:surgerypc_b}) we obtain  surgery formulae for 
    $\lambda(M(\Gamma))+(Z_K^2+|\mathcal{V}(\Gamma)|)/8$. However, the Casson's invariant 
    satisfies the classical splice formula: if the integral homology sphere $M$ is obtained by splicing the 
    3-manifolds $M_1$ and $M_2$ then $\lambda(M)=\lambda(M_1)+\lambda(M_2)$, see \cite{BN,FM}. 
    A computation shows that if we substitute in the left hand side 
    of  (\ref{eq:surgerypc}) and (\ref{eq:surgerypc_b})
    the identity (\ref{eq:lambda}) then all the contributions from Casson's invariant will cancel because of 
    the previous additivity property.
     For this denote $M_n$ the manifold corresponding to the complete subdiagram of $\mathfrak{S}$ with the only node $n$. Then $\lambda(M(\mathfrak{S})) = \sum_{n\in\mathcal{N(\mathfrak{S})}}\lambda(M_n)$. Due to this identity, we have that
    \[
        \sum_{\mathfrak{S'}\hookrightarrow\mathfrak{S}}(-1)^{|\mathfrak{S}'|}\lambda(M(\mathfrak{S}')) = 0,
    \]
    from which Lemma \ref{Lem:reverseIE} yields that 
    \[
        \sum_{\mathfrak{S}_3\hookrightarrow\mathfrak{S}_0}(-1)^{|\mathfrak{S}_0|-|\mathfrak{S}_3|}\lambda(M(\overline{\mathfrak{S}_3})) = 0.
    \]
    In particular,  (\ref{eq:surgerypc}) and (\ref{eq:surgerypc_b})
    transforms into a surgery formula of $Z_K^2+|\mathcal{V}|$. 
    The second version (\ref{eq:surgerypc_b}) applied in the situation of Section \ref{ss:two} recovers the 
    the known formula from \cite{NW} (see also \cite[Theorem 6.3.15]{nemethi2022normal}).

    
    \bibliographystyle{plain}

\begin{thebibliography}{10}
    
        \bibitem{A14casson}
        S.~Akbulut and J.~D. McCarthy.
        \newblock {\em Casson's Invariant for Oriented Homology Three-Spheres: An
          Exposition}, volume~36.
        \newblock Princeton University Press, 2014.
        
        \bibitem{BN}
        S.~Boyer and A.~Nicas.
        \newblock Varieties of group representations and {C}asson's invariant for
          rational homology 3-spheres.
        \newblock {\em Trans. Amer. Math. Soc.}, 322(2):507--522, 1990.
        
        \bibitem{BN07}
        G.~Braun and A.~N{\'e}methi.
        \newblock Invariants of {N}ewton non-degenerate surface singularities.
        \newblock {\em Compositio Mathematica}, 143(4):1003--1036, 2007.
        
        \bibitem{BN10}
        G.~Braun and A.~N{\'e}methi.
        \newblock Surgery formula for {S}eiberg--{W}itten invariants of negative
          definite plumbed 3-manifolds.
        \newblock {\em J. reine angew. Math.(CRELLE)}, 638:189--208, 2010.
        
        \bibitem{CDGZ04}
        A.~Campillo, F.~Delgado, and S.~M. Gusein-Zade.
        \newblock {P}oincar{\'e} series of a rational surface singularity.
        \newblock {\em Inventiones mathematicae}, 155(1), 2004.
        
        \bibitem{eisenbud2016three}
        D.~Eisenbud and W.~D. Neumann.
        \newblock {\em Three-Dimensional Link Theory and Invariants of Plane Curve
          Singularities.(AM-110)}.
        \newblock Princeton University Press, 2016.
        
        \bibitem{FM}
        S.~Fukahara and N.~Maruyama.
        \newblock A sum formula for {C}asson's $\lambda$-invariant.
        \newblock {\em Tokyo J. Math.}, 11(2):281--287, 1988.
        
        \bibitem{CDGZ08universal}
        S.~M. Gusein-Zade, F.~Delgado, and A.~Campillo.
        \newblock Universal abelian covers of rational surface singularities and
          multi-index filtrations.
        \newblock {\em Functional Analysis and Its Applications}, 42(2):83--88, 2008.
        
        \bibitem{laszlo2022canonical}
        T.~L{\'a}szl{\'o}.
        \newblock On a canonical polynomial for links of elliptic singularities.
        \newblock {\em arXiv preprint arXiv:2201.10837}, 2022.
        
        \bibitem{LNN19}
        T.~L{\'a}szl{\'o} and A.~Nagy, J.and~N{\'e}methi.
        \newblock Combinatorial duality for {P}oincar{\'e} series, polytopes and
          invariants of plumbed 3-manifolds.
        \newblock {\em Selecta Mathematica}, 25(2):21, 2019.
        
        \bibitem{LNN20}
        T.~L{\'a}szl{\'o}, J.~Nagy, and A.~N{\'e}methi.
        \newblock Surgery formulae for the {S}eiberg--{W}itten invariant of plumbed
          3-manifolds.
        \newblock {\em Revista Matem{\'a}tica Complutense}, 33(1):197--230, 2020.
        
        \bibitem{LN14Erhart}
        T.~L{\'a}szl{\'o} and A.~N{\'e}methi.
        \newblock {E}hrhart theory of polytopes and {S}eiberg--{W}itten invariants of
          plumbed 3--manifolds.
        \newblock {\em Geometry \& Topology}, 18(2):717--778, 2014.
        
        \bibitem{LSz17}
        T.~L{\'a}szl{\'o} and Zs. Szil{\'a}gyi.
        \newblock Non-normal affine monoids, modules and {P}oincar{\'e} series of
          plumbed 3-manifolds.
        \newblock {\em Acta Mathematica Hungarica}, 152(2):421--452, 2017.
        
        \bibitem{LSz18}
        T.~L{\'a}szl{\'o} and Zs. Szil{\'a}gyi.
        \newblock {N}\'emethi's division algortihm for zeta-functions of plumbed
          3-manofolds.
        \newblock {\em Bull. London Math. Soc.}, 50(6):1035--1055, 2018.
        
        \bibitem{Ls96global}
        C.~Lescop.
        \newblock {\em Global surgery formula for the {C}asson--{W}alker invariant}.
        \newblock Princeton University Press, 1996.
        
        \bibitem{N04invariants}
        A.~N{\'e}methi.
        \newblock Invariants of normal surface singularities.
        \newblock {\em Real and complex singularities}, pages 161--208, 2004.
        
        \bibitem{N07poincare}
        A.~N{\'e}methi.
        \newblock {P}oincar\'e series associated with surface singularities,
          {S}ingularities {I}: {A}lgebraic and {A}nalytic {A}spects.
        \newblock {\em AMS Contemporary Mathematics}, 474:271--299, 2007.
        
        \bibitem{nemethi2022normal}
        A.~N{\'e}methi.
        \newblock {\em Normal surface singularities}, volume~74 of {\em Ergebnisse der
          Math. und ihrer Grenzgebiete}.
        \newblock Springer, 2022.
        
        \bibitem{N02seiberg}
        A.~N{\'e}methi and L.~I. Nicolaescu.
        \newblock {S}eiberg--{W}itten invariants and surface singularities.
        \newblock {\em Geometry \& Topology}, 6(1):269--328, 2002.
        
        \bibitem{N04seiberg}
        A.~N{\'e}methi and L.~I. Nicolaescu.
        \newblock {S}eiberg--{W}itten invariants and surface singularities. {II}:
          Singularities with good-action.
        \newblock {\em Journal of the London Mathematical Society}, 69(3):593--607,
          2004.
        
        \bibitem{N06seiberg}
        A.~N{\'e}methi and L.~I. Nicolaescu.
        \newblock {S}eiberg--{W}itten invariants and surface singularities: splicings
          and cyclic covers.
        \newblock {\em Selecta Mathematica}, 11(3):399, 2006.
        
        \bibitem{N08seiberg}
        A.~N{\'e}methi and T.~Okuma.
        \newblock The {S}eiberg--{W}itten invariant conjecture for splice-quotients.
        \newblock {\em Journal of the London Mathematical Society}, 78(1):143--154,
          2008.
        
        \bibitem{NO09}
        A.~N{\'e}methi and T.~Okuma.
        \newblock On the {C}asson {I}nvariant {C}onjecture of {N}eumann--{W}ahl.
        \newblock {\em J. Algebraic Geometry}, 18:135--149, 2009.
        
        \bibitem{NeuCalc}
        W.~D. Neumann.
        \newblock A calculus for plumbing applied to the topology of complex surface
          singularities and degenerating complex curves.
        \newblock {\em Trans. Amer. Math. Soc.}, 268:299--344, 1981.
        
        \bibitem{NW}
        W.~D. Neumann and J.~Wahl.
        \newblock Complex surface singularities with integral homology sphere links.
        \newblock {\em Geom. Topol.}, 9(2):757--811, 2005.
        
        \bibitem{O08}
        T.~Okuma.
        \newblock The geometric genus of splice-quotient singularities.
        \newblock {\em Transactions of the American Mathematical Society},
          360(12):6643--6659, 2008.
        
    \end{thebibliography}

\end{document}